\documentclass[12pt,a4paper]{article}

\usepackage{latexsym}
\usepackage{amsmath}
\usepackage{amssymb}
\usepackage{amsthm}
\usepackage{amscd}
\usepackage{mathrsfs}
\usepackage[all]{xy}
\usepackage{graphicx}

\def\cprime{$'$}

\newtheorem{thm}{Theorem}[section]
\newtheorem{prop}[thm]{Proposition}
\newtheorem{defn}[thm]{Definition}
\newtheorem{lem}[thm]{Lemma}
\newtheorem{cor}[thm]{Corollary}

\newtheorem{rem}[thm]{Remark}

\theoremstyle{remark}

\makeatletter

\@addtoreset{equation}{section}
\makeatother

\makeatletter
\newcommand{\subsubsubsection}{\@startsection{paragraph}{4}{\z@}%
 {1.0\Cvs \@plus.5\Cdp \@minus.2\Cdp}%
 {.1\Cvs \@plus.3\Cdp}%
 {\reset@font\sffamily\normalsize}
 }
\makeatother
\DeclareMathOperator{\Gal}{Gal}

\DeclareMathOperator{\End}{End}
\DeclareMathOperator{\Hom}{Hom}

\DeclareMathOperator{\Spf}{Spf}

\DeclareMathOperator{\Aut}{Aut} 
 
\DeclareMathOperator{\Ind}{Ind} 
\DeclareMathOperator{\Tr}{Tr}
\DeclareMathOperator{\Nr}{Nr}

\DeclareMathOperator{\Irr}{Irr}
\DeclareMathOperator{\Disc}{Disc}

\begin{document}

\title
%[]
{Stable models of Lubin-Tate curves with level three} 

\author{Naoki Imai and Takahiro Tsushima}
\date{}
\maketitle


\begin{abstract}
We construct a stable formal model of 
a Lubin-Tate curve with level three, 
and study 
the action of a Weil group 
and a division algebra on its stable reduction. 
Further, we study a structure of cohomology of 
the Lubin-Tate curve. 
Our study is purely local and includes the case where 
the characteristic of the residue field of a local field is two. 
\end{abstract}

\section*{Introduction}
Let $K$ be a non-archimedean local field 
with a finite residue field $k$ of characteristic $p$. 
Let $\mathfrak{p}$ be the maximal ideal of 
the ring of integers $\mathcal{O}_K$ of $K$. 
Let $n$ be a natural number. 
We write $\mathrm{LT}(\mathfrak{p}^n)$ 
for the Lubin-Tate curve with 
full level $n$ 
as a deformation space of 
formal $\mathcal{O}_K$-modules by quasi-isogenies. 
Let $D$ be the central division algebra  over $K$ of 
invariant $1/2$. 
Let $\ell$ be a prime number different from $p$. 
We write $\mathbf{C}$ for the 
completion of an algebraic closure of $K$. 
Then the groups $W_K$, $\mathrm{GL}_2 (K)$ and $D^{\times}$ 
act on 
\[
 \varinjlim_m H^1 _{\mathrm{c}} 
 (\mathrm{LT}(\mathfrak{p}^m)_{\mathbf{C}}, 
 \overline{\mathbb{Q}}_{\ell} ), 
\]
and these actions partially 
realize the local Langlands correspondence 
and the local Jacquet-Langlands correspondence for $\mathrm{GL}_2$. 
The realization of 
the local Langlands correspondence was proved by 
global automorphic methods in \cite{CaNALT}. 
Since Lubin-Tate curves are purely local objects, 
it is desirable to have a purely local proof 
which only makes use of the geometry of Lubin-Tate curves. 

We put 
\[
 K_1 (\mathfrak{p}^n) =\biggl\{ 
 \begin{pmatrix}
 a & b \\
 c & d
 \end{pmatrix} 
 \in \mathrm{GL}_2 (\mathcal{O}_K ) \ \bigg| \ 
 c \equiv 0 ,\ d \equiv 1 \bmod \mathfrak{p}^n \biggr\}. 
\]
Let $\mathrm{LT}_1(\mathfrak{p}^n)$ be 
the Lubin-Tate curve with 
level $K_1 (\mathfrak{p}^n)$ 
as a deformation space of 
formal $\mathcal{O}_K$-modules by quasi-isogenies. 
Then the cohomology group 
\[
 H^1 _{\mathrm{c}} 
 (\mathrm{LT}_1(\mathfrak{p}^n)_{\mathbf{C}}, 
 \overline{\mathbb{Q}}_{\ell} ) = 
 \Bigl( \varinjlim_m H^1 _{\mathrm{c}} 
 (\mathrm{LT}(\mathfrak{p}^m)_{\mathbf{C}}, \overline{\mathbb{Q}}_{\ell} )
 \Bigr) ^{K_1 (\mathfrak{p}^n)} 
\]
will give 
representations of $W_K$ and $D^{\times}$ 
that correspond to smooth irreducible representations 
of $\mathrm{GL}_2(K)$ with conductor less than or equal to $n$. 
The purpose of this paper is to study this cohomology 
in the case $n=3$. 
We note that 
$3$ is the smallest conductor of 
a two-dimensional representation of $W_K$ 
which can not be written as an induction of a character. 
Such a representation is called a primitive representation. 

Our method is purely local and geometric. 
In fact, we construct a stable model 
of the connected Lubin-Tate curve 
$\mathbf{X}_1(\mathfrak{p}^3)$ with level 
$K_1 (\mathfrak{p}^3)$ 
by using the theory of semi-stable coverings 
(cf. \cite[Section 2.3]{CMp^3}). 
Our study includes the case where 
$p=2$, and 
in this case, 
primitive Galois representations of conductor $3$ appear 
in the cohomology of $\mathbf{X}_1(\mathfrak{p}^3)$. 
It gives a geometric understanding of 
a realization of the primitive Galois representations. 

Our method of the calculation of the stable reduction 
is similar to that in \cite{CMp^3}. 
In \cite{CMp^3}, Coleman-McMurdy calculate the stable reduction of 
the modular curve $X_0(p^3)$ under the assumption 
$p \geq 13$. 
The calculation of the stable reductions in the modular curve 
setting is 
equivalent to 
that in the Lubin-Tate setting where $K=\mathbb{Q}_p$. 
As for the calculation of the stable reduction of 
the modular curve $X_1(p^n)$, 
it is given in \cite{DRmc} if $n=1$. 

We explain the contents of this paper. 
In Section \ref{prel}, 
we recall a definition of the connected 
Lubin-Tate curve, and study 
the action of a division algebra in a general setting. 
In Section \ref{cohLT}, 
we study the cohomology of Lubin-Tate curves 
as representation of $\mathrm{GL}_2(K)$ by 
purely local methods. 
By this result, we can calculate the genus 
of some Lubin-Tate curves. 
In Section \ref{level2}, 
we construct a stable covering of the 
connected Lubin-Tate curve with level 
$K_1(\mathfrak{p}^2)$, 
which is used to study a covering of 
$\mathbf{X}_1(\mathfrak{p}^3)$. 

In Section \ref{redaff}, 
we define several affinoid subspaces 
$\mathbf{Y}_{1,2}$, 
$\mathbf{Y}_{2,1}$ and 
$\mathbf{Z}^0_{1,1}$ 
of $\mathbf{X}_1(\mathfrak{p}^3)$, 
and calculate their reductions. 
Let $k^{\mathrm{ac}}$ be the residue field of $\mathbf{C}$. 
We put $q=|k|$ and 
\[
 \mathcal{S}_1=
 \begin{cases}
 \mu_{2(q^2-1)}(k^{\mathrm{ac}}) &\quad  
 \textrm{if $q$ is odd,}\\
 \mu_{q^2-1}(k^{\mathrm{ac}}) 
 &\quad \textrm{if $q$ is even.}
 \end{cases}
\]
The reductions of 
$\mathbf{Y}_{1,2}$ and 
$\mathbf{Y}_{2,1}$ 
are isomorphic to the affine curve 
defined by $x^qy-xy^q=1$. 
This affine curve has genus $q(q-1)/2$, 
and is called the 
Deligne-Lusztig curve for 
$\mathrm{SL}_2 (\mathbb{F}_q)$ 
or the Drinfeld curve. 
Here, the genus of a curve means 
the genus of the smooth compactification of the 
normalization of the curve. 
The reduction $\overline{\mathbf{Z}}^0_{1,1}$ of 
$\mathbf{Z}^0_{1,1}$ 
is isomorphic to the affine curve defined by 
$Z^q+X^{q^2-1}+X^{-(q^2-1)}=0$. 
This affine curve has genus $0$ 
and singularities at $X \in \mathcal{S}_1$. 

Next, we analyze tubular neighbourhoods 
$\{\mathcal{D}_{\zeta}\}_{\zeta \in \mathcal{S}_1}$ 
of the singular points of $\overline{\mathbf{Z}}^0_{1,1}$. 
If $q$ is odd, 
$\mathcal{D}_{\zeta}$ 
is a basic wide open space with 
the underlying affinoid 
$\mathbf{X}_{\zeta}$. 
See \cite[2B]{CMp^3} for the precise definition of 
a basic wide open space. 
Roughly speaking, 
it is a smooth geometrically connected 
one-dimensional rigid space which contains an affinoid such that 
the reduction of the affinoid is irreducible and 
has at worst ordinary double points as singularities, 
and the complement of the affinoid is 
a disjoint union of open annuli. 
The reduction of $\mathbf{X}_{\zeta}$ 
is isomorphic to the Artin-Schreier affine curve of degree $2$ 
defined by $z^q-z=w^2$. 
This affine curve has genus $(q-1)/2$. 

On the other hand, if $q$ is even, 
it is harder 
to analyze 
$\mathcal{D}_{\zeta}$, 
because the space 
$\mathcal{D}_{\zeta}$ is 
not a basic wide open space. 
First, we find an affinoid 
$\mathbf{P}^0_{\zeta}$. 
The reduction 
$\overline{\mathbf{P}}^0_{\zeta}$ 
of $\mathbf{P}^0_{\zeta}$ 
has genus $0$ and 
singular points 
parametrized by $\zeta' \in k^{\times}$. 
Secondly, we analyze the tubular 
neighborhoods 
of singular points of $\overline{\mathbf{P}}^0_{\zeta}$. 
As a result, we find an affinoid 
$\mathbf{X}_{\zeta,\zeta'}$, 
whose reduction 
$\overline{\mathbf{X}}_{\zeta,\zeta'}$ is 
isomorphic to the affine curve defined by $z^2+z=w^3$. 
The smooth compactification of this curve is 
the unique supersingular elliptic curve over 
$k^{\mathrm{ac}}$, whose 
$j$-invariant is $0$, and 
its cohomology gives a primitive 
Galois representation. 
By using these affinoid spaces, 
we construct a covering $\mathcal{C}_1(\mathfrak{p}^3)$ of 
$\mathbf{X}_1(\mathfrak{p}^3)$. 

In Section \ref{actdiv}, 
we calculate the action of 
$\mathcal{O}_D^{\times}$ 
on the reductions of the affinoid spaces in 
$\mathbf{X}_1(\mathfrak{p}^3)$, 
where $\mathcal{O}_D$ is the ring of integers of $D$. 
In Section \ref{actine}, 
we calculate an action of a Weil group 
on the reductions. 
In the case where $q$ is even, 
we construct an 
$\mathrm{SL}_2 (\mathbb{F}_3)$-Galois extension 
of $K^{\mathrm{ur}}$, and 
show that the Weil action on 
$\overline{\mathbf{X}}_{\zeta,\zeta'}$ 
up to translations 
factors through the Weil group of the constructed extension. 
For such a Galois extension, see also \cite[31]{WeilExe}. 

In Section \ref{cohLT3}, 
we show that the covering 
$\mathcal{C}_1(\mathfrak{p}^3)$ 
is semi-stable. 
To show this, we calculate the summation of 
the genera of the reductions of the affinoid spaces in 
$\mathbf{X}_1(\mathfrak{p}^3)$, 
and compare it with the genus of 
$\mathbf{X}_1(\mathfrak{p}^3)$. 
Using the constructed semi-stable model, 
we study a structure of cohomology of 
$\mathbf{X}_1(\mathfrak{p}^3)$. 

The dual graph of the semi-stable reduction of 
$\mathbf{X}_1(\mathfrak{p}^3)$ in the case where $q$ is even is 
the following: \\ 
\centerline{
\xygraph{
\circ ([]!{+(+.1,+.3)} {\overline{\mathbf{Y}}^{\mathrm{c}}_{1,2}})
(-[r]
\circ ([]!{+(+.1,+.3)} {\overline{\mathbf{Z}}^{0,\mathrm{c}}_{1,1}})
(-[r] 
 \circ ([]!{+(+.1,+.3)} {\overline{\mathbf{Y}}^{\mathrm{c}}_{2,1}}), 
 -[dl] 
 \circ([]!{+(-.35,+.1)} 
 {\overline{\mathbf{P}}^{0,\mathrm{c}}_{\zeta_1}})
 ( [r] 
 ([]!{+(-.57,+0)} {\cdots} ) 
 ([]!{+(-.19,+0)} {\cdots} ) 
 ([]!{+(+.19,+0)} {\cdots} )
 ([]!{+(+.57,+0)} {\cdots} ) , 
 -[dl] 
 \circ([]!{+(+.2,-.35)} 
 {\overline{\mathbf{X}}^{\mathrm{c}}_{\zeta_1,\zeta'_1}}), 
 -[d] 
 \circ 
 ([]!{+(-.5,+0)} {\cdots} )
 ([]!{+(+.4,-.35)}
 {\overline{\mathbf{X}}^{\mathrm{c}}_{\zeta_1,\zeta'_{q-1}}}) ), 
 -[dr] 
 \circ([]!{+(+.6,+.1)} 
 {\overline{\mathbf{P}}^{0,\mathrm{c}}_{\zeta_{q^2 -1}}}) 
 (-[d] 
 \circ([]!{+(+.3,-.35)} 
 {\overline{\mathbf{X}}^{\mathrm{c}}_{\zeta_{q^2-1},\zeta'_1}}), 
 -[dr] 
 \circ 
 ([]!{+(-.5,+0)} {\cdots} )
 ([]!{+(+.7,-.35)} 
 {\overline{\mathbf{X}}^{\mathrm{c}}_{\zeta_{q^2-1},\zeta'_{q-1}}}) ) 
 ) 
}} 
where 
$\mu_{q^2 -1} (k^{\mathrm{ac}})=
 \{ \zeta_1 ,\ldots ,\zeta_{q^2 -1} \}$, 
$k^{\times} = 
 \{ \zeta'_1 ,\ldots ,\zeta'_{q-1} \}$ and 
$X^{\mathrm{c}}$ denotes 
the smooth compactification 
of the normalization of $X$ for a curve $X$ over $k^{\mathrm{ac}}$. 
The constructed semi-stable model is in fact stable, 
except in the case where $q = 2$. 
If $q=2$, we get the stable model by 
blowing down some $\mathbb{P}^1$-components. 

The realization of the local Jacquet-Langlands correspondence 
in cohomology of Lubin-Tate curves 
was proved in \cite{MieGeomJL} by a purely local method. 
Therefore, the remaining essential part of 
the study of the realization of the local Langlands correspondence 
is to study actions of Weil groups and division algebras. 
In the paper \cite{ITreal3}, 
we give a purely local proof 
of the realization of the local Langlands correspondence 
for representations of conductor three 
using the result of this paper. 

At last, we mention some recent progress on related topics 
according to a suggestion of a referee. 
In \cite{WeSemi}, Weinstein constructs 
semi-stable models of Lubin-Tate curves for arbitrary level 
in the case where the residue characteristic is not equal to two 
using Lubin-Tate perfectoid spaces. 
In \cite{ITepitame} and \cite{ITepiwild}, 
some of our results in this paper are generalized to 
arbitrary dimensional cases for Lubin-Tate perfectoid spaces. 
In \cite{ITLTsurf}, 
we construct an affinoid in the two-dimensional Lubin-Tate space 
such that 
the cohomology of the reduction of the affinoid 
realizes representations 
which are a bit more ramified than 
the epipelagic representations. 

\subsection*{Acknowledgment}
The authors thank Seidai Yasuda for helpful discussion on the 
subject in Paragraph \ref{eveninertia}. 
They are grateful to a referee for suggestions for improvements. 

\subsection*{Notation}
In this paper, we use the following notation. 
Let $K$ be a non-archimedean local field. 
Let $\mathcal{O}_K$ denote the 
ring of integers of $K$, 
and $k$ denote the residue field of $K$. 
Let $p$ be the characteristic of $k$. 
We fix a uniformizer $\varpi$ 
of $K$. 
Let $q=|k|$. 
We fix an algebraic closure 
$K^\mathrm{ac}$ of $K$. 
For any finite extension $F$ of $K$ in $K^\mathrm{ac}$, 
let $G_F$ denote the absolute Galois group of $F$, 
$W_F$ denote the Weil group of $F$ 
and $I_F$ denote the inertia subgroup of $W_F$. 
The completion of $K^\mathrm{ac}$ is denoted 
by $\mathbf{C}$. 
Let $\mathcal{O}_{\mathbf{C}}$ 
be the ring of integers of 
$\mathbf{C}$ and $k^\mathrm{ac}$ 
the residue field of 
$\mathbf{C}$. 
For an element 
$a \in \mathcal{O}_{\mathbf{C}}$, we write 
$\bar{a}$ for the image of $a$ 
by the reduction map 
$\mathcal{O}_{\mathbf{C}} \to k^\mathrm{ac}$. 
Let $v(\cdot)$ denote the valuation of 
$\mathbf{C}$ such that $v(\varpi)=1$. 
Let $K^{\mathrm{ur}}$ 
denote the maximal 
unramified extension of $K$ 
in $K^{\mathrm{ac}}$. 
The completion of 
$K^{\mathrm{ur}}$ is denoted by 
$\widehat{K}^{\mathrm{ur}}$. 
For $a,b \in \mathbf{C}$ and 
a rational number $\alpha \in 
\mathbb{Q}_{\geq 0}$, 
we write $a \equiv b \pmod{\alpha}$
if we have $v(a-b) \geq \alpha$, and 
$a \equiv b \pmod{\alpha +}$
if we have $v(a-b)>\alpha$. 
For a curve $X$ over $k^{\mathrm{ac}}$, 
we denote by $X^c$ 
the smooth compactification 
of the normalization of $X$, and 
the genus of $X$ means the genus of $X^{\mathrm{c}}$. 
For an affinoid 
$\mathbf{X}$, we write 
$\overline{\mathbf{X}}$ for its reduction. 
The category of sets is denoted by 
$\mathbf{Set}$. 
For a representation $\tau$ of a group, 
the dual representation of $\tau$ is denoted by 
$\tau^*$. 
We take rational powers of $\varpi$ compatibly 
as needed. 

\section{Preliminaries}\label{prel}
\subsection{The universal deformation}
Let $\Sigma$ 
denote a formal $\mathcal{O}_K$-module of 
dimension $1$ and height $2$ over $k^{\mathrm{ac}}$, 
which is unique up to isomorphism. 
Let $n$ be a natural number. 
We define $K_1 (\mathfrak{p}^n)$ as in the introduction. 
In the following, we define 
the connected Lubin-Tate curve 
$\mathbf{X}_1 (\mathfrak{p}^n )$ 
with level $K_1 (\mathfrak{p}^n)$. 

Let $\mathcal{C}$ 
be the category of 
Noetherian complete local 
$\mathcal{O}_{\widehat{K}^{\mathrm{ur}}}$-algebras 
with residue field $k^{\mathrm{ac}}$. 
For $A \in \mathcal{C}$, 
a formal $\mathcal{O}_K$-module 
$\mathcal{F} =\Spf A[[X]]$ over $A$ and 
an $A$-valued point $P$ of $\mathcal{F}$, 
the corresponding element of the maximal ideal of $A$ 
is denoted by $x(P)$. 
We consider the functor
\[
 \mathcal{A}_1 (\mathfrak{p}^n)\colon \mathcal{C} \to \mathbf{Set}; 
 A \mapsto [(\mathcal{F}, \iota, P)], 
\]
where 
$\mathcal{F}$ is a formal $\mathcal{O}_K$-module over $A$ 
with an isomorphism 
$\iota\colon \Sigma \simeq \mathcal {F} \otimes_A k^{\mathrm{ac}}$ and 
$P$ is a $\varpi^n$-torsion point of $\mathcal{F}$ 
such that 
\[
 \prod_{a \in \mathcal{O}_K /\varpi^n \mathcal{O}_K} 
 \bigl( X- x([a]_{\mathcal{F}} (P)) \bigr) \biggm| 
 [\varpi^n ]_{\mathcal{F}}(X) 
\]
in $A[[X]]$. This functor is represented by a regular local 
ring $\mathcal{R}_1(\mathfrak{p}^n)$ 
by \cite[\S 4. B) Lemma]{DrEmod}. 
We write $\mathfrak{X}_1(\mathfrak{p}^n)$ for 
$\Spf \mathcal{R}_1(\mathfrak{p}^n)$. 
Its generic fiber 
is denoted by $\mathbf{X}_1(\mathfrak{p}^n)$, 
which we call the connected Lubin-Tate curve 
with level $K_1 (\mathfrak{p}^n)$. 
The space  $\mathbf{X}_1(\mathfrak{p}^n)$ 
is a rigid analytic curve over ${\widehat{K}^{\mathrm{ur}}}$. 
We can define the Lubin-Tate curve 
$\mathrm{LT}_1(\mathfrak{p}^n)$ with level $n$ 
by changing 
$\mathcal{C}$ to be the category of 
$\mathcal{O}_{\widehat{K}^{\mathrm{ur}}}$-algebras where 
$\varpi$ is nilpotent, and 
$\iota$ to be a quasi-isogeny 
$\Sigma \otimes_{k^{\mathrm{ac}}} A/\varpi A \to 
 \mathcal {F} \otimes_A A/\varpi A$. 
We consider $\mathrm{LT}_1(\mathfrak{p}^n)$ 
as a rigid analytic curve over ${\widehat{K}^{\mathrm{ur}}}$. 

The ring $\mathcal{R}_1(1)$ 
is isomorphic to the ring of 
formal power series 
$\mathcal{O}_{{\widehat{K}^{\mathrm{ur}}}}[[u]]$. 
We simply write 
$\mathcal{B}(1)$ for 
$\Spf \mathcal{O}_{{\widehat{K}^{\mathrm{ur}}}}[[u]]$. 
Let $B(1)$ denote an open unit ball such that 
$B(1)(\mathbf{C})=\{ u \in \mathbf{C} \mid v(u)>0 \}$. 
The generic fiber of $\mathcal{B}(1)$ 
is equal to $B(1)$. 
Then, the space $\mathbf{X}_1(1)$ 
is identified with 
$B(1)$. 
Let $\mathcal{F}^{\mathrm{univ}}$ 
denote the universal formal 
$\mathcal{O}_K$-module 
over $\mathfrak{X}_1(1)$. 

In this subsection, we choose a parametrization of 
$\mathfrak{X}_1(1) \simeq \mathcal{B}(1)$ 
such that 
the universal formal $\mathcal{O}_K$-module has a simple 
form. 
Let $\mathcal{F}$ be a formal 
$\mathcal{O}_K$-module of dimension $1$ over a flat 
$\mathcal{O}_K$-algebra $R$. 
For a nontrivial invariant differential $\omega$ on 
$\mathcal{F}$, 
a logarithm of $\mathcal{F}$ 
means a unique isomorphism 
$F \colon \mathcal{F} \stackrel{\sim}{\to} \mathbb{G}_a$ 
over $R \otimes K$ 
with $dF=\omega$ (cf. \cite[3]{GHev}). 
In the sequel, 
we always take an invariant differential $\omega$ on 
$\mathcal{F}$ so that 
a logarithm $F$ has the following form:  
\[
 F(X)=X+\sum_{i \geq 1} f_i X^{q^i} \ 
 \textrm{with}\ f_i \in R \otimes K. 
\]

Let 
$F(X)=\sum_{i \geq 0} f_i X^{q^i} \in K[[u,X]]$ 
be the universal logarithm over 
$\mathcal{O}_K[[u]]$. 
By \cite[(5.5), (12.3), Proposition 12.10]{GHev}, 
the coefficients $\{f_i\}_{i \geq 0}$ 
satisfy $f_0=1$ and 
$\varpi f_i =\sum_{0 \leq j \leq i-1} f_j v_{i-j}^{q^{j}}$ 
for $i \geq 1$, 
where $v_1=u$, $v_2=1$ and $v_i=0$ for $i \geq 3$. 
Hence, we have the following: 
\begin{equation}\label{for}
 f_0=1,\quad 
 f_1=\frac{u}{\varpi},\quad 
 f_2=\frac{1}{\varpi}\biggl(1+\frac{u^{q+1}}{\varpi}\biggr),\quad 
 f_3=\frac{1}{\varpi^2}\biggl(u+u^{q^2}+\frac{u^{q^2+q+1}}{\varpi}
 \biggr), \quad \cdots. 
\end{equation}
By \cite[Proposition 5.7]{GHev} 
or \cite[21.5]{HazFG}, if we set 
\begin{equation}\label{sos}
 \mathcal{F}^{\mathrm{univ}}(X,Y)=F^{-1}(F(X)+F(Y)),\quad 
 [a]_{\mathcal{F}^{\mathrm{univ}}}(X)=F^{-1}(aF(X))
\end{equation}
for $a \in \mathcal{O}_K$, 
it is known that 
these power series have coefficients in 
$\mathcal{O}_K[[u]]$ 
and define the universal 
formal $\mathcal{O}_K$-module 
$\mathcal{F}^{\mathrm{univ}}$ over 
$\mathcal{O}_{{\widehat{K}^{\mathrm{ur}}}}[[u]]$ 
of dimension $1$
and height $2$ with logarithm $F(X)$. 
We have the following approximation formula for 
$[\varpi]_{\mathrm{u}}(X)$. 
\begin{lem}\label{all}
We have the following congruence:  
\begin{align*}
 [\varpi]_{\mathcal{F}^{\mathrm{univ}}}(X) \equiv 
 \varpi X+uX^q +X^{q^2} 
 -\frac{u}{\varpi}\{(uX^q+X^{q^2})^q&-u^qX^{q^2}-X^{q^3}\}\\
 &\bmod (\varpi^2 X^q, u\varpi X^q, 
 \varpi X^{q^2}, X^{q^3+1}). 
\end{align*}
\end{lem}
\begin{proof}
This follows
from a direct computation using 
the relation $F([\varpi]_{\mathcal{F}^{\mathrm{univ}}}(X))=\varpi F(X)$ 
and \eqref{for}. 
\end{proof}
In the sequel, $\mathcal{F}^{\mathrm{univ}}$ 
means the universal formal $\mathcal{O}_K$-module 
with the identification 
$\mathfrak{X}_1(1) \simeq \mathcal{B}(1)$ 
given by \eqref{sos}, 
and we simply write
$[a]_{\mathrm{u}}$ for 
$[a]_{\mathcal{F}^{\mathrm{univ}}}$. 
The reduction of \eqref{sos} gives 
a simple model of $\Sigma$ such that 
\begin{equation}\label{smodel}
 X+_{\Sigma} Y=X+Y, \quad 
 [\zeta]_{\Sigma}(X)=\bar{\zeta}X \quad 
 \textrm{for} \quad \zeta \in \mu_{q-1} (\mathcal{O}_K ),\quad  
 [\varpi]_{\Sigma}(X)=X^{q^2}. 
\end{equation}
We put 
\[
 \mathfrak{A}_n=
 \mathcal{O}_{{\widehat{K}^{\mathrm{ur}}}}[[u,X_n]]/
 \bigl( [\varpi^n]_{\mathrm{u}}(X_n)/[\varpi^{n-1}]_{\mathrm{u}}(X_n) \bigr). 
\]
Then there is a natural identification 
\begin{equation}\label{th}
\mathfrak{X}_1(\mathfrak{p}^n) \simeq 
\Spf \mathfrak{A}_n 
\end{equation}
that is compatible with the identification 
$\mathfrak{X}_1(1) \simeq \mathcal{B}(1)$. 
The Lubin-Tate curve $\mathbf{X}_1(\mathfrak{p}^n)$ 
is identified with the generic fiber of the right hand side of
\eqref{th}. 
We set $X_{i}=[\varpi^{n-i}]_{\mathrm{u}}(X_n)$ 
for $1 \leq i \leq n-1$. 
We write $\mathfrak{X}(1)$ for $\mathfrak{X}_1(1)$. 

\subsection{Action of a division algebra on $\mathfrak{X}_1(\mathfrak{p}^n)$}\label{d_1}
Let $D$ be the central 
division algebra over $K$ of invariant $1/2$. 
We write $\mathcal{O}_D$ 
for the ring of integers of $D$. 
In this subsection, 
we recall the left action of 
$\mathcal{O}_D^{\times}$ 
on the space $\mathfrak{X}_1(\mathfrak{p}^n)$. 

Let $K_2$ be the unramified quadratic extension of $K$. 
Let $k_2$ be the residue field of $K_2$, and 
$\sigma \in \Gal (K_2/K)$ be 
the non-trivial element. 
The ring $\mathcal{O}_D$ 
has the following description: 
$\mathcal{O}_D = \mathcal{O}_{K_2} \oplus \varphi \mathcal{O}_{K_2}$ 
with $\varphi^2=\varpi$ and 
$a \varphi=\varphi a^{\sigma}$ for 
$a \in \mathcal{O}_{K_2}$. 
We define an action of $\mathcal{O}_D$ on 
$\Sigma$ by 
$\zeta(X)=\bar{\zeta} X$ for 
$\zeta \in \mu_{q^2 -1} (\mathcal{O}_{K_2} )$ 
and $\varphi(X)=X^q$. 
Then this give an isomorphism 
$\mathcal{O}_D \simeq \End (\Sigma)$ 
by \cite[Proposition 13.10]{GHev}. 

Let $d=d_1 +\varphi d_2 \in \mathcal{O}_D^{\times}$, 
where $d_1 \in \mathcal{O}^{\times}_{K_2}$ and 
$d_2 \in \mathcal{O}_{K_2}$. 
By the definition of the action of 
$\mathcal{O}_D$ on $\Sigma$, we have 
\begin{equation}\label{g_1}
 d(X) \equiv \bar{d}_1 X+(\bar{d}_2 X)^q 
 \bmod (X^{q^2}). 
\end{equation}
We take a lifting 
$\tilde{d}(X) \in \mathcal{O}_{K_2}[[X]]$ of $d(X) \in k_2[[X]]$. 
Let 
$\mathcal{F}_{\tilde{d}}$ 
be the formal 
$\mathcal{O}_K$-module defined by 
\[
 \mathcal{F}_{\tilde{d}}(X,Y)={\tilde{d}} 
 \bigl( \mathcal{F}^{\mathrm{univ}}
 ( \tilde{d}^{-1}(X),\tilde{d}^{-1}(Y) ) \bigr), \quad 
 [a]_{\mathcal{F}_{\tilde{d}}}(X)=\tilde{d} 
 \bigl( [a]_{\mathrm{u}} ( \tilde{d}^{-1}(X) ) \bigr) 
\]
for $a \in \mathcal{O}_K$. 
Then, we have an isomorphism 
\[
 \tilde{d} \colon \mathcal{F}^{\mathrm{univ}} 
 \stackrel{\sim}{\longrightarrow} 
 \mathcal{F}_{\tilde{d}}; \ 
 (u,X) \mapsto (u,\tilde{d}(X)). 
\]
By \cite[Proposition 14.7]{GHev}, 
the formal $\mathcal{O}_K$-module $\mathcal{F}_{\tilde{d}}$ 
with 
\[
 \Sigma \stackrel{d^{-1}}{\longrightarrow} \Sigma 
 \stackrel{\iota}{\longrightarrow} 
 \mathcal{F}^{\mathrm{univ}} \otimes k^{\mathrm{ac}} 
 \xrightarrow{\tilde{d} \otimes k^{\mathrm{ac}}} 
 \mathcal{F}_{\tilde{d}} \otimes k^{\mathrm{ac}}
\] 
gives an isomorphism 
\begin{equation}\label{caa-1}
 d\colon \mathfrak{X}(1) \to \mathfrak{X}(1), 
\end{equation}
which is independent of a choice 
of a lifting $\tilde{d}$, such that 
there is the unique isomorphism 
\[
 j\colon d^{\ast} \mathcal{F}^{\mathrm{univ}} 
 \stackrel{\sim}{\longrightarrow} 
 \mathcal{F}_{\tilde{d}};\ (u,X) \mapsto (u,j(X))
\]
satisfying 
$j(X) \equiv X \bmod (\varpi, u)$, 
where $d^{\ast} \mathcal{F}^{\mathrm{univ}}$ 
denotes the pull-back of 
$\mathcal{F}^{\mathrm{univ}}$ over $\mathfrak{X}(1)$ 
by the map \eqref{caa-1}. 
Hence, we have 
\begin{equation}\label{caa2}
 [\varpi]_{d^{\ast} \mathcal{F}^{\mathrm{univ}}}(j^{-1}(X))
 =j^{-1}([\varpi]_{\mathcal{F}_{\tilde{d}}}(X)).
\end{equation}
On the other hand, we have the following isomorphism:
\[
 d^{\ast} \mathcal{F}^{\mathrm{univ}} \stackrel{\sim}{\to}
 \mathcal{F}^{\mathrm{univ}}; \ 
 (u,X') \mapsto (d(u),X').\ 
\]
Furthermore, 
we consider the following
isomorphism under the identification \eqref{th}: 
\begin{equation}\label{g_2}
 \psi_d \colon \mathfrak{X}_1(\mathfrak{p}^n) 
 \longrightarrow 
 \mathfrak{X}_1(\mathfrak{p}^n);\ 
 (u,X_n) \mapsto 
 \bigl( d(u),j^{-1} (\tilde{d}(X_n)) \bigr), 
\end{equation}
which depends only on $d$ as 
in \cite[Proposition 14.7]{GHev}. 
We put 
\[
 d^{\ast}(X)=j^{-1} (\tilde{d}(X)). 
\]
We define a left action of $d$ on $\mathfrak{X}_1(\mathfrak{p}^n)$ 
by 
\[
 [(\mathcal{F},\iota,P)] \mapsto [(\mathcal{F},\iota \circ d^{-1} ,P)]. 
\]
Then this action coincides with 
$\psi_d$ by the definition.

By \eqref{g_1}, we have 
\begin{equation}\label{div}
 \tilde{d}^{-1}(X)=
 d_1^{-1} X -d_1^{-(q+1)} d_2^q X^q \bmod (\varpi,X^{q^2}) 
\end{equation}
in $\mathcal{O}_{K_2}[[X]]$. 
We use the following lemma later 
to compute the 
$\mathcal{O}^{\times}_D$-action 
on the stable reduction 
of $\mathbf{X}_1(\mathfrak{p}^3)$. 
\begin{lem}
We assume $v(u)=1/(2q)$. 
Let $d=d_1+\varphi d_2 \in \mathcal{O}^{\times}_D$. 
We set $u'=d(u)$. 
We change variables as 
$u=\varpi^{1/(2q)} \tilde{u}$ and 
$u' =\varpi^{1/(2q)} \tilde{u}'$. 
Then, we have the following: 
\begin{align}
 u' &\equiv d_1^{-(q-1)} u 
 ( 1+d_1^{-q} d_2 u ) 
 \bmod (\varpi, u^3), \label{gh1} \\ 
 j^{-1}(X) &\equiv 
 X+d_1^{-q} d_2 uX \bmod (\varpi, u^2 X, uX^2). \label{gh2}
\end{align}
\end{lem}
\begin{proof}
We set 
$d^{-1} =d_1' +\varphi d_2'$. 
Then $d_1' \equiv d_1^{-1},\ d_2' \equiv -d_1^{-(q+1)} d_2 \pmod{1}$. 
First, we prove \eqref{gh1}. 
If $v(u)=1/(2q)$, a function $w(u)$
in \cite[(25.11)]{GHev}
is well-approximated by 
a function
$\varpi u(\varpi+u^{q+1})^{-1}$. 
By \cite[(25.13)]{GHev}, 
we have 
\[
 \frac{\varpi u'}{\varpi+{u'}^{q+1}}  \equiv 
 \frac{d_1'^q \varpi u(\varpi+u^{q+1})^{-1} +\varpi d_2'^q}
 {d_2' \varpi u(\varpi+u^{q+1})^{-1} +d_1'} 
 \equiv 
 \frac{\varpi u(d_1 - d_2^{q} u^q )}
 {d_1^q (\varpi+u^{q+1})-d_2 \varpi u} 
 \pmod{1+}. 
\]
Hence, we acquire the following by
$u=\varpi^{1/(2q)} \tilde{u}$ and 
$u' =\varpi^{1/(2q)} \tilde{u}'$: 
\begin{equation}\label{gi}
 \frac{\tilde{u}'}{\tilde{u}'^{q+1}+\varpi^{\frac{q-1}{2q}}} 
 \equiv 
 \frac{\tilde{u} (d_1 -\varpi^{\frac{1}{2}} d_2^q \tilde{u}^q )} 
 {d_1^q \tilde{u}^{q+1}
 +\varpi^{\frac{q-1}{2q}} d_1^q -\varpi^{\frac{1}{2}} d_2 \tilde{u}} 
 \pmod{\frac{1}{2}+}. 
\end{equation}
By taking an inverse of 
the congruence \eqref{gi}, 
we obtain 
\begin{equation}\label{qsa}
 (\tilde{u}'-d_1^{-(q-1)}\tilde{u})^q 
 \equiv \varpi^{\frac{q-1}{2q}} \biggl(
 \frac{\tilde{u}'-d_1^{-(q-1)}\tilde{u}}
 {d_1^{-(q-1)} \tilde{u} \tilde{u}'}\biggr)
 +\varpi^{\frac{1}{2}} 
 (d_1^{q-2} d_2^q \tilde{u}^{2q} -d_1^{-1} d_2 ) \pmod{\frac{1}{2}+}. 
\end{equation}
Now, we set $\tilde{u}'-d_1^{-(q-1)} \tilde{u}=\varpi^{1/(2q)} x$. 
By substituting this to \eqref{qsa} and dividing
it by $\varpi^{1/2}$, we obtain
\[
 (x-d_1^{1-2q} d_2 \tilde{u}^2)^q 
 \equiv d_1^{2q-2} \tilde{u}^{-2}
 (x-d_1^{1-2q} d_2 \tilde{u}^2) \pmod{0+}. 
\]
Since $x$ is an analytic function of 
$\tilde{u}$, a congruence
$x \equiv d_1^{1-2q} d_2 \tilde{u}^2 \pmod{0+}$ 
must hold. 
Hence we have 
\[
 \tilde{u}' \equiv d_1^{-(q-1)} \tilde{u} 
 ( 1+\varpi^{\frac{1}{2q}} d_1^{-q} d_2 \tilde{u} ) 
 \pmod{\frac{1}{2q}+} 
\]
using $\tilde{u}'-d_1^{q-1}\tilde{u}=\varpi^{1/(2q)} x$. 
This implies \eqref{gh1}, 
because $u'$ is an analytic function of $u$. 

By Lemma \ref{all}, 
\eqref{caa2} and \eqref{div}, we have 
$u' j^{-1}(X)^q \equiv 
 j^{-1}(ud_1^{-(q-1)} X^q) 
 \bmod (\varpi, X^{q^2})$. 
Hence, 
the assertion 
\eqref{gh2} follows from \eqref{gh1} 
and $j^{-1} (X) \equiv X \bmod (\varpi, u)$. 
\end{proof}

\section{Cohomology of Lubin-Tate curve}\label{cohLT}
Let $\ell$ be a prime number different from $p$. 
We take an algebraic closure 
$\overline{\mathbb{Q}}_{\ell}$ of 
$\mathbb{Q}_{\ell}$. 
Let $\mathrm{LT}(\mathfrak{p}^n)$ 
be the Lubin-Tate curve with 
full level $n$ over $\widehat{K}^{\mathrm{ur}}$ 
(cf.~\cite[3.2]{DaLTell}). 
We put 
\[
 H^i _{\mathrm{LT},\varpi} = 
 \varinjlim_n H^i _{\mathrm{c}} 
 \bigl( ( \mathrm{LT}(\mathfrak{p}^n) /
 \varpi^{\mathbb{Z}} )_{\mathbf{C}}, 
 \overline{\mathbb{Q}}_{\ell} \bigr) 
\] 
for any non-negative integer $i$, 
where 
$\mathrm{LT}(\mathfrak{p}^n) / \varpi^{\mathbb{Z}}$ 
denote the quotient of 
$\mathrm{LT}(\mathfrak{p}^n)$ by the action of 
$\varpi^{\mathbb{Z}} \subset D^{\times}$. 
Then we can define an action of 
$\mathrm{GL}_2 (K) \times D^{\times} \times W_K$ 
on $H^i _{\mathrm{LT},\varpi}$ for a non-negative integer $i$ 
(cf. \cite[3.2, 3.3]{DaLTell}). 

We write 
$\Irr (D^{\times},\overline{\mathbb{Q}}_{\ell})$ 
for the set of 
isomorphism classes of irreducible smooth representations of 
$D^{\times}$ over 
$\overline{\mathbb{Q}}_{\ell}$, 
and 
$\Disc (\mathrm{GL}_2 (K),\overline{\mathbb{Q}}_{\ell})$ 
for the set of 
isomorphism classes of irreducible 
discrete series representations of 
$\mathrm{GL}_2 (K)$ over 
$\overline{\mathbb{Q}}_{\ell}$. 
Let 
\[
 \mathrm{JL} \colon 
 \Irr (D^{\times},\overline{\mathbb{Q}}_{\ell}) 
 \to \Disc (\mathrm{GL}_2 (K),\overline{\mathbb{Q}}_{\ell}) 
\]
be the local Jacquet-Langlands correspondence. 
We denote by $\mathrm{LJ}$ the inverse of 
$\mathrm{JL}$. 
For an irreducible smooth representation $\pi$ of 
$\mathrm{GL}_2 (K)$, 
let $\omega_{\pi}$ denote the central character of $\pi$. 
We write $\mathrm{St}$ for the Steinberg representation of 
$\mathrm{GL}_2(K)$. 

The following fact is well known as a 
corollary of the Deligne-Carayol conjecture. 
Here, we give a purely local proof of this fact. 

\begin{prop}\label{cohGL}
We have isomorphisms 
\begin{align*}
 H^1 _{\mathrm{LT},\varpi} &\simeq 
 \bigoplus_{\pi} 
 \pi ^{\oplus 2\dim \mathrm{LJ}(\pi)} \oplus 
 \bigoplus_{\chi} (\mathrm{St} \otimes (\chi \circ \det)),\\ 
 H^2 _{\mathrm{LT},\varpi} &\simeq 
 \bigoplus_{\chi} (\chi \circ \det)
\end{align*}
as representations of 
$\mathrm{GL}_2 (K)$, 
where $\pi$ runs through irreducible cuspidal representations of 
$\mathrm{GL}_2 (K)$ such that 
$\omega_{\pi} (\varpi)=1$, and 
$\chi$ runs through characters of $K^{\times}$ 
satisfying $\chi (\varpi^2)=1$. 
\end{prop}
\begin{proof}
First, we show the second isomorphism. 
Let $\mathbf{X}(\mathfrak{p}^n)$ 
be the connected Lubin-Tate curve with 
full level $n$ over $\widehat{K}^{\mathrm{ur}}$ 
(cf.~\cite[2.1]{StrCcomp}). 
We put 
\begin{align*}
 H^2_{\mathbf{X}} &=\varinjlim_n H^2 _{\mathrm{c}} 
 \bigl( \mathbf{X}(\mathfrak{p}^n)_{\mathbf{C}}, 
 \overline{\mathbb{Q}}_{\ell} \bigr), \\ 
 \mathrm{GL}_2 (K)^0 &=\{ g \in \mathrm{GL}_2 (K) \mid 
 \det g \in \mathcal{O}_K^{\times} \}. 
\end{align*}
Then $\mathrm{GL}_2 (K)^0$ acts on $H^2_{\mathbf{X}}$. 
By \cite[Theorem 4.4 (i)]{StrCcomp}, we have 
\begin{equation}\label{H2chi}
 H^2_{\mathbf{X}} \simeq 
 \bigoplus_{\chi} (\chi \circ \det) 
\end{equation}
as representations of $\mathrm{GL}_2 (\mathcal{O}_K)$, 
where $\chi$ runs through characters of 
$\mathcal{O}_K^{\times}$. 
Let $H$ be the kernel of 
$\mathrm{GL}_2 (K)^0 \to \Aut (H^2_{\mathbf{X}})$. 
Then $H =\mathrm{SL}_2 (K)$, 
because a normal subgroup of 
$\mathrm{GL}_2 (K)^0$ containing 
$\mathrm{SL}_2 (\mathcal{O}_K)$ is 
$\mathrm{SL}_2 (K)$ 
by \cite[Lemme 2.2.5 (iii)]{DelGL2}. 
Hence, we see that 
\eqref{H2chi} is an isomorphism 
as representations of $\mathrm{GL}_2 (K)^0$. 
The second isomorphism follows from this, 
because we have 
\[
 H^2 _{\mathrm{LT},\varpi} \simeq 
 \mathrm{c\mathchar`-Ind}_{\mathrm{GL}_2 (K)^0}
 ^{\mathrm{GL}_2 (K)/\varpi^{\mathbb{Z}}} 
 H^2_{\mathbf{X}}. 
\]

Next, we show the first isomorphism. 
By \cite[Definition 6.2 and Theorem 6.6]{MieGeomJL}, 
the cuspidal part of 
$H^1 _{\mathrm{LT},\varpi}$ is 
\[
 \bigoplus_{\pi} 
 \pi ^{\oplus 2\dim \mathrm{LJ}(\pi)}. 
\]
Here, we note that 
the characteristic of a local field is assumed to be 
zero in 
\cite{MieGeomJL}, 
but the same proof works 
in the equal characteristic case. 
By \cite[Th\'{e}or\`{e}me 4.3]{FaDZel} and 
the Faltings-Fargues isomorphism 
(cf.~\cite{Faltwo} and \cite{FGLisom}), 
we see that the non-cuspidal part of 
$H^1 _{\mathrm{LT},\varpi}$ 
is the Zelevinsky dual of $H^2 _{\mathrm{LT},\varpi}$. 
Therefore, we have the first isomorphism. 
\end{proof}

\section{Stable covering of Lubin-Tate curve with level two}\label{level2}
In this section, 
we construct a stable covering of 
$\mathbf{X}_1(\mathfrak{p}^2)$. 
Let $(u,X_2)$ be the parameter of 
$\mathbf{X}_1(\mathfrak{p}^2)$ given 
by the identification \eqref{th}. 

Let $\mathbf{Y}_{1,1}$, $\mathbf{W}_0$, 
$\mathbf{W}_{k^{\times}}$, 
$\mathbf{W}_{\infty,1}$, 
$\mathbf{W}_{\infty,2}$ and 
$\mathbf{W}_{\infty,3}$ 
be the subspaces of 
$\mathbf{X}_1(\mathfrak{p}^2)$ 
defined by the following conditions respectively: 
\begin{align*}
 \mathbf{Y}_{1,1} &\colon 
 v(u)=\frac{1}{q+1},\quad v(X_1)=\frac{q}{q^2-1},\quad 
 v(X_2)=\frac{1}{q(q^2-1)}.\\ 
 \mathbf{W}_0 &\colon 
 0<v(u)<\frac{1}{q+1},\quad 
 v(X_1)=\frac{1-v(u)}{q-1},\quad 
 v(X_2)=\frac{1-qv(u)}{q(q-1)}.\\ 
 \mathbf{W}_{k^{\times}} &\colon 
 0<v(u)<\frac{1}{q+1},\quad 
 v(X_1)=\frac{1-v(u)}{q-1},\quad 
 v(X_2)=\frac{v(u)}{q(q-1)}.\\ 
 \mathbf{W}_{\infty,1} &\colon 
 0<v(u)<\frac{q}{q+1},\quad 
 v(X_1)=\frac{v(u)}{q(q-1)},\quad 
 v(X_2)=\frac{v(u)}{q^3(q-1)}.\\ 
 \mathbf{W}_{\infty,2} &\colon 
 v(u) \geq \frac{q}{q+1},\quad
 v(X_1)=\frac{1}{q^2-1},\quad 
 v(X_2)=\frac{1}{q^2(q^2-1)}.\\ 
 \mathbf{W}_{\infty,3} &\colon 
 \frac{1}{q+1}<v(u)<\frac{q}{q+1},\quad 
 v(X_1)=\frac{1-v(u)}{q-1},\quad 
 v(X_2)=\frac{1-v(u)}{q^2(q-1)}. 
\end{align*}
We put 
\[
 \mathbf{W}_{\infty} =\mathbf{W}_{\infty,1} \cup 
 \mathbf{W}_{\infty,2} \cup \mathbf{W}_{\infty,3}. 
\]
Note that we have 
\[
 \mathbf{X}_1(\mathfrak{p}^2)=\mathbf{Y}_{1,1} \cup 
 \mathbf{W}_0 \cup \mathbf{W}^1_{k^{\times}}
 \cup \mathbf{W}_{\infty}. 
\]
\begin{prop}\label{ltwo}
The Lubin-Tate curve 
$\mathbf{X}_1(\mathfrak{p}^2)$ is 
a basic wide open space with underlying affinoid 
$\mathbf{Y}_{1,1}$. 
Further, 
$\mathbf{W}_{0}$ and $\mathbf{W}_{\infty}$ are 
open annuli, 
and 
$\mathbf{W}_{k^{\times}}$ is a disjoint union of 
$q-1$ open annuli. 
\end{prop}
\begin{proof}
This is proved in \cite{ITsslow} 
by direct calculations 
without cohomological arguments. 
Here, we sketch another proof based on 
arguments in this paper. 

First, we note that 
$\mathcal{X}_1(1)$ is a good formal model of 
$\mathbf{X}_1(1)$. 
Then, we can show that 
$\mathbf{X}_1(\mathfrak{p})$ is isomorphic to 
an open annulus 
by a cohomological argument 
as in the proof of Theorem \ref{cov3} 
using the natural level-lowering map 
$\mathbf{X}_1(\mathfrak{p}) \to \mathbf{X}_1(1)$. 

Next, we can see that the reduction of 
$\mathbf{Y}_{1,1}$ 
is isomorphic to the affine curve defined by
$x^q y - xy^q = 1$ 
by a calculation as in the proof of 
Proposition \ref{Y12} (cf.~\cite[\S 3.1]{ITsslow}). 
Then we can prove the claim 
by a similar argument as above 
using the natural level-lowering map 
$\mathbf{X}_1(\mathfrak{p}^2) \to \mathbf{X}_1(\mathfrak{p})$. 
\end{proof}

\section{Reductions of affinoid spaces in $\mathbf{X}_1(\mathfrak{p}^3)$}\label{redaff}

\subsection{Definitions of several subspaces in $\mathbf{X}_1(\mathfrak{p}^3)$}\label{sec1}
In this subsection, 
we define several 
subspaces 
of $\mathbf{X}_1(\mathfrak{p}^3)$. 
Let $(u,X_3)$ be the parameter of 
$\mathbf{X}_1(\mathfrak{p}^3)$ given 
by the identification \eqref{th}. 

Let $\mathbf{Y}_{1,2}$, $\mathbf{Y}_{2,1}$ and 
$\mathbf{Z}^0_{1,1}$ 
be the subspaces of $\mathbf{X}_1(\mathfrak{p}^3)$ 
defined by the following conditions respectively: 
\begin{align*}
 \mathbf{Y}_{1,2} &\colon v(u)=\frac{1}{q+1},\quad 
 v(X_1)=\frac{q}{q^2-1},\quad 
 v(X_2)=\frac{1}{q(q^2-1)},\quad 
 v(X_3)=\frac{1}{q^3(q^2-1)}.\\ 
 \mathbf{Y}_{2,1} &\colon 
 v(u)=\frac{1}{q(q+1)},\quad v(X_1)
 =\frac{q^2+q-1}{q(q^2-1)},\quad 
 v(X_2)=\frac{1}{q^2-1},\quad v(X_3)
 =\frac{1}{q^2(q^2-1)}.\\ 
 \mathbf{Z}^0_{1,1} &\colon 
 v(u)=\frac{1}{2q},\quad 
 v(X_1)=\frac{2q-1}{2q(q-1)},\quad 
 v(X_2)=\frac{1}{2q(q-1)},\quad 
 v(X_3)=\frac{1}{2q^3(q-1)}. 
\end{align*}
We write down the 
following possible cases 
for $(u, X_1 , X_2)$:  
\begin{equation}\label{nnum}
\begin{split}
 &1.\quad 0<v(u)<\frac{1}{q+1},\quad v(X_1)=\frac{1-v(u)}{q-1},\quad 
 v(X_2)=\frac{1-qv(u)}{q(q-1)}, \\ 
 &2.\quad 0<v(u)<\frac{1}{q+1},\quad v(X_1)=\frac{1-v(u)}{q-1},\quad 
 v(X_2)=\frac{v(u)}{q(q-1)}, \\
 &3.\quad v(u)=\frac{1}{q+1},\quad v(X_1)=\frac{q}{q^2-1},\quad 
 v(X_2)=\frac{1}{q(q^2-1)}, \\
 &4.\quad 0<v(u) <\frac{q}{q+1},\quad v(X_1)=\frac{v(u)}{q(q-1)}, \quad 
 v(X_2)=\frac{v(u)}{q^3(q-1)}, \\
 &5.\quad v(u) \geq \frac{q}{q+1},\quad v(X_1)=\frac{1}{q^2-1},\quad 
 v(X_2)=\frac{1}{q^2(q^2-1)}, \\ 
 &6.\quad \frac{1}{q+1} <v(u)< \frac{q}{q+1},\quad 
 v(X_1)=\frac{1-v(u)}{q-1},\quad 
 v(X_2)=\frac{1-v(u)}{q^2(q-1)}. 
\end{split}
\end{equation}
Next, we consider the 
following possible 
cases for $(X_2 ,X_3)$: 
\begin{equation}\label{qqqq'1}
\begin{split}
 &1'.\ v(X_3^{q^2})=v(X_2)<v(uX_3^q),\quad  
 2'.\ v(uX_3^{q})=v(X_2)<v(X_3^{q^2}),\\
 &3'.\ v(X_2)>v(X_3^{q^2})=v(uX_3^q),\quad 
 4'.\ v(X_2)=v(X_3^{q^2})=v(uX_3^q). 
\end{split}
\end{equation}
\begin{lem}
For $2 \leq i \leq 6$ in \eqref{nnum} 
and $2' \leq j' \leq 4'$
in \eqref{qqqq'1}, 
the case $i$ and $j'$ does not happen. 
\end{lem}
\begin{proof}
This is an easy exercise. 
\end{proof}

Let $\mathbf{W}_{i,j'}$ be the subspace 
of $\mathbf{X}_1(\mathfrak{p}^3)$ 
defined by the conditions $1 \leq i \leq 6$
in \eqref{nnum} and $1' \leq j' \leq 4'$
in \eqref{qqqq'1}. 
We note that 
$\mathbf{W}_{3,1'} = \mathbf{Y}_{1,2}$ and 
$\mathbf{W}_{1,4'} = \mathbf{Y}_{2,1}$. 
Let $\mathbf{W}^{+}_{1,1'}$ and $\mathbf{W}^{-}_{1,1'}$ 
be the subspaces of $\mathbf{W}_{1,1'}$ 
defined by 
$1/(2q)<v(u)<1/(q+1)$ and $1/(q(q+1))<v(u)<1/(2q)$ 
respectively.

\subsection{Reductions of the affinoid spaces $\mathbf{Y}_{1,2}$ and $\mathbf{Y}_{2,1}$}\label{redY}
In this subsection, 
we compute the 
reductions of the affinoid spaces 
$\mathbf{Y}_{1,2}$ 
and $\mathbf{Y}_{2,1}$. 
The reductions of 
$\mathbf{Y}_{2,1}$ 
and $\mathbf{Y}_{1,2}$ 
are isomorphic to the affine curve defined by $x^qy-xy^q=1$. 
These curves have genus $q(q-1)/2$. 

\begin{prop}\label{Y12}
The reduction of 
$\mathbf{Y}_{1,2}$ 
is isomorphic to the affine curve defined by $x^qy-xy^q=1$. 
\end{prop}
\begin{proof}
We change variables as 
$u=\varpi^{1/(q+1)} \tilde{u}$, 
$X_1=\varpi^{q/(q^2 -1)}x_1$, 
$X_2=\varpi^{1/(q(q^2 -1))}x_2$ 
and 
$X_3=\varpi^{1/(q^3 (q^2 -1))} x_3$. 
By Lemma \ref{all}, 
we have 
\begin{equation}\label{Y121}
 \tilde{u} \equiv -x_1^{-(q-1)},\quad 
 x_1 \equiv \tilde{u}x_2^q+x_2^{q^2},\quad 
 x_2 \equiv x_3^{q^2} \pmod{0+}. 
\end{equation}
Then we have 
$\tilde{u} = -x_1^{-(q-1)} +F_0(\tilde{u},x_1)$ 
for some function $F_0(\tilde{u},x_1)$ 
satisfying $v(F_0(\tilde{u},x_1))>v(\tilde{u})$. 
Substituting $\tilde{u} = -x_1^{-(q-1)} +F_0(\tilde{u},x_1)$ 
to $F_0(\tilde{u},x_1)$ and repeating it, 
we see that 
$\tilde{u}$ is written as a function of $x_1$. 
Similarly, by $x_2 \equiv x_3^{q^2} \pmod{0+}$, 
we can see that 
$x_2$ is written as a function of $x_1$ and $x_3$. 
By \eqref{Y121}, we acquire 
\begin{equation}\label{124}
 1 \equiv \frac{x_3^{q^4}}{x_1}
 -\frac{x_3^{q^3}}{x_1^q} \pmod{0+}. 
\end{equation}
By setting 
$1+x_1^{-1} x_3^{q^2}=x_3^{q^3} t_1^{-1}$ 
and substituting this to \eqref{124}, 
we obtain $t_1 ^q \equiv x_1 \pmod{0+}$ 
and hence 
$(1+x_3^{q}t_1^{-1})^q \equiv x_3^{q^3} t_1^{-1} \pmod{0+}$.
By setting 
$1+x_3^{q}t_1^{-1}=x_3^{q^2} t_2^{-1}$, 
we obtain 
$t_2 ^q \equiv t_1 \pmod{0+}$. Hence 
$(1+x_3t_2^{-1})^q \equiv x_3^{q^2} t_2^{-1} \pmod{0+}$. 
Finally, by setting 
$x=x_3$ and 
$1+x_3t_2^{-1}=x_3^qy$, 
we acquire $y^q \equiv t_2^{-1} \pmod{0+}$. 
Hence we have $x^qy-xy^q \equiv1 \pmod{0+}$. 
Note that 
\begin{equation}\label{gk1}
x=x_3, \quad 
y=\frac{x_1(1+x_3^{q(q^2-1)}+x_3^{(q+1)(q^2-1)})+
x_3^{q^2}}{x_1x_3^{q^3+q^2-1}}, 
\end{equation}
which we will use later. 
\end{proof}

We put 
\[
 \gamma_i =\varpi^{\frac{q-1}{2q^i}} 
\] 
for 
$1 \leq i \leq 4$. 
We choose an element $c_0$ such that 
$c_0 ^q -\gamma_1 ^2 c_0 +1=0$. 
Note that we have $c_0 \equiv -1 \pmod{0+}$. 
Further, we choose a $q$-th root 
$c_0 ^{1/q}$ of $c_0$. 

\begin{prop}\label{Y21}
The reduction of the space $\mathbf{Y}_{2,1}$
is isomorphic to the affine curve defined by $x^q y-xy^q =1$. 
\end{prop}
\begin{proof}
We change variables as 
$u=\varpi^{1/(q(q+1))} \tilde{u}$, 
$X_1=\varpi^{(q^2+q-1)/(q(q^2 -1))} x_1$, 
$X_2=\varpi^{1/(q^2 -1)} x_2$, 
and
$X_3=\varpi^{1/(q^2 (q^2 -1))} x_3$. 
By Lemma \ref{all}, 
we have 
\begin{align}
 \tilde{u} &\equiv -x_1^{-(q-1)} 
 \pmod{\frac{q^2 -1}{q^2} +}, \label{211} \\
 x_1 &\equiv \tilde{u} x_2 ^q +\gamma_1 ^2 
 (x_2 ^{q^2}+x_2 ) 
 \pmod{\frac{q^2 -1}{q^2} +}, \label{212} \\
 x_2 &\equiv x_3^{q^2}+\tilde{u}x_3^q 
 \pmod{\frac{q-1}{q^2}+}. \label{213} 
\end{align} 
By \eqref{211} and \eqref{213}, 
we can see that 
$\tilde{u}$ is written as a function of $x_1$, 
and that 
$x_2$ is written as a function of $x_1$ and $x_3$. 
We define a parameter 
$t$ by 
\begin{equation}\label{214}
 \frac{x_2}{x_1} 
 =c_0+\gamma_2 ^2 \frac{x_2^q}{t}. 
\end{equation}
We note that $v(t)=0$. 
By considering 
$x_1^{-1} \times \eqref{212}$, 
we have 
\begin{equation}\label{215}
 \biggl( \frac{x_2}{x_1} \biggr)^q 
 +1-\gamma_1 ^2 \frac{x_2}{x_1} \equiv 
 \gamma_1 ^2 \frac{x_2 ^{q^2}}{x_1} 
 \pmod{\frac{q^2 -1}{q^2} +}. 
\end{equation}
By substituting \eqref{214} 
to the left hand side of the congruence \eqref{215},
and dividing it by 
$\gamma_1^2 x_2^{q^2}$, we acquire 
\begin{equation}\label{216}
 x_1 \equiv t^q \biggl( 1-\gamma_2 ^2 
 \frac{t^{q-1}}{x_2 ^{q(q-1)}} \biggr)^{-1} 
 \pmod{\frac{q -1}{q^2}+}. 
\end{equation}
By this congruence, we can see that 
$x_1$ is written as a function of $t$ and $x_3$. 
By considering $x_1^{-1} \times \eqref{213}$, 
we acquire 
\begin{equation}\label{217}
 c_0 +\gamma_2 ^2 \frac{x_2^q}{t} \equiv
 \frac{x_3^{q^2}}{x_1}-\biggl(\frac{x_3}{x_1}\biggr)^q 
 \pmod{\frac{q -1}{q^2}+} 
\end{equation}
by \eqref{214}. 
Substituting \eqref{216} to \eqref{217}, we have 
\begin{equation}\label{218}
 \biggl( c_0 ^{1/q}-\frac{x_3^q}{t}+\frac{x_3}{x_1}\biggr)^q
 \equiv -\gamma_2 ^2 
 \frac{(x_2+x_3)^{q^2}}{t x_2^{q(q-1)}} 
 \pmod{\frac{q -1}{q^2}+}. 
\end{equation}
By \eqref{214} and $c_0 \equiv -1 
 \pmod{0+}$, we have 
 $x_2 \equiv -x_1 \pmod{0+}$. 
 Therefore, we acquire 
\[
 (x_2+x_3)^q \equiv x_2^{q-1}x_3^{q^2} 
 \pmod{0+}
\] 
by \eqref{211} and \eqref{213}.
In particular, we obtain 
$v(x_2+x_3)=0$. 
We introducing a 
new parameter $t_1$ as 
\begin{equation}\label{219}
 c_0 ^{1/q}
 -\frac{x_3^q}{t}+\frac{x_3}{x_1}
 =-\gamma_3 ^q 
 \frac{(x_2 +x_3 )^q}{t_1 x_2 ^{q-1}}. 
\end{equation}
Substituting this to the 
left hand side of 
the congruence \eqref{218},
and dividing it by 
$-\gamma_2 ^2 x_2 ^{-q(q-1)} (x_2+x_3)^{q^2}$, 
we acquire 
$t \equiv t_1^q \pmod{0+}$. 
By this congruence, we can see that 
$t$ is written as a function of $t_1$ and $x_3$. 
By \eqref{219}, we obtain 
\[
 x_3 \equiv t_1^{q^2}(1+x_3t_1^{-1})^q \pmod{0+} 
\]
using $t \equiv t_1^q \pmod{0+}$ 
and $x_1 \equiv t^q \pmod{0+}$. 
Hence, by setting 
$x=t_1 ^{-1}$ and 
$\ y=t_1 ^q (1+x_3 t_1 ^{-1} )$, 
we acquire $x^q y-y x^q \equiv 1 \pmod{0+}$. 
\end{proof}

\subsection{Reduction of the affinoid space $\mathbf{Z}^0_{1,1}$}\label{sec2}
In this subsection, we 
calculate the reduction of the 
affinoid space $\mathbf{Z}^0_{1,1}$. 
We define $\mathcal{S}_1$ as in the introduction. 
The reduction 
$\overline{\mathbf{Z}}^0 _{1,1}$ 
is isomorphic to the affine curve defined by 
$Z^q+x_3^{q^2-1}+x_3^{-(q^2-1)}=0$. 
This affine curve has genus $0$ and 
singularities at $x_3 \in \mathcal{S}_1$. 

We put 
\[
 \omega_i =\varpi^{\frac{1}{2q^i(q-1)}}, 
 \quad 
 \epsilon_i =\frac{1}{2q^i} 
\] 
for $1 \leq i \leq 4$. 
We change variables as
$u=\omega_1^{q-1} \tilde{u}$, 
$X_1=\omega_1^{2q-1} x_1$, 
$X_2=\omega_1 x_2$ 
and
$X_3=\omega_3 x_3$. 
By Lemma \ref{all}, 
we have 
\begin{align}
 \tilde{u} &\equiv 
 -x_1^{-(q-1)} \pmod{\frac{1}{2} +}, \label{31} \\
 x_1 &\equiv \tilde{u}x_2^q+
 \gamma_1 x_2^{q^2}+
 \gamma_1 ^2 x_2 \pmod{\frac{1}{2} +}, \label{32} \\
 x_2 &\equiv x_3^{q^2}+
 \gamma_2 \tilde{u} x_3^{q} \pmod{\epsilon_1 +}. \label{33}
\end{align}
Note that we have $v(\gamma_1 ^2) >1/2$ if $q \neq 2$. 
By \eqref{31} and \eqref{33}, 
we can see that 
$\tilde{u}$ is written as a function of $x_1$, 
and that 
$x_2$ is written as a function of $x_1$ and $x_3$. 
We define a parameter $t$ by 
\begin{equation}\label{34}
 \frac{x_2}{x_1}=
 -1+\gamma_2 \frac{x_2^q}{t}.
\end{equation}
By considering $x_1^{-1} \times \eqref{32}$, 
we acquire 
\begin{equation}\label{35}
 \biggl(\frac{x_2}{x_1}+1\biggr)^q \equiv \gamma_1 
 \frac{x_2^{q^2}}{x_1}\biggl(1
 +\frac{\gamma_1}{x_2^{q^2-1}}\biggr) 
 \pmod{\frac{1}{2} +} 
\end{equation}
by \eqref{31}. 
Substituting \eqref{34} to \eqref{35}, 
and dividing it by $\gamma_1 x_2^{q^2}$, 
we obtain 
\begin{equation}\label{36}
 x_1 \equiv t^q\biggl(1+\frac{\gamma_1}{x_2^{q^2-1}}
 \biggr) \pmod{\epsilon_1 +}.
\end{equation}
Therefore we have $v(t)=0$. 
By considering 
$x_1^{-1} \times \eqref{33}$, 
we acquire 
\begin{equation}\label{37}
 \biggl(1+\frac{x_3^q}{t}\biggr)^q
 -\gamma_1 \frac{x_3^{q^2}}{t^qx_2^{q^2-1}} 
 \equiv 
 \gamma_2 \biggl(\frac{x_2^q}{t}
 +\biggl(\frac{x_3}{x_1}\biggr)^q
 \biggr) 
 \pmod{\epsilon_1 +}
\end{equation}
by \eqref{31}, \eqref{34} and \eqref{36}. 
We define a parameter $Z_0$ by 
\begin{equation}\label{38}
 1+\frac{x_3^q}{t}=\gamma_3 Z_0. 
\end{equation} 
We note that $v(Z_0 ) \geq 0$. 
Substituting this 
to \eqref{37}, 
and dividing 
it by $\gamma_2$, 
we obtain 
\begin{equation}\label{39}
 Z_0 ^q \equiv 
 \frac{x_2^q}{t}+
 \biggl(\frac{x_3}{x_1}\biggr)^q
 +\gamma_2^{q-1}
 \frac{x_3^{q^2}}
 {t^qx_2^{q^2-1}} 
 \pmod{\epsilon_2 +}.
\end{equation}
By \eqref{38} and \eqref{39}, 
we acquire 
\begin{equation}\label{40}
 \biggl(Z_0 +\frac{x_2}{x_3}
 -\frac{x_3}{x_1}\biggr)^q 
 \equiv \gamma_3
 \biggl(\frac{x_2}{x_3}\biggr)^q Z_0
 +\gamma_2^{q-1}\frac{x_3^{q^2}}{t^qx_2^{q^2-1}} 
 \pmod{\epsilon_2 +}.
\end{equation}
We introduce 
a new parameter $Z$ as 
\begin{equation}\label{41'}
 Z_0 +\frac{x_2}{x_3}-\frac{x_3}{x_1}
 =\gamma_4 \frac{x_2}{x_3} Z. 
\end{equation}
We note that $v(Z) \geq 0$. 
Substituting this to 
the left hand 
side of the congruence
\eqref{40}, and 
dividing it by 
$\gamma_3(x_2/x_3)^q$,
we acquire 
\begin{equation}\label{qz}
 Z^q \equiv Z_0 +\gamma_3^{q^2-q-1}
 \frac{x_3^{q(q+1)}}{t^qx_2^{q^2+q-1}} 
 \pmod{\epsilon_3 +}.
\end{equation}
By substituting 
\eqref{41'} to \eqref{qz}, 
we obtain 
\begin{equation}\label{41}
 Z^q+x_3^{q^2-1}(1-\gamma_4 Z)
 +x_3^{-(q^2-1)} \equiv 
 -\gamma_3^{q^2-q-1}
 x_3^{-q(q^2-1)(q+1)} 
 \pmod{\epsilon_3 +} 
\end{equation}
by \eqref{33}, \eqref{36} and 
\eqref{38}. 
Note that we have 
$v(\gamma_3 ^{q^2-q-1}) > \epsilon_3$, 
if $q \neq 2$. 
\begin{prop}\label{wee}
The reduction of the space 
$\mathbf{Z}^0 _{1,1}$
is isomorphic to the affine curve defined by 
$Z^q+x_3^{q^2-1}+x_3^{-(q^2-1)}=0$. 
This affine curve has genus $0$ and singularities at 
$x_3 \in \mathcal{S}_1$. 
\end{prop}
\begin{proof}
The required assertion 
follows from 
the congruence \eqref{41}
modulo $0+$. 
\end{proof}

\begin{defn}
1.\ For any 
$\zeta \in \mathcal{S}_1$, 
we define a subspace 
\[
 \mathcal{D}_{\zeta} \subset 
 \mathbf{Z}^0 _{1,1} \times_{{\widehat{K}^{\mathrm{ur}}}} 
 {\widehat{K}^{\mathrm{ur}}}(\omega_3 ) 
\] 
by
$\bar{x}_3=\zeta$. 
We call the space 
$\mathcal{D}_{\zeta}$
a singular residue 
class of $\mathbf{Z}^0_{1,1}$. \\
2.\ We define 
a subspace 
\[
 \mathbf{Z}_{1,1} \subset 
 \mathbf{Z}^0_{1,1}\times_{{\widehat{K}^{\mathrm{ur}}}} 
 {\widehat{K}^{\mathrm{ur}}}(\omega_3 ) 
\]
by the complement 
$\mathbf{Z}^0_{1,1}\times_{{\widehat{K}^{\mathrm{ur}}}} 
 {\widehat{K}^{\mathrm{ur}}}(\omega_3 ) 
 \backslash 
 \bigcup_{\zeta \in \mathcal{S}_1}
 \mathcal{D}_{\zeta}$. 
\end{defn}
\begin{prop}\label{Z11}
The reduction of the space 
$\mathbf{Z} _{1,1}$ 
is isomorphic to the affine curve defined by 
$Z^q +x_3 ^{q^2-1}+x_3 ^{-(q^2-1)} =0$ with $x_3 
\notin \mathcal{S}_1$. 
\end{prop}
\begin{proof}
This follows from Proposition \ref{wee}. 
\end{proof}

\subsection{Analysis of the singular residue classes of $\mathbf{Z}^0_{1,1}$}\label{sec3}
In this subsection, we analyze 
the singular residue classes 
$\{ \mathcal{D}_{\zeta} \}_{\zeta \in \mathcal{S}_1}$ 
of $\mathbf{Z}^0_{1,1}$. 
If $q$ is odd, 
the space 
$\mathcal{D}_{\zeta}$ 
is a basic 
wide open space with an underlying 
affinoid 
$\mathbf{X}_{\zeta}$, whose reduction 
$\overline{\mathbf{X}}_{\zeta}$ 
is isomorphic to the affine curve defined by $z^q-z=w^2$. 
On the other hand, if 
$q$ is even, 
the situation 
is slightly complicated, 
because 
the space $\mathcal{D}_{\zeta}$ 
is not basic wide open. 
Hence, we have to 
cover $\mathcal{D}_{\zeta}$ 
by smaller basic wide open spaces. 
As a result, 
in $\mathcal{D}_{\zeta}$, 
we find an affinoid 
$\mathbf{P}^0_{\zeta}$, whose reduction is 
isomorphic to the affine curve defined by 
$z_{f+1}^2 =w_1 (w_1 ^{q-1} -1)^2$. 
This affine curve has $q-1$ singular points at 
$w_1 \in k^{\times}$. 
Then, by analyzing 
the tubular neighborhoods 
of these singular points, 
we find an affinoid 
$\mathbf{X}_{\zeta,\zeta'} \subset \mathbf{P}^0_{\zeta}$ 
for each $\zeta' \in k^{\times}$, 
whose reduction is isomorphic to the affine curve defined by
$z^2 +z=w^3$. 

\subsubsection{$q$\ :\ odd}\label{qodd}
We assume that $q$ is odd. 
For each 
$\zeta \in \mu_{2(q^2-1)}(k^{\mathrm{ac}})$, 
we define an affinoid 
$\mathbf{X}_{\zeta} 
\subset 
\mathcal{D}_{\zeta}$ and 
compute its reduction 
$\overline{\mathbf{X}}_{\zeta}$. 

For $\iota \in \mu_2 (k^{\mathrm{ac}})$, 
we choose an element 
$c'_{1,\iota} \in \mathcal{O}_{K^{\mathrm{ac}}} ^{\times}$ 
such that $\bar{c}'_{1,\iota} =-2\iota$ and 
$c'^{2q} _{1,\iota} =4(1-\gamma_4 c'_{1,\iota})$. 
We take $\zeta \in \mu_{2(q^2-1)}(k^{\mathrm{ac}})$. 
We put $c_{1,\zeta} =c'_{1, \zeta^{q^2 -1}}$, 
and define 
$c_{2,\zeta} \in \mathcal{O}_{K^{\mathrm{ac}}} ^{\times}$ 
by $c_{2,\zeta} ^{q^2 -1} =-2c_{1,\zeta} ^{-q}$ 
and $\bar{c}_{2,\zeta} =\zeta$. 
We put 
\[
 a_{\zeta} =\omega_4^{q-1}c_{2,\zeta}^{q+1}, \quad 
 b_{\zeta} =-2\zeta^{q^2 -1} \omega_3 ^{\frac{q-1}{2}} 
 c_{1,\zeta} ^{-q} c_{2,\zeta}^{\frac{q+3}{2}}. 
\]
Note that we have $v(a_{\zeta} )=1/(2q^4)$ 
and $v(b_{\zeta} )=1/(4q^3)$. 

For an element 
$\zeta \in \mu_{2(q^2-1)}(k^{\mathrm{ac}})$, 
we define an affinoid 
$\mathbf{X}_{\zeta}$ by 
$v(x_3-c_{2,\zeta} ) \geq 1/(4q^3)$. 
We change variables as 
\[
 Z=a_{\zeta} z +c_{1,\zeta}, 
 \quad 
 x_3=b_{\zeta} w +c_{2,\zeta}. 
\]
Then, we acquire 
\[
 a_{\zeta} ^q (z^q -z-w^2 ) \equiv 0 \pmod{\epsilon_3 +} 
\]
by \eqref{41}. 
Dividing this by $a_{\zeta} ^q$, 
we have 
$z^q-z=w^2 \pmod{0+}$. 
Hence, the 
reduction of 
$\mathbf{X}_{\zeta}$
is isomorphic to the affine curve defined 
by $z^q-z=w^2$. 
\begin{prop}\label{Xodd}
For each 
$\zeta \in \mu_{2(q^2-1)}(k^{\mathrm{ac}})$, 
the reduction $\overline{\mathbf{X}}_{\zeta}$ 
is isomorphic to the affine curve defined by $z^q-z=w^2$ and 
the complement 
$\mathcal{D}_{\zeta} \setminus \mathbf{X}_{\zeta}$ 
is an open annulus. 
\end{prop}
\begin{proof}
We have already proved the first assertion. 
We prove the second assertion. 
We change variables as 
\[
 Z=z' +c_{1,\zeta}, 
 \quad 
 x_3=w' +c_{2,\zeta}
\] 
with 
$0<v(w')<1/(4q^3)$. 
Substituting them to 
\eqref{41}, we obtain 
\[
 z'^q \equiv w'^2\ \pmod{2v(w')+}. 
\]
Note that we have $0<v(z')<1/(2q^4)$. 
By setting $w'=z'' z'^{(q-1)/2}$, we acquire
\[
 z''^2 \equiv z' \pmod{v(z')+}. 
\]
Hence, we can see that 
$z'$ is written as a function of $z''$. 
Then $w'$ is also 
written as a function of $z''$. 
Therefore, 
$(\mathcal{D}_{\zeta} \setminus \mathbf{X}_{\zeta})
(\mathbf{C})$ is identified with 
$\{z'' \in \mathbf{C} \mid 0<v(z'')<1/(4q^4)\}$. 
\end{proof}

\subsubsection{$q$\ :\ even}\label{qeven}
We assume that 
$q$ is even. 
We put 
\[
 Z_1 =x_3^{q^2-1}. 
\]
Then, the congruence 
\eqref{41}
has the following form: 
\begin{equation}\label{zz1}
 Z^q+Z_1 (1-\gamma_4Z)+Z_1 ^{-1} \equiv 
 -\gamma_3 ^{q^2-q-1} Z_1 ^{-q(q+1)} 
  \pmod{\epsilon_3 +}. 
\end{equation}

\noindent \textbf{1.\ Projective lines}\ \ 
For each $\zeta \in k_2^{\times}$, 
we define subaffinoid 
$\mathbf{P}^0_{\zeta} \subset \mathcal{D}_{\zeta}$ by 
$v(Z) \geq 1/(4q^4)$. 
We change variables as 
\[
 Z=\varpi^{\frac{1}{4q^4}} w_1, 
 \quad 
 Z_1=1+\varpi^{\frac{1}{8q^3}} z_1. 
\]
Substituting these to \eqref{zz1} 
and dividing 
it by $\varpi^{1/(4q^3)}$, 
we acquire 
\begin{equation}\label{p1}
 (z_1 +w_1 ^{\frac{q}{2}})^2+\varpi^{\frac{1}{8q^3}} z_1^3 
 +\varpi^{\frac{1}{4q^3}} z_1^4 
 +\varpi^{\frac{q-1}{4q^4}} w_1+\varpi^{\frac{3q-2}{8q^4}} z_1w_1
 \equiv \varpi^{\frac{2q-3}{4q^3}} 
 \pmod{\frac{1}{4q^3} +}. 
\end{equation}
We can check that $v(z_1) \geq 0$. 
We set $q=2^f$ and 
put 
\[
 l_i =\frac{(2^i-1)q}{2^i}, \quad 
 m_i =\frac{1}{2^{i+2} q^3} 
\]
for $1 \leq i \leq f+1$. 
Furthermore, we define parameters 
$z_i$ for $2 \leq i \leq f+1$ by 
\begin{equation}\label{aqw11}
 z_i + w_1 ^{l_i} 
 =\varpi^{m_{i+1}} z_{i+1} \quad \textrm{for $1 \leq i \leq f$}. 
\end{equation}
\begin{lem}\label{eq2}
We assume that $v(Z) \geq 1/(4q^4)$. 
Then we have 
\begin{equation}\label{pred}
 z_{f+1} ^2 + w_1 ^{2q-1} +w_1 +\varpi^{\frac{1}{8q^4}} z_{f+1} w_1^q  
 \equiv 
 (q/2) \varpi^{\frac{1}{4q^4}}  \pmod{\frac{1}{4q^4} +}. 
\end{equation}
\end{lem}
\begin{proof}
If $q=2$, 
we can check that 
\begin{equation}\label{f1}
z_2^2+w_1^3+w_1+\varpi^{\frac{1}{128}} z_2 w_1^2 \equiv \varpi^{\frac{1}{64}}
(w_1 z_2^2+z_1^4+w_1^2+1) \pmod{\frac{1}{64}+} 
\end{equation}
by 
\[
 z_1=-w_1+\varpi^{\frac{1}{128}} z_2. 
\]
We have $v(z_2^2+w_1^3+w_1)>0$. 
Therefore, we obtain 
\[
 w_1 z_2^2+z_1^4+w_1^2 \equiv 
 w_1(z_2^2+w_1^3+w_1) \equiv 0 \pmod{0+}. 
\]
Hence, the required assertion in this case 
follows from \eqref{f1}.  
Assume that $f \geq 2$. 
For $1 \leq i \leq f+1$, we put 
\[
 n_i =\frac{q-2^{i-1}}{2^{i+1} q^4}. 
\]
We prove 
\begin{equation}\label{p2}
 (z_i +w_1 ^{l_i})^2+\varpi^{m_i} z_i w_1^q +\varpi^{n_i} w_1 
 \equiv 0
 \pmod{\frac{1}{2^{i+1} q^3} +} 
\end{equation}
for $2 \leq i \leq f+1$ by induction on $i$. 
Eliminating $z_1$ from \eqref{p1} by \eqref{aqw11} 
and dividing it by $\varpi^{1/(8q^4)}$, 
we obtain
\[
 (z_2 +w_1^{\frac{3q}{4}} )^2 +\varpi^{\frac{1}{16q^3}} z_2 w_1 ^q 
 +\varpi^{\frac{q-2}{8q^4}} w_1 
 +\varpi^{\frac{1}{8q^3}} w_1 ^{\frac{q}{2}} 
 (z_2 +w_1^{\frac{3q}{4}} )^2 \equiv 
0 \pmod{\frac{1}{8q^3} +}. 
\]
This shows 
\[
 v \Bigl( z_2 +w_1^{\frac{3q}{4}} \Bigr) \geq \frac{1}{32q^3}. 
\]
Hence we have \eqref{p2} for $i=2$. 
Assuming \eqref{p2} for $i$. 
Eliminating $z_i$ from \eqref{p2} by \eqref{aqw11} 
and dividing it by $\varpi^{m_i}$, 
we obtain \eqref{p2} for $i+1$. 
Hence, we have \eqref{p2} for $f+1$, 
which is equivalent to \eqref{pred}. 
\end{proof}
\begin{prop}\label{Peven}
For each $\zeta \in k_2^{\times}$, 
the reduction 
$\overline{\mathbf{P}}^0_{\zeta}$ 
is isomorphic to the affine curve defined by 
$z_{f+1}^2 =w_1 (w_1 ^{q-1} -1)^2$, 
which has genus $0$ and singularities at 
$w_1 \in k^{\times}$, 
and the complement 
$\mathcal{D}_{\zeta} 
\setminus \mathbf{P}_{\zeta}^0$ 
is an open annulus. 
\end{prop}
\begin{proof}
The claim on $\overline{\mathbf{P}}^0_{\zeta}$ 
follows from 
the congruence \eqref{pred} modulo $0+$. 
We prove the last assertion. 
We change a variable as $Z_1=1+z_1'$ with 
$0<v(z_1')<1/(8q^3)$. 
Similarly as \eqref{aqw11}, 
we introduce parameters 
$\{z_i'\}_{2 \leq i \leq f+1}$ by 
$z_i'+Z^{l_i}=z_{i+1}'$ for $1 \leq i \leq f$. 
Then, by similar computations to 
those in the proof of Lemma \ref{eq2}, 
we obtain 
\[
 z_{f+1}'^2 \equiv Z^{2q-1} \pmod{2v(z_{f+1}')+}. 
\]
By setting $z_{f+2}'=Z^q/z_{f+1}'$, we obtain 
\[
 z_{f+2}'^2 \equiv Z \pmod{v(Z)+}. 
\]
Then we can see that 
all parameters $z_i'$ for $1 \leq i \leq f+1$ 
and $Z$ are written as functions of $z_{f+2}'$. 
Hence, 
$(\mathcal{D}_{\zeta} \setminus \mathbf{P}_{\zeta}^0)(\mathbf{C})$ is identified with 
$\{z_{f+2}' \in \mathbf{C} \mid 0<v(z_{f+2}')<1/(8q^4)\}$. 
\end{proof}

\noindent \textbf{2.\ Elliptic curves}\ \ 
For $\zeta' \in k^{\times}$, 
we choose 
$c_{2,\zeta'} \in \mathcal{O}^{\times}_{\mathbf{C}}$ 
such that $\bar{c}_{2,\zeta'} =\zeta'$ and 
\[
 c_{2,\zeta'} ^{4(q-1)} +1+\varpi^{\frac{1}{4q^4}} c_{2,\zeta'} ^{4q-3} =0, 
\]
and a square root $c_{2,\zeta'} ^{1/2}$ of $c_{2,\zeta'}$. 
Further, we choose $c_{1,\zeta'}$ such that 
\[
 c_{1,\zeta'} ^2 +\varpi^{\frac{1}{8q^4}} c_{2,\zeta'} ^q c_{1,\zeta'} 
 + c_{2,\zeta'} (c_{2,\zeta'} ^{2(q-1)} +1) 
 = \frac{q}{2} \varpi^{\frac{1}{4q^4}}, 
\]
and $b_{2,\zeta'}$ such that 
$b_{2,\zeta'} ^3 =\varpi^{1/(4q^4)} c_{2,\zeta'} ^4$. 
We put 
\[
 a_{1,\zeta'} =\varpi^{\frac{1}{8q^4}} c_{2,\zeta'} ^q, 
 \quad 
 b_{1,\zeta'}=c_{2,\zeta'} ^{\frac{2q-3}{2}} b_{2,\zeta'}. 
\]
For each $\zeta' \in k^{\times}$, 
we define a subspace 
$\mathcal{D}_{\zeta,\zeta'} \subset \mathbf{P}^0_{\zeta}$
by 
$v(w_1 -c_{2,\zeta'})>0$. 
Furthermore, 
we define 
$\mathbf{X}_{\zeta,\zeta'} \subset \mathcal{D}_{\zeta,\zeta'}$ 
by $v(w_1 -c_{2,\zeta'}) \geq 1/(12q^4)$. 
We put 
\[
 \mathbf{P}_{\zeta}=\mathbf{P}^0_{\zeta} \backslash 
 \bigcup_{\zeta' \in k^{\times}} \mathcal{D}_{\zeta,\zeta'}. 
\]

We take 
$(\zeta,\zeta') \in k_2^{\times} \times k^{\times}$ 
and compute the reduction of 
$\mathbf{X}_{\zeta,\zeta'}$. 
In the sequel, 
we omit the subscript $\zeta'$ 
of $a_{1,\zeta'}$, $b_{1,\zeta'}$, $b_{2,\zeta'}$
$c_{1,\zeta'}$ and $c_{2,\zeta'}$, 
if there is no confusion. 
We change variables as 
\[
 z_{f+1}=a_1 z+b_1 w +c_1, \quad 
 w_1 =b_2 w +c_2. 
\]
By substituting these to \eqref{pred}, 
we acquire 
\begin{equation}\label{ell}
 a_1 ^2 (z^2 +z +w^3) \equiv 0 \pmod{\frac{1}{4q^4} +} 
\end{equation}
by the definition of $a_1$, $b_1$, $b_2$, $c_1$ and $c_2$. 

\begin{prop}\label{Xeven}
For each 
$(\zeta,\zeta') \in k_2^{\times} \times k^{\times}$, 
the reduction of $\mathbf{X}_{\zeta,\zeta'}$ 
is isomorphic to the affine curve defined by $z^2+z=w^3$ 
and the complement 
$\mathcal{D}_{\zeta,\zeta'} \setminus \mathbf{X}_{\zeta,\zeta'}$ 
is an open annulus. 
\end{prop}
\begin{proof}
The first assertion follows from \eqref{ell}. 
We prove the second assertion. 
We change variables as 
\[
 z_{f+1} =z' +c_2^{\frac{2q-3}{2}}w'+c_1, 
 \quad 
 w_1=w' +c_2
\] 
with 
$0<v(w')<1/(12q^4)$. 
Substituting them 
to \eqref{pred}, we acquire 
\[
 z'^2 \equiv c_2^{2(q-2)}w'^3 \pmod{2v(z')+} 
\]
by the choice of $c_2$. 
Note that we have 
\[
 v(z')=3v(w')/2 <\frac{1}{8q^4}. 
\]
By setting $z''=z'/(c_2^{q-2}w')$, 
we obtain 
\[
 z''^2 \equiv w' \pmod{v(w')+}. 
\]
Then we can see that 
$z'$ and $w'$ are written as functions of $z''$. 
Hence, 
$(\mathcal{D}_{\zeta,\zeta'} \setminus \mathbf{X}_{\zeta,\zeta'})(\mathbf{C})$ 
is identified with 
$\{z'' \in \mathbf{C} \mid 0 <v(z'')<1/(24q^4)\}$. 
\end{proof}

\subsection{Stable covering of $\mathbf{X}_1(\mathfrak{p}^3)$}
In this subsection, we show the existence of 
the stable covering of $\mathbf{X}_1(\mathfrak{p}^3)$ 
over some finite extension of the base field 
$\widehat{K}^{\mathrm{ur}}$. 
See 
\cite[Section 2.3]{CMp^3} 
for the notion of semi-stable coverings. 
A semi-stable covering is called stable, 
if the corresponding semi-stable model is stable. 
\begin{prop}\label{exssc}
There exists a stable covering of 
$\mathbf{X}_1(\mathfrak{p}^3)$ over a finite extension of 
the base field. 
\end{prop}
\begin{proof}
First, we show that, 
after taking a finite extension of the base field, 
$\mathbf{X}_1(\mathfrak{p}^3)$ is a wide open space. 
By \cite[Theorem 2.3.1 (i)]{Strdef1}, 
$\mathbf{X}_1(\mathfrak{p}^3)$ is the Raynaud generic fiber of 
the formal completion of an affine scheme over 
$\mathcal{O}_{\widehat{K}^{\mathrm{ur}}}$ 
at a closed point on the special fiber. 
Then we can apply \cite[Theorem 2.29]{CMp^3} 
to the formal completion of the affine scheme along its 
special fiber, after shrinking the affine scheme. 
Hence, 
$\mathbf{X}_1(\mathfrak{p}^3)$ is a wide open space over 
some extension. 

By \cite[Theorem 2.18]{CMp^3}, 
a wide open space can be embedded to a 
proper algebraic curve so that 
its complement is a disjoint union of closed disks. 
Therefore, $\mathbf{X}_1(\mathfrak{p}^3)$ 
has a semi-stable covering over some finite extension 
by \cite[Theorem 2.40]{CMp^3}. 
Then a simple modification gives a stable covering. 
\end{proof}

In the following, 
we construct a candidate of a semi-stable covering of 
$\mathbf{X}_1(\mathfrak{p}^3)$ over some finite extension. 
We put 
\[
 \mathbf{V}_1
 =\mathbf{W}^{+}_{1,1'} \cup 
 \bigcup_{2 \leq i \leq 6}\mathbf{W}_{i,1'}, \quad 
 \mathbf{V}_2 =\mathbf{W}^{-}_{1,1'}
 \cup \bigcup_{2 \leq i \leq 4}\mathbf{W}_{1,i'}, \quad 
 \mathbf{U}
 =\mathbf{W}_{1,1'} \backslash 
 \bigcup_{\zeta \in \mathcal{S}_1}\mathbf{X}_{\zeta}. 
\]
We note that 
$\mathbf{V}_1 \supset \mathbf{Y}_{1,2}$, 
$\mathbf{V}_2 \supset \mathbf{Y}_{2,1}$, 
$\mathbf{U} \supset \mathbf{Z}_{1,1}$, 
$\mathbf{V}_1 \cap \mathbf{V}_2=\emptyset$, 
$\mathbf{V}_1 \cap \mathbf{U}=\mathbf{W}^{+}_{1,1'}$ and 
$\mathbf{V}_2 \cap \mathbf{U}=\mathbf{W}^{-}_{1,1'}$. 

We consider the case 
where $q$ is even in this paragraph. 
We set 
$\hat{\mathcal{D}}_{\zeta} =\mathcal{D}_{\zeta} 
 \backslash \bigl(\bigcup_{\zeta' \in k^{\times}} 
 \mathbf{X}_{\zeta,\zeta'}\bigr)$ 
for $\zeta \in k_2^{\times}$. 
Then, 
$\hat{\mathcal{D}}_{\zeta}$ contains
$\mathbf{P}_{\zeta}$ 
as the underlying affinoid. 
On the other hand, for 
$(\zeta, \zeta') \in k_2^{\times} \times k^{\times}$ 
the space $\mathcal{D}_{\zeta, \zeta'}$ 
has the underlying affinoid 
$\mathbf{X}_{\zeta, \zeta'}$. 

We put 
\[
 \mathcal{S}=
 \begin{cases}
 \mathcal{S}_1 &\quad 
 \textrm{if $q$ is odd,}\\
 k_2^{\times} 
 \times k^{\times} 
 &\quad \textrm{if $q$ is even.} 
 \end{cases}
\] 
Now, 
we define an admissible covering of 
$\mathbf{X}_1(\mathfrak{p}^3)$ as 
\[
 \mathcal{C}_1(\mathfrak{p}^3)=
 \begin{cases}
 \{\mathbf{V}_1,\mathbf{V}_2,\mathbf{U},
 \{\mathcal{D}_{\zeta}\}_{\zeta \in \mathcal{S}_1}\} 
 &\quad  \textrm{if $q$ is odd,} \\
 \{\mathbf{V}_1,\mathbf{V}_2,\mathbf{U},
 \{\hat{\mathcal{D}}_{\zeta}\}_{\zeta \in k_2^{\times}}, 
 \{\mathcal{D}_{\zeta, \zeta'} \}_{(\zeta, \zeta') 
 \in \mathcal{S}} 
 &\quad \textrm{if $q$ is even.} 
\end{cases}
\]
In Subsection \ref{gcal}, 
we will show that $\mathcal{C}_1(\mathfrak{p}^3)$ 
is a semi-stable covering of 
$\mathbf{X}_1(\mathfrak{p}^3)$ over some finite extension. 

%\begin{proof}
%This follows from 
%Proposition \ref{ppro3}, 
%\ref{Y12}, \ref{Y21}, \ref{Z11}, \ref{Xodd}, \ref{Peven} and \ref{Xeven}. 
%\end{proof}

\section{Action of the division algebra on the reductions}\label{actdiv}
In this section, we determine the action of 
of $\mathcal{O}_D^{\times}$ 
on the reductions 
$\overline{\mathbf{Y}}_{1,2}$, $\overline{\mathbf{Y}}_{2,1}$ 
$\overline{\mathbf{Z}}_{1,1}$, 
$\{\overline{\mathbf{P}}_{\zeta}\}_{\zeta \in k_2^{\times}}$ and 
$\{\overline{\mathbf{X}}_\zeta\}_{\zeta \in \mathcal{S}}$ 
by using the description of 
$\mathcal{O}^{\times}_D$-action in \eqref{g_2}. 
We take 
\[
 d=d_1 +\varphi d_2 \in \mathcal{O}_D^{\times}, 
\]
where $d_1 \in \mathcal{O}^{\times}_{K_2}$ and 
$d_2 \in \mathcal{O}_{K_2}$. 
We put 
\[
 \kappa_1 (d)=\bar{d}_1, 
 \quad 
 \kappa_2 (d)=-\bar{d}_1^{-q} \bar{d}_2 . 
\]

\begin{lem}\label{dYact}
The element $d$ 
induces the following morphisms:
\begin{align*}
 &\overline{\mathbf{Y}}_{1,2}
 \to
 \overline{\mathbf{Y}}_{1,2};\ 
 (x,y) \mapsto 
 (\kappa_1 (d) x, 
 \kappa_1 (d)^{-q} y),\\ 
 &\overline{\mathbf{Y}}_{2,1}
 \to
 \overline{\mathbf{Y}}_{2,1};\ 
 (x,y) \mapsto 
 (\kappa_1 (d)^{-1} x, \kappa_1 (d)^q y). 
\end{align*}
\end{lem}
\begin{proof}
We prove the assertion for 
$\overline{\mathbf{Y}}_{1,2}$. 
By \eqref{g_1}, we have 
\[
 d^{\ast} x_1 \equiv d_1 x_1, \quad 
 d^{\ast} x_3 \equiv d_1 x_3 \pmod{0+}. 
\]
Therefore, the required 
assertion follows from \eqref{gk1}. 
The assertion for 
$\overline{\mathbf{Y}}_{1,2}$ 
is proved similarly. 
\end{proof}

Now, let the notation be as in 
Subsection \ref{sec2}. 
We put 
\begin{align*}
 x'_i &=d^{\ast} x_i \quad \textrm{for} \quad 1 \leq i \leq 3, \\ 
 t' &=d^{\ast} t, \quad 
 Z_0'=d^{\ast} Z_0, \quad Z'=d^{\ast} Z. 
\end{align*}
We have 
\begin{align*}
 j^{-1}(x_1) &\equiv 
 x_1+d_1^{-q} d_2 \varpi^{\epsilon_1}\tilde{u} x_1 
 \pmod{\epsilon_1 +},\\ 
 j^{-1}(x_2) &\equiv 
 x_2+d_1^{-q} d_2 \varpi^{\epsilon_1} \tilde{u} x_2
 \pmod{\epsilon_1 +}, \\ 
 j^{-1}(x_3) &\equiv x_3 \pmod{\epsilon_2 +} 
\end{align*}
by \eqref{gh2}. 
On the other hand, we have 
\begin{align*}
 \tilde{d}(x_1) &\equiv d_1 x_1 \pmod{\epsilon_1 +},\\ 
 \tilde{d}(x_2) &\equiv 
 d_1 x_2 + d_2^q \varpi^{\epsilon_1} x_2^q 
 \pmod{\epsilon_1 +}, \\ 
 \tilde{d} (x_3) &\equiv 
 d_1 x_3 +d_2^q \varpi^{\epsilon_3}  x_3^q \pmod{\epsilon_2 +} 
\end{align*}
by \eqref{g_1}. 
Hence, we obtain 
\begin{align}
 x'_1 &\equiv d_1 x_1 
 +d_1^{-(q-1)} d_2 \varpi^{\epsilon_1} \tilde{u} x_1 
 \pmod{\epsilon_1 +}, \label{b1}\\ 
 x'_2 &\equiv d_1 x_2 
 +d_1^{-(q-1)} d_2 \varpi^{\epsilon_1} \tilde{u} x_2 
 +d_2^q \varpi^{\epsilon_1} x_2^q 
 \pmod{\epsilon_1 +}, \label{b2}\\ 
 x'_3 &\equiv d_1 x_3 
 +d_2^q \varpi^{\epsilon_3} x_3^q
 \pmod{\epsilon_2 +}. \label{b3}
\end{align}
By 
the definition of $t$ and the equation 
$x'_2/x'_1=-1+\gamma_2(x_2 '^q/t')$, 
we acquire 
\begin{equation}\label{b4}
 t' \equiv d_1^q t
 -d_1^{q-1} d_2^q t^{2-q} 
 \varpi^{\epsilon_2} 
 \pmod{\epsilon_2 +}
\end{equation}
using \eqref{b1} and \eqref{b2}. 
We put 
\[
 G_0=d_1^{-q} d_2 x_3^{q(q-1)}+d_1^{-1} d_2^q x_3^{-q(q-1)}. 
\]
By 
the definition of $Z_0$ and the equation 
$1+(x_3'^q/t')=\gamma_3 Z_0'$, 
we obtain 
\begin{equation}\label{b5}
 Z_0' \equiv Z_0 -\varpi^{\epsilon_3} G_0 \pmod{\epsilon_3+} 
\end{equation}
using \eqref{b3} and \eqref{b4}. 
We put 
\[
 G =G_0 +d_1^{-1} d_2^q (x_2x_3^{q-2}+x_1^{-1} x_3^q). 
\]
By 
the definition of $Z$ and the equation 
$Z_0'+(x_2'/x_3')-(x_3'/x_1')=\gamma_4 (x_2'/x_3') Z'$, 
we obtain 
\begin{equation}\label{b6}
 Z' \equiv Z -\frac{x_3}{x_2} \varpi^{\epsilon_4} G 
 \pmod{\epsilon_4 +} 
\end{equation}
using \eqref{b1}, \eqref{b2}, \eqref{b3} and \eqref{b5}. 
We have 
\[
 G \equiv d_1^{-q} d_2 x_3^{q(q-1)} +
 d_1^{-1} d_2^q x_3^{(q-1)(q+2)} \pmod{0+} 
\]
by 
$x_1 \equiv -x_3^{q^2},\ x_2 \equiv x_3^{q^2} \pmod{0+}$. 
We put 
\[
 \Delta=d_1^{-q} d_2 x_3^{-(q-1)} +d_1^{-1} d_2^q x_3^{q-1}. 
\]
Then the congruence \eqref{b6} has the following form: 
\begin{equation}\label{b7}
 Z' \equiv Z -\varpi^{\epsilon_4} \Delta 
 \pmod{\epsilon_4 +}. 
\end{equation}
\begin{prop}\label{3w}
The element $d$ acts on $\overline{\mathbf{Z}}_{1,1}$
by $(Z, x_3) \mapsto (Z, \kappa_1 (d) x_3)$. 
\end{prop}
\begin{proof}
This follows from \eqref{b3} and \eqref{b7}. 
\end{proof}
\begin{prop}
The element $d$ induces the morphism 
\[
 \overline{\mathbf{P}}_{\zeta} 
 \to \overline{\mathbf{P}}_{\kappa_1 (d) \zeta};\ 
 w_1 \mapsto w_1. 
\]
\end{prop}
\begin{proof}
This follows from \eqref{b7}, Proposition \ref{3w} and 
$Z=\varpi^{1/(4q^4)} w_1$. 
\end{proof}
\begin{prop}\label{dXact}
We take $\zeta \in \mathcal{S}_1$. 
Further, we take $\zeta' \in k^{\times}$, 
if $q$ is even. 
We set as follows: 
\begin{align*}
 \eta &=
 \begin{cases}
 \zeta  &\quad  \textrm{if $q$ is odd,}\\
 (\zeta, \zeta') 
 &\quad \textrm{if $q$ is even,} 
 \end{cases} \\ 
 d \eta &=
 \begin{cases}
 \kappa_1 (d) \zeta 
 &\quad \textrm{if $q$ is odd,}\\ 
 (\kappa_1 (d)\zeta, \zeta') 
 &\quad \textrm{if $q$ is even,} 
 \end{cases}\\
 f_d&=
 \begin{cases}
 \Tr_{k_2/k}
 ( \zeta^{-2q} \kappa_2 (d) ) 
  &\quad \textrm{if $q$ is odd,}\\
 \Tr_{k_2/\mathbb{F}_2}
 ( \zeta^{1-q}  \zeta'^{-2} \kappa_2 (d)) 
 &\quad \textrm{if $q$ is even,} 
 \end{cases}
\end{align*}
where $\eta, d \eta \in \mathcal{S}$. 
Then, the element $d$ induces 
\[
 \overline{\mathbf{X}}_{\eta} \to
 \overline{\mathbf{X}}_{d \eta} \colon 
 \begin{cases}
 (z,w) \mapsto (\kappa_1 (d)^{-(q+1)}(z+f_d),
 \kappa_1 (d)^{-(q+1)/2} w) 
 &\quad  \textrm{if $q$ is odd,}\\
 (z,w) \mapsto (z+f_d, w) &\quad \textrm{if $q$ is even.}
 \end{cases}
\]
\end{prop}
\begin{proof}
First, we assume that $q$ is odd. 
Recall that $Z=a_{\zeta} z+c_{1,\zeta}$ 
and $x_3=b_{\zeta} w+c_{2,\zeta}$. 
Similarly, we have 
$Z'=a_{\bar{d}_1 \zeta} z'+c_{1,\bar{d}_1 \zeta}$ 
and 
$x_3=b_{\bar{d}_1 \zeta} w'+c_{2,\bar{d}_1 \zeta}$. 
Then, the claim follows from \eqref{b7}. 

Next, we assume that $q$ is even. 
By \eqref{b7} and 
$d^{\ast} x_3 \equiv d_1 x_3 \pmod{(\epsilon_3 /2)+}$, 
we acquire 
\begin{equation}\label{bbq1}
 d^{\ast} z_{f+1} -z_{f+1} \equiv
 \varpi^{\frac{\epsilon_4}{4}} 
 \sum_{i=1}^{f} w_1^{q-2^i} \Delta^{2^{i-1}} 
 \pmod{\frac{\epsilon_4}{4}+} 
\end{equation}
on the locus where $v(Z) \geq \epsilon_4 /2$. 
By $z_{f+1} =a_{1,\zeta'} z +b_{1,\zeta'} w +c_{1,\zeta'}$ and 
$w_1=b_{2,\zeta'} w +c_{2,\zeta'}$, 
we obtain 
\begin{align*}
 d^{\ast} z-z &\equiv 
 \sum_{i=1}^{f} c_{2,\zeta'}^{-2^i} \Delta^{2^{i-1}} \pmod{0+}, \\ 
 d^{\ast} w &\equiv w \pmod{\frac{\epsilon_4}{3}+} 
\end{align*}
on $\mathbf{X}_{\zeta, \zeta'}$ 
by \eqref{b7} and \eqref{bbq1}. 
On the other hand, we have 
\[
 \sum_{i=1}^{f} \bar{c}_{2,\zeta'}^{-2^i} \overline{\Delta}^{2^{i-1}} =f_d, 
\]
because $\bar{x_3}=\zeta$ and $\bar{c}_{2,\zeta'} =\zeta'$. 
Hence, we have proved the claim. 
\end{proof}

\section{Action of the Weil group on the reductions}\label{actine}
In this section, we compute 
the actions of the Weil group 
on the reductions 
$\overline{\mathbf{Y}}_{1,2}$, 
$\overline{\mathbf{Y}}_{2,1}$, 
$\overline{\mathbf{Z}}_{1,1}$, 
$\{\overline{\mathbf{P}}_{\zeta}\}_{\zeta \in k_2^{\times}}$ and 
$\{\overline{\mathbf{X}}_{\eta}\}_{\eta \in \mathcal{S}}$. 

Let $\mathbf{X}$ be 
a reduced affinoid over $\mathbf{C}$ 
with an action of $W_K$. 
For $P \in \mathbf{X}(\mathbf{C})$, 
the image of $P$ under the natural reduction map 
$\mathbf{X}(\mathbf{C}) \to \overline{\mathbf{X}}(k^{\mathrm{ac}})$ 
is denoted by $\overline{P}$. 
The action of 
$W_K$ 
on $\overline{\mathbf{X}}$ 
is a homomorphism 
\[
 w_{\mathbf{X}} \colon W_K \to \Aut (\overline{\mathbf{X}}) 
\]
characterized by 
$\overline{\sigma(P)}=w_{\mathbf{X}} (\sigma) (\overline{P})$ 
for $\sigma \in W_K$ and 
$P \in \mathbf{X}(\mathbf{C})$. 
For $\sigma \in W_K$, we define 
$r_{\sigma} \in \mathbb{Z}$ so that 
$\sigma$ induces the $q^{-r_{\sigma}}$-th power map 
on the residue field of $K^{\mathrm{ac}}$. 

\begin{rem}
In the usual sense, 
$W_K$ does not act on $\mathbf{X}_1(\mathfrak{p}^3)$, 
because the action of $W_K$ does not 
preserve the connected components of 
$\mathrm{LT}_1(\mathfrak{p}^3)$. 
Precisely, 
$w_{\mathbf{X}}$ is the action of 
\[
 \{ (\sigma, \varphi^{-r_{\sigma}}) \in W_K \times D^{\times} \}, 
\]
which preserves the connected components of 
$\mathrm{LT}_1(\mathfrak{p}^3)$. 
\end{rem}

\subsection{Actions of the Weil group on $\overline{\mathbf{Y}}_{1,2}$, $\overline{\mathbf{Y}}_{2,1}$ and $\overline{\mathbf{Z}}_{1,1}$}
For $\sigma \in W_K$, 
we put 
\[
 \lambda (\sigma)=\overline{\sigma(\varpi^{1/(q^2-1)})/\varpi^{1/(q^2-1)}}
 \in k_2^{\times}. 
\]
We note that $\lambda$ is not a group homomorphism 
in general. 

\begin{lem}\label{iYact}
Let $\sigma \in W_K$. 
Then, the element 
$\sigma$ induces the automorphisms 
\begin{align*}
 &\overline{\mathbf{Y}}_{1,2}
 \to
 \overline{\mathbf{Y}}_{1,2};\ 
 (x,y) \mapsto (\lambda (\sigma)^q x^{q^{-r_{\sigma}}},
 \lambda (\sigma)^{-1}y^{q^{-r_{\sigma}}}),\\ 
 &\overline{\mathbf{Y}}_{2,1}
 \to 
 \overline{\mathbf{Y}}_{2,1};\ 
 (x,y) \mapsto 
 (\lambda (\sigma)^{-1}x^{q^{-r_{\sigma}}}, 
 \lambda (\sigma)^q y^{q^{-r_{\sigma}}}) 
\end{align*}
as schemes over $k$. 
\end{lem}
\begin{proof}
We prove the claim for 
$\overline{\mathbf{Y}}_{1,2}$. 
We set 
\[
 \sigma(\varpi^{\frac{1}{q^3(q^2 -1)}} )=
 \xi \varpi^{\frac{1}{q^3(q^2 -1)}} 
\]
with 
$\xi \in \mu_{q^3(q^2-1)} (K^{\mathrm{ac}} )$. 
Let 
$P \in \mathbf{Y}_{1,2}(\mathbf{C})$. 
We have $X_3(\sigma(P))=\sigma(X_3(P))$. 
By applying $\sigma$ to $X_3(P)=\varpi^{1/(q^3(q^2 -1))} x_3(P)$,
we obtain 
\[
 x_3(\sigma(P))=\xi 
 \sigma(x_3(P)) \equiv \xi x_3(P)^{q^{-r_{\sigma}}} \pmod{0+}. 
\]
In the same way, we have 
\[
 x_1(\sigma(P))
 \equiv \xi^{q^4}x_1(P)^{q^{-r_{\sigma}}} \pmod{0+}. 
\]
Therefore, we acquire 
$x^{\sigma}=\bar{\xi} x^{q^{-r_{\sigma}}}$
and 
$y^{\sigma} =\bar{\xi}^{-q}y^{q^{-r_{\sigma}}}$ 
by \eqref{gk1}. 
Hence, the claim follows from 
$\bar{\xi}=\lambda(\sigma)^q$. 
We can prove the claim for 
$\overline{\mathbf{Y}}_{2,1}$ similarly. 
\end{proof}
For $\sigma \in W_K$, 
we put 
\[
 \xi_{\sigma} =\frac{\sigma(\omega_3 )}{\omega_3} 
 \in \mu_{2q^3(q-1)}(K^{\mathrm{ac}}). 
\]
\begin{lem}\label{Z11ine}
Let $\sigma \in W_K$. 
Then, $\sigma$ acts on 
$\overline{\mathbf{Z}}_{1,1}$ by 
$(Z, x_3) \mapsto (Z^{q^{-r_{\sigma}}}, 
 \bar{\xi}_{\sigma} x_3 ^{q^{-r_{\sigma}}})$. 
\end{lem}
\begin{proof}
We use the notation in Subsection \ref{sec2}. 
Let $P \in \mathbf{Z}_{1,1} (\mathbf{C})$. 
Since we set $X_1=\omega_1 ^{2q-1}x_1$, 
$X_2=\omega_1 x_2$ and $X_3=\omega_3 x_3$, we have 
\begin{align*}
 x_1(\sigma(P))&=\xi_{\sigma}^{q^2(2q-1)}\sigma(x_1(P)), \\ 
 x_2(\sigma(P))&=\xi_{\sigma}^{q^2}\sigma(x_2(P)), \\ 
 x_3(\sigma(P))&=\xi_{\sigma} \sigma(x_3(P)). 
\end{align*}
Hence, we obtain 
\[
 \frac{x_2(\sigma(P))}{x_1(\sigma(P))} 
 =\xi_{\sigma}^{-2q^2(q-1)} 
 \sigma \biggl( \frac{x_2(P)}{x_1(P)} \biggr) \equiv
 \sigma \biggl( \frac{x_2(P)}{x_1(P)} \biggr) \pmod{\epsilon_1 +}. 
\]
Since we set $x_2/x_1=-1+\gamma_2 (x_2^q/t)$, 
we acquire 
\[
 t(\sigma(P)) \equiv \xi_{\sigma}^{q^3}\sigma(t(P)) \pmod{\epsilon_2 +}. 
\]
Therefore, we obtain 
\[
 \frac{x_3(\sigma(P))^q}{t({\sigma}(P))} =
 \xi_{\sigma}^{-q(q^2-1)} 
 \sigma \biggl( \frac{x_3(P)^q}{t(P)} \biggr) 
 \equiv \sigma \biggl( 
 \frac{x_3(P)^q}{t(P)} \biggr) \pmod{\epsilon_2 +}. 
\]
Since we set $1+(x_3^q/t)=\gamma_3 Z_0$, we obtain 
\[
 Z_0 ({\sigma}(P)) \equiv \sigma(Z_0 (P)) \pmod{\epsilon_3 +}. 
\]
Therefore we acquire 
\begin{equation}\label{Z2}
 Z(\sigma(P)) \equiv \sigma(Z(P)) \pmod{\epsilon_4 +} 
\end{equation}
by $Z_0 +(x_2/x_3)-(x_3/x_1)=\gamma_4(x_2/x_3)Z$. 

The assertion follows from 
\[
 x_3(\sigma(P))=\xi_{\sigma} \sigma(x_3(P)) 
 \equiv \xi_{\sigma} x_3(P)^{q^{-r_{\sigma}}} \pmod{0+} 
\]
and \eqref{Z2}. 
\end{proof}

\subsection{Action of the Weil group on $\overline{\mathbf{X}}_{\eta}$}
In this subsection, 
let $\zeta \in \mu_{2(q^2 -1)} (k^{\mathrm{ac}} )$. 
Until Lemma \ref{delta}, let $\sigma \in W_K$. 
\subsubsection{$q$\ :\ odd}
We assume that $q$ is odd. 
We use the notation in Paragraph \ref{qodd}. 
By \eqref{Z2} and $x_3(\sigma(P)) =\xi_{\sigma} \sigma(x_3(P))$, 
we have 
\begin{equation}\label{Zine}
\begin{split}
 a_{\bar{\xi}_{\sigma} \zeta^{q^{-r_{\sigma}}}} z(\sigma(P)) +
 c_{1,\bar{\xi}_{\sigma} \zeta^{q^{-r_{\sigma}}}} &=
 Z(\sigma(P)) \equiv \sigma(Z(P)) \\ 
 &= 
 \sigma (a_{\zeta}) \sigma (z(P)) +\sigma (c_{1,\zeta}) 
 \pmod{\epsilon_4 +} 
\end{split}
\end{equation}
and 
\begin{equation}\label{x3ine}
\begin{split}
 b_{\bar{\xi}_{\sigma} \zeta^{q^{-r_{\sigma}}}} w(\sigma(P)) 
 +c_{2,\bar{\xi}_{\sigma} \zeta^{q^{-r_{\sigma}}}} &=
 x_3(\sigma(P)) = \xi_{\sigma} \sigma(x_3 (P)) \\ 
 &= 
 \xi_{\sigma} \sigma (b_{\zeta}) \sigma (w(P)) 
 +\xi_{\sigma} \sigma (c_{2,\zeta}) 
\end{split}
\end{equation}
for $P \in \mathbf{X}_{\zeta}(\mathbf{C})$. 
Note that 
$c_{1,\bar{\xi}_{\sigma} \zeta^{q^{-r_{\sigma}}}} =c_{1,\zeta}$ 
and $c_{2,\bar{\xi}_{\sigma} \zeta^{q^{-r_{\sigma}}}} =
 \xi_{\sigma}^{q^4} \zeta^{q^{-r_{\sigma}} -1} c_{2, \zeta}$. 
We have 
\[
 v(\sigma(c_{1,\zeta})-c_{1,\zeta}) \geq \epsilon_4 
\] 
by \eqref{Zine}. 
We put 
\[
 a_{\sigma,\zeta} =
 \frac{\sigma(a_{\zeta} )}
 {\zeta^{r_{\sigma}(q^2-1)} \xi_{\sigma}^{q+1}a_{\zeta}}, \quad  
 b_{\sigma,\zeta} = 
 \frac{\sigma(c_{1,\zeta})-c_{1,\zeta}}
 {\zeta^{r_{\sigma}(q^2-1)} \xi_{\sigma}^{q+1} a_{\zeta}} , \quad 
 c_{\sigma,\zeta} =
 \frac{\sigma(b_{\zeta})}{\zeta^{(q^{-r_{\sigma}}-1)\frac{q+3}{2}} \xi_{\sigma}^{\frac{q+1}{2}} b_{\zeta}}. 
\]
Then we have 
$a_{\sigma,\zeta} ,b_{\sigma,\zeta} ,c_{\sigma,\zeta} 
 \in \mathcal{O}_{K^{\mathrm{ac}}}$. 
In the sequel, we omit the subscript 
$\zeta$ of $a_{\sigma,\zeta}$, 
$b_{\sigma,\zeta}$ and $c_{\sigma,\zeta}$. 
\begin{prop}\label{ine}
We have $\bar{a}_{\sigma} \in k^{\times}$, 
$\bar{b}_{\sigma} \in k$ and 
$\bar{a}_{\sigma} =\bar{c}_{\sigma} ^2$. 
Further, $\sigma$ induces the morphism 
\[
 \overline{\mathbf{X}}_{\zeta} \to
 \overline{\mathbf{X}}_{\bar{\xi}_{\sigma} \zeta^{q^{-r_{\sigma}}}} ;\ 
 (z,w) \mapsto
 (\bar{a}_{\sigma} z^{q^{-r_{\sigma}}} +\bar{b}_{\sigma}, 
 \bar{c}_{\sigma} w^{q^{-r_{\sigma}}}). 
\]
\end{prop}
\begin{proof}
We have 
\[
 v(\xi_{\sigma} \sigma (c_{2,\zeta})- 
 \xi_{\sigma}^{q^4} \zeta^{q^{-r_{\sigma}} -1} c_{2, \zeta}) 
 \geq \epsilon_3 
\]
by $v(\sigma(c_{1,\zeta})-c_{1,\zeta}) \geq \epsilon_4$. 
Hence we have the last assertion by \eqref{Zine} and \eqref{x3ine}. 
By the definition of $a_{\zeta}$, 
$b_{\zeta}$ and $c_{1,\zeta}$, 
we can check that 
\[
 \bar{a}_{\sigma} ^{q-1} =1, 
 \quad 
 \bar{b}_{\sigma} ^q=\bar{b}_{\sigma}, 
 \quad 
 \bar{a}_{\sigma} =\bar{c}_{\sigma} ^2 
\] 
using 
$c_{1,\zeta}^q \equiv -\iota (2-\gamma_4 c_{1,\zeta}) \pmod{(q-1)/q^4}$. 
\end{proof}

We put $L=K(\varpi^{1/2})$ and 
$L_2=K_2 (\varpi^{1/2})$ in $K^{\mathrm{ac}}$. 
Let $\mathrm{LT}_{L_2}$ be the formal 
$\mathcal{O}_{L_2}$-module over $\mathcal{O}_{L^{\mathrm{ur}}}$ 
of dimension $1$ such that 
\begin{align*}
 [\varpi^{\frac{1}{2}}]_{\mathrm{LT}_{L_2}}(X)
 &=\varpi^{\frac{1}{2}} X-X^{q^2}, \\ 
 [\zeta]_{\mathrm{LT}_{L_2}}(X)
 &=\zeta X \quad 
 \textrm{for} \quad \zeta \in \mu_{q^2 -1}(L_2) \cup \{ 0 \}. 
\end{align*} 
We put 
$\varpi_{1,L_2} =\varpi^{1/(2(q^2 -1))}$ and take 
$\varpi_{2,L_2} \in \mathcal{O}_{K^{\mathrm{ac}}}$ 
such that 
$[\varpi^{1/2}]_{\mathrm{LT}_{L_2}}(\varpi_{2,L_2})=\varpi_{1,L_2}$. 
Let $\mathrm{Art}_{L_2} \colon L_2^{\times} 
\stackrel{\sim}{\to} W^{\mathrm{ab}}_{L_2}$ 
be the Artin reciprocity 
map normalized so that 
the image by $\mathrm{Art}_{L_2}$ 
of a uniformizer is a lift of the geometric Frobenius. 
We consider the following 
homomorphism: 
\[
 I_{L_2} \to 
 k_2^{\times} \times k_2;\ 
 \sigma \mapsto 
 \bigl( \lambda_1 (\sigma),\lambda_2 (\sigma) \bigr)=
 \Biggl( \overline{\frac{\sigma(\varpi_{1,L_2})}{\varpi_{1,L_2}}}, 
 \overline{\frac{\varpi_{1,L_2}\sigma(\varpi_{2,L_2})
 -\sigma(\varpi_{1,L_2})\varpi_{2,L_2}}
 {\sigma(\varpi_{1,L_2})\varpi_{1,L_2}}}\Biggr).
\]
This map is equal to the composite 
\[
 I_{L_2} 
 \to \mathcal{O}^{\times}_{L_2} 
 \to k_2^{\times} 
 \times k_2, 
\]
where the first homomorphism is 
induced from the inverse of $\mathrm{Art}_{L_2}$, 
and the second homomorphism 
is given by 
$a+b \varpi^{1/2} \mapsto (\bar{a},\bar{b}/\bar{a})$ 
for $a \in \mu_{q^2 -1} (L_2)$ and $b \in \mathcal{O}_{L_2}$. 
Then, we rewrite Proposition \ref{ine} as follows: 
\begin{cor}\label{oddXine}
Let $\sigma \in I_L$. 
We put 
\[
 g_0=2 \zeta^{-(q+1)} (\lambda_2 (\sigma)^q
 +\zeta^{q^2-1} \lambda_2 (\sigma)) \in k. 
\]
Then, $\sigma$ induces the morphism 
\[
 \overline{\mathbf{X}}_{\zeta} \to 
 \overline{\mathbf{X}}_{\lambda_1 (\sigma)^{q+1} \zeta} ;\ 
 (z,w) \mapsto 
 (\lambda_1 (\sigma)^{-2(q+1)}(z+g_0 ), \lambda_1 (\sigma)^{-(q+1)}w). 
\]
\end{cor}
\begin{proof}
We can check that 
$\bar{a}_{\sigma} =\lambda_1 (\sigma)^{-2(q+1)}$ and 
$\bar{c}_{\sigma} =\lambda_1 (\sigma)^{-(q+1)}$ easily. 
We prove that 
\[
 \bar{b}_{\sigma} =\lambda_1 (\sigma)^{-2(q+1)} g_0. 
\]
We simply write 
$\varpi_i$ for $\varpi_{i,L_2}$. 
We put $\iota =\zeta^{q^2 -1}$ 
and 
\[
 C=\varpi_1^{\frac{q^2-1}{q}} 
 \biggl\{ \biggl( \frac{\varpi_2}{\varpi_1} 
 \biggr)^q + \iota \biggl( 
 \frac{\varpi_2}{\varpi_1} \biggr) \biggr\}. 
\]
Then, we have 
\[
 C^q-\iota \gamma_1 C \equiv -1 \pmod{\frac{1}{2}+} 
\]
by $\varpi_2^{q^2}-\varpi^{1/2} \varpi_2=-\varpi_1$. 
We can easily check the equality 
\[
 \sigma(C)-C \equiv \varpi^{\epsilon_1} 
 (\lambda_2 (\sigma)^q+\iota \lambda_2(\sigma)) \pmod{\epsilon_1 +}. 
\]
On the other hand, we can check 
\[
 c_{1,\zeta}^q \equiv -\iota (2-\gamma_4 c_{1,\zeta})
 \pmod{\frac{q-1}{2q^4} +} 
\]
by the definition of $c_{1,\zeta}$. 
Therefore, the elements 
$C$ and $c_{1,\zeta}^{q^3} /(2\iota)$ 
satisfy 
\[
 x^{q}-\iota \gamma_1 x \equiv -1 \pmod{\frac{1}{2} +}. 
\]
Hence, we obtain 
$C \equiv c_{1,\zeta}^{q^3}/(2\iota) \pmod{\epsilon_1 +}$. 
This implies 
\[
 (\sigma(c_{1,\zeta})-c_{1,\zeta})^{q^3} \equiv 
 2\iota (\sigma(C)-C) \pmod{\epsilon_1 +}. 
\]
Therefore, we obtain 
\[
 \bar{b}_{\sigma} \equiv \bar{b}_{\sigma} ^{q^3} 
 \equiv \lambda_1 (\sigma)^{-2(q+1)} g_0 \pmod{0+} 
\]
by $\xi_{\sigma} =\lambda_1 (\sigma)^{q+1} \pmod{0+}$. 
\end{proof}

\subsubsection{$q$\ :\ even}\label{eveninertia}
We assume that $q$ is even. 
We use the notation in Paragraph \ref{qeven}. 
For $P \in \mathbf{P}^0 (\mathbf{C})$, 
we have 
\begin{equation}\label{w1ine}
 w_1 (\sigma (P)) \equiv \sigma (w_1 (P)) \pmod{\frac{1}{4q^4}+} 
\end{equation} 
by \eqref{Z2}. We can see that 
\begin{equation}\label{f+1ine}
 z_{f+1} (\sigma(P)) \equiv \sigma (z_{f+1}(P)) \pmod{\frac{1}{8q^4}+} 
\end{equation}
using \eqref{pred} and \eqref{w1ine}. 
\begin{lem}
The element $\sigma$ induces the morphism 
\[
 \overline{\mathbf{P}}_{\zeta} \to 
 \overline{\mathbf{P}}_{\bar{\xi}_{\sigma} \zeta^{q^{-r_{\sigma}}}};\ 
 w_1 \mapsto w_1 ^{q^{-r_{\sigma}}}. 
\]
\end{lem}
\begin{proof}
This follows from Lemma \ref{Z11ine} and \eqref{w1ine}. 
\end{proof}
We take 
$\zeta' \in k^{\times}$. 
By \eqref{w1ine} and \eqref{f+1ine}, 
we have 
\begin{equation}\label{zwine}
\begin{split}
 a_{1,\zeta'} z(\sigma (P)) +b_{1,\zeta'} w(\sigma (P)) +c_{1,\zeta'} 
 &\equiv \\ 
 a_{1,\zeta'} \sigma(z(P)) &+\sigma(b_{1,\zeta'}) \sigma(w(P)) 
 +\sigma(c_{1,\zeta'}) 
 \pmod{\frac{1}{8q^4}+} 
\end{split}
\end{equation}
and 
\begin{equation}\label{w1wine}
 b_{2,\zeta'} w(\sigma (P)) +c_{2,\zeta'} \equiv 
 \sigma (b_{2,\zeta'} ) \sigma (w(P)) +\sigma (c_{2,\zeta'}) 
 \pmod{\frac{1}{4q^3}+}  
\end{equation}
using $\sigma (a_{1,\zeta'} ) \equiv a_{1,\zeta'} \pmod{1/(8q^4)+}$. 
We put 
\begin{align*}
 a_{\sigma,\zeta'} &=\frac{\sigma (b_{2,\zeta'})}{b_{2,\zeta'}}, \quad 
 b_{\sigma,\zeta'} =
 \frac{\sigma (b_{1,\zeta'}) b_{2,\zeta'} -b_{1,\zeta'} \sigma (b_{2,\zeta'} )}
 {a_{1,\zeta'} b_{2,\zeta'}},\\ 
 b'_{\sigma,\zeta'} &=
 \frac{\sigma (c_{2,\zeta'})-c_{2,\zeta'}}{b_{2,\zeta'}},\quad 
 c_{\sigma,\zeta'} 
 =\frac{\sigma (c_{1,\zeta'})-c_{1,\zeta'} -b_{1,\zeta'} b_{2,\zeta'}^{-1} 
 (\sigma (c_{2,\zeta'})-c_{2,\zeta'} )}{a_{1,\zeta'}}. 
\end{align*}
In the sequel, we omit the subscript 
$\zeta'$ of $a_{\sigma,\zeta'}$, 
$b_{\sigma,\zeta'}$, $b'_{\sigma,\zeta'}$ and $c_{\sigma,\zeta'}$. 
We note that $v(a_{\sigma})=0$. 
We have $v(b'_{\sigma}) \geq 0$ by \eqref{w1wine}. 
This implies $v(b_{\sigma}) \geq 0$. 
By \eqref{zwine} and \eqref{w1wine}, 
we obtain $v(c_{\sigma}) \geq 0$ 
using $v(b_{\sigma}) \geq 0$. 

\begin{prop}\label{ellact}
The element $\sigma$ induces the morphism
\[
 \overline{\mathbf{X}}_{\zeta, \zeta'} 
 \to 
 \overline{\mathbf{X}}_{\bar{\xi}_{\sigma} \zeta^{q^{-r_{\sigma}}}, \zeta'};\ 
 (z,w) \mapsto 
 (z^{q^{-r_{\sigma}}} +\bar{b}_{\sigma} w^{q^{-r_{\sigma}}} +\bar{c}_{\sigma}, 
 \bar{a}_{\sigma}w^{q^{-r_{\sigma}}} +\bar{b}'_{\sigma}). 
\]
\end{prop}
\begin{proof}
This follows from \eqref{zwine} and \eqref{w1wine}. 
\end{proof}

In the following, 
we simplify the description of 
$\bar{a}_{\sigma}$, $\bar{b}_{\sigma}$, $\bar{b}'_{\sigma}$ and 
$\bar{c}_{\sigma}$. 
Let $\tilde{\zeta}' \in \mu_{q-1}(K)$ 
be the lift of $\zeta'$. 
We put 
\[
 h_{\zeta'} (x)=x^4 -\varpi^{\frac{1}{4}} 
 \tilde{\zeta}'^4 x -\tilde{\zeta}'^4. 
\]

\begin{lem}\label{delta}
There is a root $\delta_1$ of $h_{\zeta'} (x)=0$ 
such that 
\[
 \delta_1 \equiv c_{2,\zeta'} ^{q^4} 
 +\frac{q}{2} \varpi^{\frac{1}{4}} \tilde{\zeta}'^2  \pmod{\frac{1}{4}+}. 
\]
\end{lem}
\begin{proof}
We put 
\[
 h(x)=x^{4(q-1)} +1 +\varpi^{\frac{1}{4}} x^{4q-3}. 
\]
By the definition of $c_{2,\zeta'}$, 
we have 
$h(c_{2,\zeta'}^{q^4}) \equiv 0 \pmod{1}$. 
Hence, we have a root $c_2 '$ of 
$h$ such that 
$c_2 ' \equiv c_{2,\zeta'} ^{q^4} \pmod{3/4}$ 
by Newton's method. 
We can check that 
\[
 c_2' \equiv \tilde{\zeta}'+
 \varpi^{\frac{1}{16}} \tilde{\zeta}'^{\frac{5}{4}} \pmod{\frac{1}{16}+}. 
\]
We define 
a parameter $s$ with $v(s) \geq 1/16$ by 
$x=\tilde{\zeta}' +s$. 
Then we have 
\begin{align*}
 h(\tilde{\zeta}' +s) & \equiv 
 \tilde{\zeta}'^{-4} s^4 +\biggl( \frac{q}{2}-1 \biggl) 
 \tilde{\zeta}'^{-8} s^8+\varpi^{\frac{1}{4}} (\tilde{\zeta}'+s) \\
 & \equiv 
 \tilde{\zeta}'^{-4} h_{\zeta'} (x) +
 \biggl( \frac{q}{2} -1 \biggr) 
 \tilde{\zeta}'^{-8} s^8+\varpi^{\frac{1}{4}} s^4 \tilde{\zeta}'^{-3}
 \pmod{\frac{1}{2} +}. 
\end{align*}
This implies 
\[
 h_{\zeta'} (c_2 ') \equiv 
 \frac{q}{2} \varpi^{\frac{1}{2}} \tilde{\zeta}'^6 
 \pmod{\frac{1}{2}+}. 
\]
Therefore, 
we have a root $\delta_1$ of $h_{\zeta'} (x)=0$ 
such that 
\[
 \delta_1 \equiv c_2 ' +\frac{q}{2} 
 \varpi^{\frac{1}{4}} \tilde{\zeta}'^2 \pmod{\frac{1}{4}+} 
\]
by Newton's method. 
\end{proof}

By the definition of $b_{2,\zeta'}$, 
we have 
\[
 b_{2,\zeta'}^{3q^4} \varpi^{-\frac{1}{4}} \equiv \tilde{\zeta}'^4 \pmod{0+}. 
\]
Let $\zeta''$ be the element of 
$\mu_{3(q-1)} (K^{\mathrm{ur}})$ satisfying 
$\zeta'' \equiv b_{2,\zeta'}^{q^4} \varpi^{-1/12} \pmod{0+}$. 
Note that $\zeta''^3=\tilde{\zeta}'^4$. 
We take $\delta_1$ as in Lemma \ref{delta} 
and put 
$\delta =\delta_1/(\zeta''\varpi^{1/12})$. 
Then we have 
\[
 \delta^4 -\delta=\frac{1}{\zeta''\varpi^{\frac{1}{3}}}. 
\]
Note that $v(\delta)=-1/12$. 
We take $\zeta_3 \in \mu_3(K^{\mathrm{ur}})$ 
such that $\zeta_3 \neq 1$, 
and put 
\[
 h_{\delta_1} (x)=x^2 -(1+2\zeta_3 )\varpi^{\frac{1}{4}} \delta_1 ^{2q} x -
 \varpi^{\frac{1}{4}} \delta_1 ^{4q-1} (1+2\varpi^{\frac{1}{4}} \delta_1). 
\]
\begin{lem}\label{theta}
There is a root $\theta_1$ of $h_{\delta_1} (x)=0$ 
such that 
$\theta_1 \equiv c_{1,\zeta'}^{2q^4} \pmod{1/4+}$. 
\end{lem}
\begin{proof}
By the definition of $c_{1,\zeta'}$ and $c_{2,\zeta'}$, 
we have 
$h_{\delta_1} (c_{1,\zeta'}^{2q^4}) \equiv 0 \pmod{1/2+}$. 
Hence, we can show the claim using Newton's method. 
\end{proof}

We take $\theta_1$ as in Lemma \ref{theta} 
and put 
\[
 \theta =\frac{\theta_1}{\varpi^{\frac{1}{4}} \delta_1^{2q}} -\zeta_3. 
\]
Then we have $\theta^2 -\theta =\delta^3$. 
Note that $v(\theta)=-1/8$. 
Let $\sigma \in W_K$ in this paragraph. 
We put 
\[
 \zeta_{3,\sigma} =
 \frac{\sigma(\zeta''\varpi^{\frac{1}{3}})}{\zeta''\varpi^{\frac{1}{3}}}. 
\]
We take $\nu_{\sigma} \in \mu_3(K^{\mathrm{ur}})\cup\{0\}$ such that 
$\sigma(\delta) \equiv 
 \zeta_{3,\sigma} ^{-1} (\delta +\nu_{\sigma}) \pmod{5/6}$. 
Then we have 
\begin{equation}\label{muF_4}
\begin{split}
 &\bigl( \sigma (\theta)-\theta +\nu_{\sigma}^2 \delta \bigr)^2 
 \equiv 
 \sigma (\theta)-\theta +\nu_{\sigma}^2 \delta +\nu_{\sigma}^3, \\ 
 &\bigl( \sigma (\theta)-\theta 
 +\nu_{\sigma}^2 \delta +\nu_{\sigma}^3 \bigr)^2 
 \equiv 
 \sigma (\theta)-\theta + \nu_{\sigma}^2 \delta \pmod{0+}. 
\end{split}
\end{equation}
By these equations, 
we can take $\mu_{\sigma} \in \mu_3(K^{\mathrm{ur}})\cup\{0\}$ such that 
\[
 \mu_{\sigma} \equiv 
 \sigma (\theta)-\theta +\nu_{\sigma}^2 \delta +\nu_{\sigma}^3 
 +\sigma(\zeta_3) -\zeta_3 \pmod{0+}. 
\]
Then we have 
$\mu_{\sigma}^2 +\mu_{\sigma} \equiv \nu_{\sigma}^3 \pmod{1}$ 
by \eqref{muF_4} and 
$\nu_{\sigma}, \mu_{\sigma} \in \mu_3(K^{\mathrm{ur}})\cup\{0\}$. 

\begin{lem}\label{descr}
1. Let $\sigma \in W_K$. Then we have 
\[
 a_{\sigma} \equiv \zeta_{3,\sigma},\quad 
 b_{\sigma} \equiv \zeta_{3,\sigma} \nu_{\sigma}^2,\quad 
 b'_{\sigma} \equiv \nu_{\sigma},\quad 
 c_{\sigma} \equiv \mu_{\sigma} \pmod{0+}. 
\]
2. Let $\sigma \in W_K$. 
Then we have $\bar{a}_{\sigma} \in \mathbb{F}_4 ^{\times}$ and 
$\bar{b}_{\sigma}, \bar{b}'_{\sigma}, \bar{c}_{\sigma} \in \mathbb{F}_4$. 
Further, $\bar{a}_{\sigma} \bar{b}_{\sigma}^2 =\bar{b}'_{\sigma}$ and 
$\bar{b}_{\sigma}^3=\bar{c}_{\sigma}^2 +\bar{c}_{\sigma}$ hold. 
\end{lem}
\begin{proof}
By the definition of $b_{2,\zeta'}$, 
we have 
\[
 a_{\sigma}^{4q^4} \equiv 
 \frac{\sigma(\zeta''^4 \varpi^{\frac{1}{3}})}{\zeta''^4 \varpi^{\frac{1}{3}}} 
 \pmod{0+}. 
\]
Hence we have 
$\bar{a}_{\sigma}^{4q^4} =\bar{\zeta}_{3,\sigma} \in \mathbb{F}_4^{\times}$. 
This implies 
$\bar{a}_{\sigma} =\bar{\zeta}_{3,\sigma} \in \mathbb{F}_4^{\times}$. 

By the definition of $a_{1,\zeta'}$ and $b_{1,\zeta'}$, 
we have 
\begin{align*}
 b_{\sigma}^{2q^4} &\equiv 
 \frac{\sigma(b_{2,\zeta'} ^{2q^4}) \bigl( 
 \sigma(c_{2,\zeta'}^{q^4(2q-3)})-c_{2,\zeta'}^{q^4(2q-3)}\bigr)}
 {\varpi^{\frac{1}{4}} c_{2,\zeta'}^{2q^5}} 
 \equiv 
 \frac{a_{\sigma}^{2q^4} b_{2,\zeta'} ^{2q^4} \bigl( 
 \sigma(\delta_1^{2q-3})-\delta_1^{2q-3}\bigr)}
 {\varpi^{\frac{1}{4}} \delta_1^{2q}} \\ 
 &\equiv 
 \frac{\zeta_{3,\sigma}^2 \zeta''^2 \bigl( 
 \sigma(\delta_1)-\delta_1 \bigr)}{\varpi^{\frac{1}{12}} \delta_1^4} 
 \equiv 
 \zeta_{3,\sigma}^2 
 \biggl( \frac{\sigma(\zeta'' \varpi^{\frac{1}{12}})}{\zeta'' \varpi^{\frac{1}{12}}} 
 \sigma(\delta) -\delta \biggr)
 \equiv \zeta_{3,\sigma}^2 \nu_{\sigma} 
 \pmod{0+}, 
\end{align*}
where we use 
Lemma \ref{delta} in the second congruence, 
$b_{2,\zeta'}^{q^4}/\varpi^{1/12} \equiv \zeta'' \pmod{0+}$ 
in the third congruence, 
$\delta_1^4 =\tilde{\zeta}'^4 \pmod{1/4}$ and 
$\zeta''^3=\tilde{\zeta}'^4$ in the fourth congruence 
and 
$\sigma(\zeta'' \varpi^{1/12})/(\zeta'' \varpi^{1/12} )
 \equiv \zeta_{3,\sigma} \pmod{0+}$ in the 
last congruence. 
Hence, we obtain 
$\bar{b}_{\sigma}= 
 \bar{\zeta}_{3,\sigma} \bar{\nu}_{\sigma}^2 \in \mathbb{F}_4$. 

By Lemma \ref{delta} and 
$b_{2,\zeta'}^{q^4}/\varpi^{1/12} \equiv \zeta'' \pmod{0+}$, 
we have
\[
 b_{\sigma}'^{q^4} \equiv 
 \frac{\sigma(c_{2,\zeta'}^{q^4}) -c_{2,\zeta'}^{q^4}}{b_{2,\zeta'}^{q^4}} 
 \equiv 
 \frac{\sigma(\delta_1) -\delta_1}{\zeta'' \varpi^{\frac{1}{12}}} 
 = \frac{\sigma(\zeta'' \varpi^{\frac{1}{12}})}{\zeta'' \varpi^{\frac{1}{12}}} 
 \sigma(\delta) -\delta \equiv \nu_{\sigma} \pmod{0+}. 
\]
Hence, we obtain 
$\bar{b}'_{\sigma}= \bar{\nu}_{\sigma} \in \mathbb{F}_4$.

By Lemma \ref{delta}, Lemma \ref{theta} and the definition of $a_{1,\zeta'}$, 
we have 
\begin{align*}
 c_{\sigma}^{2q^4} &\equiv 
 \frac{\sigma(\theta_1) -\theta_1 -\delta_1^{2q-3}
 (\sigma(\delta_1^2) -\delta_1^2)}{\varpi^{\frac{1}{4}} \delta_1^{2q}} \\ 
 &\equiv 
 \frac{\delta_1^{-2q} 
 \sigma\bigl(\delta_1^{2q} \varpi^{\frac{1}{4}} (\theta +\zeta_3) \bigr)
 -\varpi^{\frac{1}{4}} (\theta +\zeta_3) -\delta_1^{-3}
 (\sigma(\delta_1) -\delta_1)^2}{\varpi^{\frac{1}{4}}} \\
 &\equiv 
 \frac{\sigma\bigl( \varpi^{\frac{1}{4}} (\theta +\zeta_3) \bigr)
 -\varpi^{\frac{1}{4}} (\theta +\zeta_3) 
 -\varpi^{\frac{1}{12}} \delta \bigl( \sigma(\varpi^{\frac{1}{12}} \delta) 
 -\varpi^{\frac{1}{12}} \delta \bigr)^2}{\varpi^{\frac{1}{4}}} \\ 
 &\equiv \sigma (\theta)-\theta +\nu_{\sigma}^2 \delta 
 +\sigma(\zeta_3)-\zeta_3 
 \pmod{0+}, 
\end{align*}
where we use $\sigma(\delta_1) \equiv \delta_1 \pmod{1/4}$ 
in the second congruence, 
$\delta_1^4 =\tilde{\zeta}'^4 \pmod{1/4}$ in the third congruence. 
Then we have $\bar{c}_{\sigma}^{2q^4} \in \mathbb{F}_4$ by \eqref{muF_4}. 
Hence we have $\bar{c}_{\sigma} \in \mathbb{F}_4$ and 
$c_{\sigma} \equiv \mu_{\sigma} \pmod{0+}$ again by \eqref{muF_4}. 

By the above calculations, 
we can easily check 
$\bar{a}_{\sigma} \bar{b}_{\sigma}^2 =\bar{b}'_{\sigma}$ and 
$\bar{b}_{\sigma}^3=\bar{c}_{\sigma}^2 +\bar{c}_{\sigma}$. 
\end{proof}
\begin{lem}\label{Galois}
The field 
$K(\zeta_3,\zeta''\varpi^{1/3},\theta)$ is a Galois extension 
over $K$. 
\end{lem}
\begin{proof}
Let $\sigma \in W_K$. 
It suffices to show 
$\sigma(\theta) \in K(\zeta_3,\zeta''\varpi^{1/3},\theta)$. 
We put 
\[
 \theta_{\sigma} =
 \theta +\nu_{\sigma}^2 \delta +\nu_{\sigma}^3 +\mu_{\sigma}
 +\sigma(\zeta_3 )-\zeta_3. 
\]
Then we have 
$\theta_{\sigma}^2 -\theta_{\sigma} \equiv 
 \sigma(\delta)^3 \pmod{2/3}$. 
Hence, 
we can find $\theta'$ such that 
$\theta'^2 -\theta'=\sigma(\delta)^3$ and 
$\theta' \equiv \theta_{\sigma} \pmod{2/3}$. 
By the choice of $\mu_{\sigma}$, 
we have $\theta' =\sigma(\theta) \pmod{0+}$. 
Hence, we obtain $\theta' =\sigma(\theta)$. 

We take $\sigma' \in W_K$ such that 
$\sigma'(\theta) \neq \sigma(\theta)$. 
We can define $\theta_{\sigma'}$ as above, 
and have 
$\sigma'(\theta) \equiv \theta_{\sigma'} \pmod{2/3}$. 
If $\nu_{\sigma} = \nu_{\sigma'}$, 
then we have 
$\zeta_{3,\sigma} \sigma(\delta) \equiv 
 \zeta_{3,\sigma'} \sigma'(\delta) \pmod{5/6}$, 
which implies 
$\zeta_{3,\sigma} \sigma(\delta) = 
 \zeta_{3,\sigma'} \sigma'(\delta)$ 
because both sides are roots of 
\[
 x^4 -x- \frac{1}{\zeta''\varpi^{\frac{1}{3}}} =0. 
\]
Hence, if $\sigma(\delta)^3 \neq \sigma'(\delta)^3$, 
we have 
$\nu_{\sigma} \neq \nu_{\sigma'}$, 
which implies 
\[
 \sigma (\theta) \equiv \theta_{\sigma} \not\equiv 
 \theta_{\sigma'} \equiv \sigma' (\theta) \pmod{0+}. 
\]
If $\sigma(\delta)^3 =\sigma'(\delta)^3$, we have 
$\sigma (\theta) \not\equiv \sigma' (\theta) \pmod{0+}$. 
Therefore we have 
\[
 v(\sigma(\theta) -\theta_{\sigma}) >v(\sigma'(\theta)-\theta_{\sigma}). 
\]
Then, we obtain 
\[
 \sigma(\theta) \in K(\theta_{\sigma}) 
 \subset K(\zeta_3,\zeta''\varpi^{1/3},\theta) 
\]
by Krasner's lemma. 
\end{proof}

Let $E$ be the elliptic curve over $k^{\mathrm{ac}}$ 
defined by 
$z^2+z=w^3$. 
We put 
\[ 
 Q=
 \Biggl\{
 g(\alpha, \beta, \gamma)=
 \begin{pmatrix}
 \alpha & \beta &  \gamma \\
 & \alpha^2 &  \beta^2\\
 & & \alpha
 \end{pmatrix}
 \in \mathrm{GL}_3 (\mathbb{F}_4) \ \Bigg| \ 
 \alpha\gamma^2
 +\alpha^2 \gamma=\beta^3 
 \Biggr\}.
\]
We note that $|Q|=24$ and 
$Q$ is isomorphic to 
$\mathrm{SL}_2(\mathbb{F}_3)$ (cf. \cite[8.5. Exercices 2]{SerLin}). 
Let $Q \rtimes \mathbb{Z}$ be a semidirect product, 
where $r \in \mathbb{Z}$ acts on $Q$ by 
$g(\alpha, \beta, \gamma) \mapsto 
 g(\alpha^{q^r} , \beta^{q^r} , \gamma^{q^r} )$. 
Then $Q \rtimes \mathbb{Z}$ 
acts faithfully on $E$ as a scheme over $k$, 
where 
$(g(\alpha, \beta, \gamma),r) \in Q \rtimes \mathbb{Z}$ 
acts on $E$ by 
\[
 (z,w) \mapsto 
 \bigl(z^{q^{-r}} +\alpha^{-1} (\beta w^{q^{-r}} + \gamma), 
 \alpha(w^{q^{-r}} +(\alpha^{-1} \beta )^2)\bigr) 
\]
for $k^{\mathrm{ac}}$ valued points. 

\begin{prop}\label{Wact}
The element $\sigma \in W_K$ sends 
$\overline{\mathbf{X}}_{\zeta, \zeta'}$ to 
$\overline{\mathbf{X}}_{\bar{\xi}_{\sigma} \zeta^{q^{-r_{\sigma}}}, \zeta'}$. 
We identify $\overline{\mathbf{X}}_{\zeta, \zeta'}$ with 
$\overline{\mathbf{X}}_{\bar{\xi}_{\sigma} \zeta^{q^{-r_{\sigma}}}, \zeta'}$ 
by 
$(z, w) \mapsto (z, w)$. 
Then the action of $W_K$ gives a homomorphism 
\[
 \Theta_{\zeta'} \colon W_K \to 
 Q \rtimes \mathbb{Z} 
 \subset \Aut_k (\overline{\mathbf{X}}_{\zeta, \zeta'}) 
 ;\ \sigma \mapsto 
 \bigl( g(\bar{\zeta}_{3,\sigma}, \bar{\zeta}_{3,\sigma}^2 
 \bar{\nu}_{\sigma}^2, 
 \bar{\zeta}_{3,\sigma} \bar{\mu}_{\sigma}), 
 r_{\sigma} \bigr). 
\]
\end{prop}
\begin{proof}
This follows from Proposition \ref{ellact} and Lemma \ref{descr}. 
\end{proof}
\begin{prop}
The homomorphism $\Theta_{\zeta'}$ factors through 
$W (K^{\mathrm{ur}} (\varpi^{1/3},\theta)/K)$ 
and gives an isomorphism 
$W (K^{\mathrm{ur}} (\varpi^{1/3},\theta)/K) \simeq Q \rtimes \mathbb{Z}$. 
\end{prop}
\begin{proof}
By Lemma \ref{descr}.1, the homomorphism
$\Theta_{\zeta'}$ factors through 
$W (K^{\mathrm{ur}} (\varpi^{1/3},\theta)/K)$ 
and induces an injective homomorphism 
$W (K^{\mathrm{ur}} (\varpi^{1/3},\theta)/K) \to Q \rtimes \mathbb{Z}$. 

To prove the surjectivity, 
it suffices to show that 
$\Theta_{\zeta'}$ sends $I_K$ onto $Q$. 
Let $g=g(\alpha, \beta, \gamma) \in Q$. 
We take $\zeta_{\alpha} \in \mu_3(K^{\mathrm{ur}})$, 
$\nu_{\beta}, \mu_{\gamma} \in \mu_3(K^{\mathrm{ur}})\cup \{0\}$
such that 
$\bar{\zeta}_{\alpha} =\alpha$, 
$\bar{\nu}_{\beta} =\alpha^{-1} \beta$ and 
$\bar{\mu}_{\gamma} =\alpha^{-1} \gamma$. 
We put 
$\delta_g =\zeta_{\alpha}^{-1} (\delta +\nu_{\beta})$ and 
$\theta_g =\theta +\nu_{\beta}^2 \delta +\nu_{\beta}^3 +\mu_{\gamma}$. 
Then we have 
\[
 \delta_g^4 -\delta_g \equiv 
 \frac{1}{\zeta_{\alpha} \zeta'' \varpi^{\frac{1}{3}}} \pmod{\frac{5}{6}}. 
\] 
Hence, we can find 
$\delta'_g$ such that 
$\delta'^4_g -\delta'_g = 1/(\zeta_{\alpha} \zeta'' \varpi^{1/3})$ 
and $\delta'_g \equiv \delta_g \pmod{5/6}$. 
Further, we have 
$\theta_g ^2 -\theta_g \equiv \delta'^3 _g \pmod{2/3}$. 
Hence, we can find $\theta'_g$ such that 
$\theta'^2_g -\theta'_g = \delta'^3 _g$ and 
$\theta'_g \equiv \theta_g \pmod{2/3}$. 
Then 
$\varpi^{1/3} \mapsto \zeta_{\alpha} \varpi^{1/3}$ and 
$\theta \mapsto \theta'_g$ gives an element 
of $I_K$, whose image by $\Theta_{\zeta'}$ is $g$. 
\end{proof}

\section{Cohomology of $\mathbf{X}_1(\mathfrak{p}^3)$}\label{cohLT3}
In this section we show that 
the covering $\mathcal{C}_1 (\mathfrak{p}^3)$ is 
semi-stable, and study a structure of 
$\ell$-adic cohomology of $\mathbf{X}_1(\mathfrak{p}^3)$. 
In the sequel, for a projective 
smooth curve $X$ over $k$, 
we simply write 
$H^1(X,\overline{\mathbb{Q}}_{\ell})$ for 
$H^1(X_{k^{\mathrm{ac}}},\overline{\mathbb{Q}}_{\ell})$. 
For a finite abelian group $A$, 
the character group 
$\Hom_{\mathbb{Z}} (A,\overline{\mathbb{Q}}_{\ell} ^{\times})$ 
is denoted by $A^{\vee}$. 

\subsection{Cohomology of reductions}
Let $X_{\mathrm{DL}}$ be 
the smooth compactification 
of the affine curve over $k$ 
defined by $X^q-X=Y^{q+1}$. 
The curve $X_{\mathrm{DL}}$ is also the 
smooth compactification of 
the Deligne-Lusztig curve $x^qy-xy^q=1$ 
for $\mathrm{SL}_2(\mathbb{F}_q)$. 
Then, $a \in k$ 
acts on $X_{\mathrm{DL}}$ by 
\[
 \alpha_a \colon (X,Y) \mapsto (X+a,Y). 
\]
On the other hand, 
$\zeta \in k_2^{\times}$ acts 
on $X_{\mathrm{DL}}$ by 
\[
 \beta_{\zeta} \colon (X,Y) \mapsto (\zeta^{q+1} X,\zeta Y). 
\]
By these actions, 
we consider 
$H^1(X_{\mathrm{DL}},\overline{\mathbb{Q}}_{\ell})$
as a 
$\overline{\mathbb{Q}}_{\ell}
 [k \times k_2^{\times}]$-module. 
\begin{lem}\label{oo3}
We have an isomorphism  
\[
 H^1(X_{\mathrm{DL}},\overline{\mathbb{Q}}_{\ell})
 \simeq \bigoplus_{\psi \in 
 k^{\vee} 
 \backslash \{1\}}
 \bigoplus_{\chi 
 \in \mu_{q+1}(k_2)^{\vee} \backslash \{1\}} 
 \psi \otimes \chi 
\] 
as 
$\overline{\mathbb{Q}}_{\ell}
 [k \times \mu_{q+1}(k_2)]$-modules. 
\end{lem}
\begin{proof}
As 
$\overline{\mathbb{Q}}_{\ell}
 [k \times \mu_{q+1}(k_2)]$-modules, 
we have the short exact sequence
\begin{equation}\label{bb1}
 0 \to 
 \bigoplus_{\psi \in 
 k^{\vee}}\psi 
 \to
 H_c^1(X_{\mathrm{DL}} \backslash X_{\mathrm{DL}}(k),
 \overline{\mathbb{Q}}_{\ell})
 \to H^1(X_{\mathrm{DL}},
 \overline{\mathbb{Q}}_{\ell})
 \to 0. 
\end{equation}
Let $\mathcal{L}_{\psi}$ 
denote the Artin-Schreier 
$\overline{\mathbb{Q}}_{\ell}$-sheaf 
associated to 
$\psi \in k^{\vee}$. 
Let $\mathcal{K}_{\chi}$ 
denote the Kummer 
$\overline{\mathbb{Q}}_{\ell}$-sheaf 
associated to
$\chi \in \mu_{q+1}(k_2)^{\vee}$. 
Since 
\[
 X_{\mathrm{DL}} \backslash X_{\mathrm{DL}}(k) 
 \to \mathbb{G}_m;\ (X,Y) \mapsto Y^{q+1} 
\]
is a finite etale Galois covering 
with a Galois group $k \times \mu_{q+1}(k_2)$, 
we have the isomorphism 
\begin{equation}\label{bb3}
 H_c^1(X_{\mathrm{DL}} \backslash X_{\mathrm{DL}}(k), 
 \overline{\mathbb{Q}}_{\ell}) \simeq 
 \bigoplus_{\psi \in 
 k^{\vee}} 
 \bigoplus_{\chi \in \mu_{q+1}(k_2)^{\vee}}
 H_c^1(\mathbb{G}_m,
 \mathcal{L}_{\psi} \otimes \mathcal{K}_{\chi})
\end{equation}
as 
$\overline{\mathbb{Q}}_{\ell}
 [k \times \mu_{q+1}(k_2)]$-modules. 
Note that we have 
\[
 \dim H_c^1(\mathbb{G}_m, \mathcal{L}_{\psi} \otimes \mathcal{K}_{\chi})=1 
\]
if $\psi \neq 1$ 
by the Grothendieck-Ogg-Shafarevich formula 
(cf. \cite[Expos\'{e} X Th\'{e}or\`{e}me 7.1]{SGA5}). 
Clearly, if $\chi \neq 1$, we have 
$H_c^1(\mathbb{G}_m,\mathcal{K}_{\chi})=0$ and 
$H_c^1(\mathbb{G}_m,\mathcal{L}_{\psi}) \simeq \psi$. 
Hence, we acquire the isomorphism 
\begin{equation}\label{bb2}
 \bigoplus_{\psi \in k^{\vee}}
 \bigoplus_{\chi \in \mu_{q+1}(k_2)^{\vee}}
 H_c^1(\mathbb{G}_m,\mathcal{L}_{\psi} 
 \otimes \mathcal{K}_{\chi}) \simeq 
 \bigoplus_{\psi \in k^{\vee} \setminus \{ 1 \} }
 \bigoplus_{\chi \in \mu_{q+1}(k_2)^{\vee} \setminus \{ 1 \} }
 H_c^1(\mathbb{G}_m,
 \mathcal{L}_{\psi} \otimes \mathcal{K}_{\chi}) \oplus
 \bigoplus_{\psi \in k^{\vee}}\psi
\end{equation}
as 
$\overline{\mathbb{Q}}_{\ell}
 [k \times \mu_{q+1}(k_2)]$-modules. 
By \eqref{bb1}, \eqref{bb3} and \eqref{bb2}, 
the required assertion follows. 
\end{proof}
For a character 
$\psi \in k^{\vee}$ and an element 
$\zeta \in k^{\times}$, we denote 
by $\psi_{\zeta}$ 
the character $x \mapsto \psi(\zeta x)$. 
We consider a character group 
$(k^{\times})^{\vee}$ 
as a subgroup of 
$(k_2^{\times})^{\vee}$
by 
$\mathrm{Nr}_{k_2/k}^{\vee}$. 
\begin{lem}\label{cohDL}
We have an isomorphism  
\[
 H^1(X_{\mathrm{DL}},\overline{\mathbb{Q}}_{\ell})
 \simeq \bigoplus_{\tilde{\chi} \in 
 (k_2^{\times})^{\vee} 
 \backslash (k^{\times})^{\vee}}\tilde{\chi} 
\] 
as 
$\overline{\mathbb{Q}}_{\ell}[k_2^{\times}]$-modules. 
\end{lem}
\begin{proof}
By Lemma \ref{oo3}, we take a basis 
\[
 \{e_{\psi,\chi}\}_{\psi \in k^{\vee} \backslash \{ 1 \}, \, 
 \chi \in \mu_{q+1}(k_2)^{\vee} \backslash \{1\}} 
\]
of 
$H^1(X_{\mathrm{DL}},\overline{\mathbb{Q}}_{\ell})$
over $\overline{\mathbb{Q}}_{\ell}$ 
such that 
$k \times \mu_{q+1}(k_2)$ acts on 
$e_{\psi,\chi}$ by $\psi \otimes \chi$. 
For $\zeta \in k_2^{\times}$ 
and $a \in k$, 
we have 
\[
 \beta_{\zeta} \circ \alpha_{a} 
 \circ \beta_{\zeta}^{-1}=\alpha_{\zeta^{q+1}a} 
\]
in $\Aut_{k_2}(X_{\mathrm{DL}})$. 
Hence, 
$\zeta \in k_2^{\times}$ acts on 
$H^1(X_{\mathrm{DL}},\overline{\mathbb{Q}}_{\ell})$ 
by 
\[
 e_{\psi,\chi} \mapsto 
 c_{\psi,\chi,\zeta}e_{\psi_{\zeta^{-(q+1)}},\chi} 
\]
with some constant 
$ c_{\psi,\chi,\zeta} \in \overline{\mathbb{Q}}^{\times}_{\ell}$. 
Therefore, we acquire an isomorphism 
\[
 H^1(X_{\mathrm{DL}},\overline{\mathbb{Q}}_{\ell}) \simeq 
 \bigoplus_{\chi \in \mu_{q+1}(k_2)^{\vee} \backslash \{1\}}
 \Ind_{\mu_{q+1}(k_2)}^{k_2^{\times}}(\chi) 
\]
as 
$\overline{\mathbb{Q}}_{\ell}[k_2^{\times}]$-modules. 
Hence, the required assertion follows. 
\end{proof}
\begin{prop}\label{Yact}
We have isomorphisms 
\begin{align*}
 &H^1(\overline{\mathbf{Y}}^{\mathrm{c}}_{1,2}, 
 \overline{\mathbb{Q}}_{\ell}) \simeq 
 \bigoplus_{\tilde{\chi} \in 
 (k_2^{\times})^{\vee} 
 \backslash (k^{\times})^{\vee}} 
 (\tilde{\chi} \circ \lambda) \otimes 
 (\tilde{\chi}^{q} \circ \kappa_1 ),\\ 
 &H^1(\overline{\mathbf{Y}}^{\mathrm{c}}_{2,1}, 
 \overline{\mathbb{Q}}_{\ell}) \simeq
 \bigoplus_{\tilde{\chi} \in 
 (k_2^{\times})^{\vee} 
 \backslash (k^{\times})^{\vee}} 
 (\tilde{\chi} \circ \lambda )\otimes 
 (\tilde{\chi} \circ \kappa_1 )
\end{align*}
as $(I_K \times \mathcal{O}_D^{\times})$-representations 
over $\overline{\mathbb{Q}}_{\ell}$. 
\end{prop}
\begin{proof}
This follows from 
Lemma \ref{dYact}, Lemma \ref{iYact} and 
Lemma \ref {cohDL}. 
\end{proof}

Let $X_{\mathrm{AS}}$ be the smooth compactification 
of the affine curve 
$X'_{\mathrm{AS}}$ over $k$ 
defined by $z^q-z=w^2$. 
Let $a \in k$ act on 
$X_{\mathrm{AS}}$ by 
\[
 \alpha_a \colon (z,w) \mapsto (z+a,w). 
\]
By this action, 
we consider 
$H^1(X_{\mathrm{AS}},\overline{\mathbb{Q}}_{\ell})$ 
as a 
$\overline{\mathbb{Q}}_{\ell} [k]$-module. 
On the other hand, let 
$b \in \mu_{2(q-1)}(k^{\mathrm{ac}})$ 
act on $X_{\mathrm{AS}}$ 
by 
\[
 \beta_b \colon 
 (z,w) \mapsto (b^2z,bw). 
\]
\begin{lem}\label{swan}
We assume that $q$ is odd. 
Let $G$ be the Galois group of the Galois extension 
$F$ over $k((s))$ defined by $z^q-z=1/s^2$. 
Let $G^r$ be the upper numbering 
ramification filtration of $G$. 
Then $G^r =G$ if $r \leq 2$, and 
$G^r =1$ if $r>2$. 
\end{lem}
\begin{proof}
We take $a \in F$ such that $a^q-a=1/s^2$. 
Then $sa^{(q-1)/2}$ is a uniformizer of $F$. 
Let $v_F$ be the normalized valuation of $F$. 
For $\sigma \in G$ and an integer $i$, 
the condition 
\[
 v_F \bigl( \sigma(sa^{\frac{q-1}{2}}) -sa^{\frac{q-1}{2}} \bigr) \geq i 
\]
is equivalent to 
the condition 
\[
 v_F \bigl( \sigma(a) -a \bigr) \geq i-3. 
\]
Hence, the claim follows. 
\end{proof}

For a character 
$\psi \in k^{\vee}$ and 
$x \in k^{\times}$, 
we write $\psi_x \in k^{\vee}$ 
for the character 
$y \mapsto \psi(xy)$. 
We set 
\[
 V=\bigoplus_{\psi \in k^{\vee} \backslash \{1\}} \psi 
\]
as $\overline{\mathbb{Q}}_{\ell}[k]$-modules. 
Let 
$\{e_{\psi}\}_{\psi \in k^{\vee} \backslash \{1\}}$
be the standard basis of $V$. 

\begin{lem}\label{cohAS}
We assume that $q$ is odd. 
\\1.\ Then we have 
$H^1(X_{\mathrm{AS}},\overline{\mathbb{Q}}_{\ell}) \simeq V$ 
as $\overline{\mathbb{Q}}_{\ell}[k]$-modules. 
\\2.\ 
For $b \in \mu_{2(q-1)}(k^{\mathrm{ac}})$, 
the automorphism $\beta_b$ of $X_{\mathrm{AS}}$ 
induces the action 
\[
 e_{\psi} \mapsto c_{\psi,b}e_{\psi_{b^{-2}}} 
\] 
on 
$H^1(X_{\mathrm{AS}},\overline{\mathbb{Q}}_{\ell}) \simeq V$
with some constant 
$c_{\psi,b} \in 
\overline{\mathbb{Q}}^{\times}_{\ell}$. 
Furthermore, we have 
$c_{\psi,-1}=-1$. 
\end{lem}
\begin{proof}
We have 
$H^1(X_{\mathrm{AS}},\overline{\mathbb{Q}}_{\ell})
 \simeq H_c^1(X'_{\mathrm{AS}},\overline{\mathbb{Q}}_{\ell})$, 
because the complement 
$X_{\mathrm{AS}} \backslash X'_{\mathrm{AS}}$ 
consists of one point. 
The curve $X'_{\mathrm{AS}}$ is a finite etale Galois 
covering of $\mathbb{A}^1$ with a Galois group 
$k$ 
by $(z,w) \mapsto w$. 
For $\psi \in k^{\vee}$, 
let $\mathcal{L}_{2,\psi}$ 
be the smooth $\overline{\mathbb{Q}}_{\ell}$-sheaf
on $\mathbb{A}^1$ defined by the covering 
$X'_{\mathrm{AS}}$ and $\psi$. 
Then we have 
\[
 H_c^1(X'_{\mathrm{AS}},\overline{\mathbb{Q}}_{\ell})
 \simeq \bigoplus_{\psi \in k^{\vee} \backslash \{1\}}
 H_c^1(\mathbb{A}^1,\mathcal{L}_{2,\psi})
\]
as $\overline{\mathbb{Q}}_{\ell}[k]$-modules. 
By Lemma \ref{swan} and the Grothendieck-Ogg-Shafarevich formula, 
we have 
\[
 \dim H_c^1(\mathbb{A}^1,\mathcal{L}_{2,\psi})=1 
\]
and $H_c^1(\mathbb{A}^1,\mathcal{L}_{2,\psi}) \simeq \psi$ 
as $\overline{\mathbb{Q}}_{\ell}[k]$-modules 
for $\psi \in k^{\vee} \backslash \{1\}$. 
Hence, the first assertion follows. 

The second assertion follows from that 
$\beta_b \alpha_a \beta^{-1}_b =\alpha_{ab^2}$ for $a \in k$ and 
$b \in \mu_{2(q-1)}(k^{\mathrm{ac}})$. 
The assertion $c_{\psi,-1}=-1$ follows from the 
Lefschetz trace formula. 
\end{proof}

We put 
\[
 U_D =\{d \in \mathcal{O}_D^{\times} 
 \mid \kappa_1(d) \in k^{\times} \}. 
\]
We take 
$\zeta_0 \in \mu_{2(q^2-1)}(k^{\mathrm{ac}}) \backslash k_2^{\times}$. 
Let $\Delta \in (k^{\times})^{\vee}$ 
be the character defined by 
\[
 x \mapsto x^{\frac{q-1}{2}} \in \{\pm 1\} \subset 
 \overline{\mathbb{Q}}_{\ell}^{\times} 
\]
for $x \in k^{\times}$. 
If $q$ is odd, we put 
\begin{align*}
 \tau_{\chi, \psi} &=\Ind_{I_L}^{I_K} 
 \bigl( (\chi \circ \lambda_1^{q+1} ) \otimes 
 (\psi^2 \circ \Tr_{k_2/k} \circ 
 \lambda_2 ) \bigr) ,\\ 
 \tau'_{\chi, \psi} &=\Ind_{I_L}^{I_K} 
 \Bigl( (\chi \circ \lambda_1^{q+1} ) \otimes 
 \bigl( \psi^2 \circ \Tr_{k_2/k} \circ (-\zeta_0^{-(q+1)} 
 \lambda_2 ) \bigr) \Bigr) , \\ 
 \theta_{\chi, \psi} &= 
 (\Delta \chi \circ \kappa_1 ) \otimes 
 (\psi \circ \Tr_{k_2/k} \circ 
 \kappa_2 ), \\ 
 \theta'_{\chi, \psi} &= 
 (\Delta \chi \circ \kappa_1 ) \otimes 
 \bigl( \psi \circ \Tr_{k_2/k} \circ 
 (\zeta_0 ^{-2q} \kappa_2 )\bigr), \\ 
 \rho_{\chi, \psi} 
 &=\Ind_{U_D}^{\mathcal{O}_D^{\times}} \theta_{\chi, \psi}, \\ 
 \rho'_{\chi, \psi} 
 &=\Ind_{U_D}^{\mathcal{O}_D^{\times}} \theta'_{\chi, \psi} 
\end{align*}
for 
$\chi \in (k^{\times})^{\vee}$ and 
$\psi \in k^{\vee} \backslash \{1\}$. 
We note that 
\[
 \dim \rho_{\chi, \psi}=\dim \rho'_{\chi, \psi}=q+1. 
\]
For $\psi, \psi' \in k^{\vee} \backslash \{1\}$, 
we can check that 
$\tau_{\chi, \psi}=\tau_{\chi, \psi'}$ 
if and only if $\psi'=\psi^{-1}$, and 
$\rho_{\chi, \psi}=\rho_{\chi, \psi'}$ 
if and only if $\psi'=\psi^{-1}$. 
Similar things hold also for 
$\tau'_{\chi, \psi}$ and $\rho'_{\chi, \psi}$. 
We define an equivalence relation $\sim$ on 
$k^{\vee} \backslash \{1\}$ by 
$\psi \sim \psi^{-1}$. 
We put 
\begin{equation*}
 \Pi_{\chi, \psi} = 
 \tau_{\chi, \psi} \otimes \rho_{\chi, \psi}, \quad 
 \Pi'_{\chi, \psi} = 
 \tau'_{\chi, \psi} \otimes \rho'_{\chi, \psi} 
\end{equation*}
for 
$\chi \in (k^{\times})^{\vee}$ and 
$\psi \in k^{\vee} \backslash \{1\}$. 

\begin{prop}\label{oddact}
We assume that $q$ is odd. Then we have an isomorphism 
\[
 \bigoplus_{\zeta \in \mu_{2(q^2-1)}(k^{\mathrm{ac}})} 
 H^1 (\overline{\mathbf{X}}^{\mathrm{c}}_{\zeta}, 
 \overline{\mathbb{Q}}_{\ell} ) 
 \simeq \bigoplus_{\chi \in (k^{\times})^{\vee}} 
 \bigoplus_{\psi \in (k^{\vee} \backslash \{1\})/\!\sim} 
 \Pi_{\chi, \psi} \oplus \Pi'_{\chi, \psi}
\]
as representations of $I_K \times \mathcal{O}_D^{\times}$. 
\end{prop}
\begin{proof}
The actions of $I_L$ and $U_D$ on 
$\bigoplus_{\zeta \in k^{\times}} 
 H^1 (\overline{\mathbf{X}}^{\mathrm{c}}_{\zeta}, 
 \overline{\mathbb{Q}}_{\ell} )$ 
factor through $k^{\times} \times k$ 
by Proposition \ref{dXact} 
and Corollary \ref{oddXine}. 
On the other hand, 
the action of $k^{\times} \times k$ on 
$\bigoplus_{\zeta \in k^{\times}} 
 H^1 (\overline{\mathbf{X}}^{\mathrm{c}}_{\zeta}, 
 \overline{\mathbb{Q}}_{\ell} )$ 
is induced from the action of 
$\{1\} \times k$ on 
$H^1 (\overline{\mathbf{X}}^{\mathrm{c}}_1, 
 \overline{\mathbb{Q}}_{\ell} )$. 
Hence, we have 
\[
 \bigoplus_{\zeta \in k^{\times}} 
 H^1 (\overline{\mathbf{X}}^{\mathrm{c}}_{\zeta}, 
 \overline{\mathbb{Q}}_{\ell} ) 
 \simeq \bigoplus_{\chi \in (k^{\times})^{\vee}} 
 \bigoplus_{\psi \in k^{\vee} \backslash \{1\}} 
 \chi \otimes \psi 
\]
as representations of 
$k^{\times} \times k$ 
by Lemma \ref{cohAS}.1. 
Therefore, 
we have an 
isomorphism 
\[
 \bigoplus_{\zeta \in k^{\times}} 
 H^1 (\overline{\mathbf{X}}^{\mathrm{c}}_{\zeta}, 
 \overline{\mathbb{Q}}_{\ell} ) 
 \simeq \bigoplus_{\chi \in (k^{\times})^{\vee}} 
 \bigoplus_{\psi \in k^{\vee} \backslash \{1\}} 
 (\chi \circ \lambda_1^{q+1} ) \otimes 
 (\psi^2 \circ \Tr_{k_2/k} \circ 
 \lambda_2 ) \otimes \theta_{\chi, \psi} 
\]
as representations of $I_L \times U_D$ 
by Proposition \ref{dXact}, 
Corollary \ref{oddXine} and Lemma \ref{cohAS}.2. 
Inducing this representation 
from $U_D$ to 
$\mathcal{O}_D^{\times}$, 
we obtain an isomorphism 
\[
 \bigoplus_{\zeta \in k_2^{\times}} 
 H^1 (\overline{\mathbf{X}}^{\mathrm{c}}_{\zeta}, 
 \overline{\mathbb{Q}}_{\ell} ) 
 \simeq \bigoplus_{\chi \in (k^{\times})^{\vee}} 
 \bigoplus_{\psi \in k^{\vee} \backslash \{1\}} 
 (\chi \circ \lambda_1^{q+1} ) \otimes 
 (\psi^2 \circ \Tr_{k_2/k} \circ 
 \lambda_2 ) \otimes \rho_{\chi, \psi} 
\]
as representations of $I_L \times \mathcal{O}_D^{\times}$. 
On the left hand side of this isomorphism, 
we have an action of $I_K$ that commutes with the action of 
$\mathcal{O}_D^{\times}$. 
Hence, we have 
\[
 \bigoplus_{\zeta \in k_2^{\times}} 
 H^1 (\overline{\mathbf{X}}^{\mathrm{c}}_{\zeta}, 
 \overline{\mathbb{Q}}_{\ell} ) 
 \simeq \bigoplus_{\chi \in (k^{\times})^{\vee}} 
 \bigoplus_{\psi \in (k^{\vee} \backslash \{1\})/\!\sim} 
 \tau_{\chi, \psi} \otimes \rho_{\chi, \psi} 
\]
as representations of $I_K \times \mathcal{O}_D^{\times}$. 
By the similar arguments, we have 
\[
 \bigoplus_{\zeta \in \mu_{2(q^2-1)}(k^{\mathrm{ac}}) 
 \backslash k_2^{\times}} 
 H^1 (\overline{\mathbf{X}}^{\mathrm{c}}_{\zeta}, 
 \overline{\mathbb{Q}}_{\ell} ) 
 \simeq \bigoplus_{\chi \in (k^{\times})^{\vee}} 
 \bigoplus_{\psi \in (k^{\vee} \backslash \{1\})/\!\sim} 
 \tau'_{\chi, \psi} \otimes \rho'_{\chi, \psi} 
\]
as representations of $I_K \times \mathcal{O}_D^{\times}$. 
Therefore, we have the isomorphism in the assertion. 
\end{proof}

Let $E$ and $Q$ be as in Paragraph \ref{eveninertia}. 
Let $Z \subset Q$ 
be the subgroup consisting of 
$g(1,0,\gamma)$ with $\gamma^2+\gamma=0$, 
and $\phi$ be the unique
non-trivial 
character of $Z$. 
By \cite[Lemma 22.2]{BHLLC}, 
there exists 
a unique irreducible 
two-dimensional representation 
$\tau$ of $Q$ 
such that 
\begin{equation}\label{qq}
 \tau|_{Z} \simeq \phi^{\oplus 2}, \quad 
 \Tr \tau(g(\alpha,0,0)) =-1 
\end{equation}
for 
$\alpha \in \mathbb{F}^{\times}_{4} \backslash \{1\}$. 
Then, it is easily checked that 
the determinant character of 
$\tau$ is trivial. 
Note that every 
two-dimensional irreducible 
representation of $Q$ 
has a form $\tau \otimes \chi$ with 
$\chi \in (\mathbb{F}^{\times}_4)^{\vee}$, 
where we consider $\chi$ as a character of 
$Q$ by 
$g(\alpha,\beta,\gamma)\mapsto \chi(\alpha)$. 
\begin{lem}\label{Qrep}
The $Q$-representation 
$H^1(E,\overline{\mathbb{Q}}_{\ell})$
is isomorphic to $\tau$. 
\end{lem}
\begin{proof}
The $Q$-representation 
$H^1(E,\overline{\mathbb{Q}}_{\ell})$ 
satisfies \eqref{qq} by Lemma \ref{oo3}. 
Hence, the assertion follows. 
\end{proof}

Let $\tau_{\zeta'}$ be the 
representation of $W_K$ induced from 
the $(Q \rtimes \mathbb{Z})$-representation 
$H^1(E,\overline{\mathbb{Q}}_{\ell})$ 
by $\Theta_{\zeta'}$. 
Then the restriction to $I_K$ of $\tau_{\zeta'}$ 
is isomorphic to the representation induced from $\tau$ 
by Lemma \ref{Qrep}. 

We say that 
a continuous two-dimensional irreducible 
representation $V$ of $W_K$ 
over $\overline{\mathbb{Q}}_{\ell}$ is primitive, 
if there is no pair of a quadratic extension $K'$ and 
a continuous character $\chi$ of $W_{K'}$ such that 
$V \simeq \Ind_{W_{K'}}^{W_K} \chi$. 
\begin{lem}
The representation 
$\tau_{\zeta'}$ is primitive of Artin conductor $3$. 
\end{lem}
\begin{proof}
We use the notations in the proof of Lemma \ref{Galois}. 
The element $1/(\varpi^{1/3}\theta^3)$ is a uniformizer of 
$K^{\mathrm{ur}}(\varpi^{1/3},\theta)$. 
For $\sigma \in I_K$, we can show that 
\[
 v \biggl( \sigma\biggl(\frac{1}{\varpi^{\frac{1}{3}}\theta^3}\biggl) 
 -\frac{1}{\varpi^{\frac{1}{3}}\theta^3} \biggr) =
 \begin{cases}
 \frac{1}{24} &\quad \textrm{if $\zeta_{3,\sigma} \neq 1$,}\\ 
 \frac{1}{12} &\quad 
 \textrm{if $\zeta_{3,\sigma}=1$, $\nu_{\sigma} \neq 0$,}\\ 
 \frac{1}{6} &\quad 
 \textrm{if $\zeta_{3,\sigma}=1$, $\nu_{\sigma}=0$, $\mu_{\sigma} \neq 0$,} 
 \end{cases}
\] 
using $\sigma(\theta) \equiv \theta_{\sigma} \pmod{2/3}$. 
The claim on the Artin conductor follows from this. 

The unique index $2$ subgroup of 
$Q \rtimes \mathbb{Z}$ is 
$Q \rtimes 2\mathbb{Z}$, 
because $Q$ has no index $2$ subgroup. 
Hence, if $\tau_{\zeta'}$ is not primitive, 
it is induced from a character of $W_{K_2}$. 
However, this is impossible, because 
the restriction of $\tau_{\zeta'}$ 
to $W_{K_2}$ is irreducible. 
\end{proof}

We define a character 
$\lambda_{\xi} :W_K \to k^{\times}$ by 
$\lambda_{\xi} (\sigma) =\bar{\xi}_{\sigma}$. 
We put 
\begin{align*}
 \tau_{\zeta', \chi} &=
 \tau_{\zeta'} \otimes (\chi \circ \lambda_{\xi}), \\ 
 \theta_{\zeta', \chi} &= 
 (\chi \circ \kappa_1 ) \otimes 
 (\phi \circ \Tr_{k_2/\mathbb{F}_2}
 (\zeta'^{-2}\kappa_2)) , \\ 
 \rho_{\zeta', \chi} &=
 \Ind_{U_D}^{\mathcal{O}_D^{\times}} \theta_{\zeta', \chi}, \\ 
 \Pi_{\zeta', \chi} &= 
 \tau_{\zeta', \chi} \otimes \rho_{\zeta', \chi} 
\end{align*}
for $\zeta' \in k^{\times}$ and 
$\chi \in (k^{\times})^{\vee}$. 
In the sequel, we consider $\tau_{\zeta', \chi}$ 
as a representation of $I_K$. 
\begin{prop}\label{evenact}
We assume that $q$ is even. 
Let $\zeta' \in k^{\times}$. 
Then we have an isomorphism 
\[
 \bigoplus_{\zeta \in k_2^{\times}} 
 H^1 (\overline{\mathbf{X}}^{\mathrm{c}}_{\zeta,\zeta'}, 
 \overline{\mathbb{Q}}_{\ell} ) 
 \simeq \bigoplus_{\chi \in (k^{\times})^{\vee}} 
 \Pi_{\zeta', \chi} 
\]
as representations of $I_K \times \mathcal{O}_D^{\times}$. 
\end{prop}
\begin{proof}
The actions of $I_K$ and $U_D$ on 
$\bigoplus_{\zeta \in k^{\times}} 
 H^1 (\overline{\mathbf{X}}^{\mathrm{c}}_{\zeta,\zeta'}, 
 \overline{\mathbb{Q}}_{\ell})$ 
factor through $Q \times k^{\times}$ 
by Proposition \ref{dXact} and 
Proposition \ref{Wact}. 
On the other hand, 
the action of $Q \times k^{\times}$ on 
$\bigoplus_{\zeta \in k^{\times}} 
 H^1 (\overline{\mathbf{X}}^{\mathrm{c}}_{\zeta,\zeta'}, 
 \overline{\mathbb{Q}}_{\ell})$ 
is induced from the action of 
$Q$ on 
$H^1 (\overline{\mathbf{X}}^{\mathrm{c}}_{1,\zeta'}, 
 \overline{\mathbb{Q}}_{\ell} )$. 
Hence, we have an isomorphism 
\[
 \bigoplus_{\zeta \in k^{\times}} 
 H^1 (\overline{\mathbf{X}}^{\mathrm{c}}_{\zeta,\zeta'}, 
 \overline{\mathbb{Q}}_{\ell} ) 
 \simeq \bigoplus_{\chi \in (k^{\times})^{\vee}} 
 \tau \otimes \chi 
\]
as representations of $Q \times k^{\times}$. 
Therefore, 
we have an isomorphism 
\[
 \bigoplus_{\zeta \in k^{\times}} 
 H^1 (\overline{\mathbf{X}}^{\mathrm{c}}_{\zeta,\zeta'}, 
 \overline{\mathbb{Q}}_{\ell} ) 
 \simeq \bigoplus_{\chi \in (k^{\times})^{\vee}} 
 \tau_{\zeta', \chi} 
 \otimes \theta_{\zeta', \chi}
\]
as representations of $I_K \times U_D$ 
by Proposition \ref{dXact} and 
Proposition \ref{Wact}. 
Inducing this representation 
from $U_D$ to $\mathcal{O}_D^{\times}$, 
we obtain the isomorphism in the assertion. 
\end{proof}

\subsection{Genus calculation}\label{gcal}
\begin{lem}\label{dimH}
We have 
$\dim H^1_{\mathrm{c}} 
 (\mathbf{X}_1(\mathfrak{p}^3)_{\mathbf{C}}, 
 \overline{\mathbb{Q}}_{\ell}) = 2q^3 -2q +1$. 
\end{lem}
\begin{proof}
It suffices to show that 
\[
 \dim H^1_{\mathrm{c}} 
 \bigl( (\mathrm{LT}_1(\mathfrak{p}^3)/\varpi^{\mathbb{Z}})
 _{\mathbf{C}}, 
 \overline{\mathbb{Q}}_{\ell}\bigr) =4q^3 -4q +2, 
\]
because we have 
\[
 \dim H^1_{\mathrm{c}} 
 \bigl( (\mathrm{LT}_1(\mathfrak{p}^3)/\varpi^{\mathbb{Z}})
 _{\mathbf{C}}, 
 \overline{\mathbb{Q}}_{\ell}\bigr) = 
 2\dim H^1_{\mathrm{c}} 
 (\mathbf{X}_1(\mathfrak{p}^3)_{\mathbf{C}}, 
 \overline{\mathbb{Q}}_{\ell}). 
\] 
For an irreducible smooth representation $\pi$ 
of $\mathrm{GL}_2 (K)$, 
we write $c(\pi)$ for the conductor of $\pi$. 
By Proposition \ref{cohGL}, 
we have 
\[
 H^1_{\mathrm{c}} 
 \bigl( (\mathrm{LT}_1(\mathfrak{p}^3)/\varpi^{\mathbb{Z}})
 _{\mathbf{C}}, 
 \overline{\mathbb{Q}}_{\ell}\bigr) \simeq 
 \bigoplus_{\pi} 
 \bigl(\pi^{K_1(\mathfrak{p}^3)}\bigr)
 ^{\oplus 2\dim \mathrm{LJ}(\pi)} \oplus 
 \bigoplus_{\chi} (\mathrm{St} \otimes \chi )^{K_1(\mathfrak{p}^3)}, 
\]
where $\pi$ runs through irreducible cuspidal representations of 
$\mathrm{GL}_2 (K)$ such that $c(\pi) \leq 3$ and 
$\omega_{\pi} (\varpi)=1$, and 
$\chi$ runs through characters of $K^{\times}$ 
such that $c(\mathrm{St} \otimes \chi) \leq 3$ and 
$\chi (\varpi^2)=1$. 
We have the following list of 
discrete series representations $\pi$ of 
$\mathrm{GL}_2 (K)$ such that $c(\pi) \leq 3$ and 
$\omega_{\pi} (\varpi)=1$: 
\begin{enumerate}
\item[(1)]
$\pi \simeq \mathrm{St} \otimes \chi$ 
for an unramified character 
$\chi \colon K^{\times} \to \overline{\mathbb{Q}}_{\ell}^{\times}$ 
such that $\chi (\varpi^2)=1$. 
Then $c(\pi) =1$ and 
$\dim \mathrm{LJ}(\pi)=1$. 
There are two such representations. 
\item[(2)]
$\pi \simeq \mathrm{St} \otimes \chi$ 
for a tamely ramified character 
$\chi \colon K^{\times} \to \overline{\mathbb{Q}}_{\ell}^{\times}$ 
that is not unramified 
and satisfies $\chi (\varpi^2)=1$. 
Then $c(\pi) =2$ and 
$\dim \mathrm{LJ}(\pi)=1$. 
There are $2(q-2)$ such representations. 
\item[(3)] 
$\pi \simeq \pi_{\chi}$, 
in the notation of \cite[19.1]{BHLLC}, 
for a character 
$\chi \colon K_2^{\times} \to \overline{\mathbb{Q}}_{\ell}^{\times}$ 
of level zero 
such that $\chi$ does not factor through 
$\Nr_{K_2/K}$ and $\chi(\varpi) =1$. 
Then $c(\pi) =2$ and 
$\dim \mathrm{LJ}(\pi)=2$. 
There are $q(q-1)/2$ such representations. 
\item[(4)] 
The cuspidal representations $\pi$ of 
$\mathrm{GL}_2 (K)$ such that 
$c(\pi) =3$ and 
$\omega_{\pi} (\varpi)=1$. 
Then $\dim \mathrm{LJ}(\pi)=q+1$ 
by \cite[Theorem 3.6]{TunLLC}. 
There are $2(q-1)^2$ such representations 
by \cite[Theorem 3.9]{TunLLC}. 
\end{enumerate}
We note that 
$\dim \pi^{K_1(\mathfrak{p}^3)} =4-c(\pi)$ if 
$\pi$ is a discrete series representation of 
$\mathrm{GL}_2 (K)$ such that $c(\pi) \leq 3$. 
Then we obtain the claim by 
taking a summation according to the above list. 
\end{proof}

For an affinoid rigid space $X$, 
a Zariski subaffinoid of $X$ is the inverse image of 
a nonempty open subscheme of $\overline{X}$ 
under the reduction map 
$X \to \overline{X}$. 

\begin{prop}\label{posg}
Let $W$ be a wide open rigid curve over 
a finite extension of $\widehat{K}^{\mathrm{ur}}$ 
with a stable covering 
$\{(U_i,U_i^{\mathrm{u}} )\}_{i \in I}$. 
Let $X$ be a subaffinoid space of $W$ such that 
$\overline{X}$ is a connected smooth curve 
with a positive genus. 
Then there exists $i \in I$ such that 
$X$ is a Zariski subaffinoid of 
$U_i^{\mathrm{u}}$. 
\end{prop}
\begin{proof}
Assume that 
$X \cap U_i^{\mathrm{u}}$ is contained in 
a finite union of residue classes of $X$ 
for any $i \in I$. 
Then a Zariski subaffinoid of $X$ 
appears in an open annulus. 
This is a contradiction, because 
$\overline{X}$ has a positive genus. 
Hence there exists $i' \in I$ such that 
$X \cap U_{i'}^{\mathrm{u}}$ is not contained in 
any finite union of residue class of $X$. 
We fix such $i'$ in the sequel. 

Then some open irreducible subscheme of 
the reduction of $X \cap U_{i'}^{\mathrm{u}}$ 
does not go to one point in $\overline{X}$ 
under the natural map 
$\overline{X \cap U_{i'}^{\mathrm{u}}} \to \overline{X}$. 
Let $Y$ be the inverse image of such an 
open subscheme under the reduction map 
$X \cap U_{i'}^{\mathrm{u}} \to 
 \overline{X \cap U_{i'}^{\mathrm{u}}}$. 
Then we see that 
$Y$ is a Zariski subaffinoid of $X$ by 
\cite[Lemma 2.24 (i)]{CMp^3}. 
Each connected component of $X \setminus Y$ 
is an open disk, and included in 
$U_{i'}^{\mathrm{u}}$ or $U_i^{\mathrm{u}}$ 
for $i \neq i'$ or an open annulus outside 
the underlying affinoids. 
This can be checked by applying 
\cite[Corollary 2.39]{CMp^3} to every closed disk in 
a connected component of $X \setminus Y$. 
Hence, 
$X \cap U_{i'}^{\mathrm{u}}$ 
is a Zariski subaffinoid of $X$. 
If $X \cap U_{i'}^{\mathrm{u}} \neq X$, 
then $U_{i'}^{\mathrm{u}}$ is connected to 
an open disk in $U_i^{\mathrm{u}}$ 
for $i \neq i'$ or in an open annulus outside 
the underlying affinoids. 
This is a contradiction. 
Therefore, we have $X \subset U_i^{\mathrm{u}}$. 
Then we obtain the claim by 
\cite[Lemma 2.24 (i)]{CMp^3}. 
\end{proof}

\begin{lem}\label{zerog}
Let $W$ be a wide open rigid curve over 
a finite extension of $\widehat{K}^{\mathrm{ur}}$ 
with a stable covering. 
Let $X$ be a subaffinoid space of $W$ such that 
$\overline{X}$ is a connected smooth curve 
with genus zero. 
Then there is a basic wide open subspace of 
$W$ such that its underlying affinoid is 
$X$. 
\end{lem}
\begin{proof}
We note that 
we have the claim if 
$X$ appears in an open subannulus of $W$. 
Let $\{(U_i,U_i^{\mathrm{u}} )\}_{i \in I}$ be 
the stable covering of $W$. 

First, we consider the case where 
$X \cap U_i^{\mathrm{u}}$ is contained in 
a finite union of residue classes of $X$ 
for any $i \in I$. 
Then a Zariski subaffinoid of $X$ 
appears in an open annulus. 
Further, $X$ itself appears in the open annulus, 
because $X$ is connected. 
Hence, we have the claim in this case. 

Therefore, we may assume that 
there exists $i' \in I$ such that 
$X \cap U_{i'}^{\mathrm{u}}$ is not contained in 
any finite union of residue class of $X$. 
We fix such $i'$. 
By the same argument as in the proof of 
Proposition \ref{posg}, 
we have $X \subset U_i^{\mathrm{u}}$. 
If the image of the induced map 
$\overline{X} \to \overline{U}_i^{\mathrm{u}}$ is 
one point, 
we have the claim because $X$ appears in an open disk. 
Otherwise, 
$X$ is a Zariski subaffinoid of $U_i^{\mathrm{u}}$, 
and we have the claim. 
\end{proof}

We consider the natural level-lowering map 
\[
 \pi_f \colon \mathbf{X}_1(\mathfrak{p}^3) 
 \to \mathbf{X}_1(\mathfrak{p}^2) ; \ 
 (u,X_3) \mapsto (u,X_2). 
\]

\begin{lem}\label{nball}
The connected components of 
$\mathbf{W}_{1,2'}$, $\mathbf{W}_{1,3'}$, 
$\mathbf{W}_{2,1'}$ and 
$\mathbf{W}_{4,1'} \cup 
 \mathbf{W}_{5,1'} \cup \mathbf{W}_{6,1'}$ 
are not open balls. 
\end{lem}
\begin{proof}
Let $\mathbf{W}'_0$ be a subannulus 
of $\mathbf{W}_0$ defined by 
$v(u)<1/(q(q+1))$. 
Then we have 
$\pi_f^{-1}(\mathbf{W}_{k^{\times}}) =\mathbf{W}_{2,1'}$, 
$\pi_f^{-1}(\mathbf{W}_{\infty}) = 
 \mathbf{W}_{4,1'} \cup \mathbf{W}_{5,1'} \cup \mathbf{W}_{6,1'}$ 
and 
$\pi_f^{-1}(\mathbf{W}'_0)=\mathbf{W}_{1,2'} \cup \mathbf{W}_{1,3'}$. 
Hence we have the claim by 
Proposition \ref{ltwo} and 
\cite[Lemma 1.4]{Colstmap}. 
\end{proof}

The smooth projective curves 
$\overline{\mathbf{Y}}_{1,2}^{\mathrm{c}}$ and 
$\overline{\mathbf{Y}}_{2,1}^{\mathrm{c}}$ 
have defining equations 
$X^q Y-XY^q =Z^{q+1}$ determined by the 
equation in 
Proposition \ref{Y12}
and Proposition \ref{Y21}. 
The infinity points of 
$\overline{\mathbf{Y}}_{1,2}$ 
in $\mathbb{P}^2_k$ 
consist of 
$P_a^+ =(a,1,0)$ for $a \in k$ and 
$P_{\infty}^+ =(1,0,0)$. 
The infinity points of 
$\overline{\mathbf{Y}}_{2,1}$ 
consist of 
$P_a^- =(a,1,0)$ for $a \in k$ and 
$P_{\infty}^- =(1,0,0)$. 

For a wide open space $W$, 
let $e(W)$ be the number of the ends of $W$, and 
$g(W)$ be the genus of $W$ 
(cf.~\cite[p.~369 and p.~380]{CMp^3}). 
For a proper smooth curve $C$ over $k^{\mathrm{ac}}$, 
we write $g(C)$ for the genus of $C$. 

\begin{thm}\label{cov3}
The covering $\mathcal{C}_1(\mathfrak{p}^3)$ 
is a semi-stable covering of 
$\mathbf{X}_1(\mathfrak{p}^3)$ over some finite extension. 
\end{thm}
\begin{proof}
We consider the stable covering of 
$\mathbf{X}_1(\mathfrak{p}^3)_{\mathbf{C}}$ 
by Proposition \ref{exssc}. 
Then 
$\overline{\mathbf{Y}}_{1,2}^{\mathrm{c}}$ and 
$\overline{\mathbf{Y}}_{2,1}^{\mathrm{c}}$ 
appear in the stable reduction of 
$\mathbf{X}_1(\mathfrak{p}^3)_{\mathbf{C}}$ 
as irreducible components 
by Proposition \ref{posg}. 
The point $P_0^+$ is the unique 
infinity point of $\overline{\mathbf{Y}}_{1,2}$ 
whose tube is contained in $\mathbf{W}^{+}_{1,1'}$, 
because 
$v(X_3) > 1/(q^3(q^2-1))$ in $\mathbf{W}^{+}_{1,1'}$.
Similarly, 
$P_0^-$ is the unique 
infinity point of $\overline{\mathbf{Y}}_{2,1}$ 
whose tube is contained in $\mathbf{W}^{-}_{1,1'}$. 
Hence, 
we have 
$e(\mathbf{X}_1(\mathfrak{p}^3)_{\mathbf{C}}) 
 \geq 2q$ by Lemma \ref{nball}. 
Therefore, we have 
$g(\mathbf{X}_1(\mathfrak{p}^3)_{\mathbf{C}}) 
 \leq q^3 -2q +1$ by Lemma \ref{dimH}. 
On the other hand, 
we have 
\[
 g(\mathbf{X}_1(\mathfrak{p}^3)_{\mathbf{C}}) 
 \geq g(\overline{\mathbf{Y}}^{\mathrm{c}}_{1,2}) + 
 g(\overline{\mathbf{Y}}^{\mathrm{c}}_{2,1}) + 
 \begin{cases}
 \sum_{\zeta \in \mu_{2(q^2-1)}(k^{\mathrm{ac}})} 
 g(\overline{\mathbf{X}}^{\mathrm{c}}_{\zeta}) 
 & \textrm{if $q$ is odd,}\\ 
 \sum_{\zeta \in k_2^{\times},\, \zeta' \in k^{\times}} 
 g(\overline{\mathbf{X}}^{\mathrm{c}}_{\zeta,\zeta'}) 
 & \textrm{if $q$ is even,}
 \end{cases}
\]
where the summation on the right hand side is 
$q^3 -2q +1$ by Proposition \ref{Yact}, 
Proposition \ref{oddact} and 
Proposition \ref{evenact}. 
Then the affinoids 
$\mathbf{Y}_{1,2}$, $\mathbf{Y}_{2,1}$, 
$\mathbf{X}_{\zeta}$ for 
$\zeta \in \mu_{2(q^2-1)}(k^{\mathrm{ac}})$ and 
$\mathbf{X}_{\zeta,\zeta'}$ for 
$\zeta \in k_2^{\times}$ and $\zeta' \in k^{\times}$ 
are underlying affinoids of basic wide open spaces in 
the stable covering by 
Proposition \ref{posg} and 
Lemma \ref{nball}. 
Therefore, 
by the above genus inequalities, 
we see that 
$e(\mathbf{X}_1(\mathfrak{p}^3)_{\mathbf{C}}) 
 = 2q$ and 
the connected components of 
$\mathbf{W}_{1,2'}$, $\mathbf{W}_{1,3'}$, 
$\mathbf{W}_{2,1'}$ and 
$\mathbf{W}_{4,1'} \cup 
 \mathbf{W}_{5,1'} \cup \mathbf{W}_{6,1'}$ 
are open annuli. 

The connected components of 
$\mathbf{X}_1(\mathfrak{p}^3) \setminus 
 \mathbf{Z}_{1,1}^0$ are 
two wide open spaces, 
because each connected component is connected to 
$\mathbf{Z}_{1,1}^0$ at 
an open subannulus by Lemma \ref{zerog}. 
Then we see that 
these two wide open spaces are basic wide open spaces with 
underlying affinoids 
$\mathbf{Y}_{1,2}$ and $\mathbf{Y}_{2,1}$ 
by the above genus inequalities. 
Therefore we have the claim by 
Proposition \ref{Xodd}, Proposition \ref{Peven} and 
Proposition \ref{Xeven}. 
\end{proof}

\subsection{Structure of cohomology}
In this subsection, we study the action of 
$I_K \times \mathcal{O}_D ^{\times}$ on 
$\ell$-adic cohomology of $\mathbf{X}_1(\mathfrak{p}^3)$. 
We put 
\[
 (W_K \times D^{\times})^0 = 
 \{(\sigma,\varphi^{-r_{\sigma}} ) \in W_K \times D^{\times} \}. 
\]
Although it is possible to study the action of 
$(W_K \times D^{\times})^0$ 
using the result of Section \ref{actine}, 
here we study only the inertia action 
for simplicity. 
The result in this subsection is 
essentially used in \cite{ITreal3}. 

Let $\mathcal{X}_1(\mathfrak{p}^3)$ be the 
semi-stable formal scheme constructed from 
$\mathcal{C}_1 (\mathfrak{p}^3)$ by \cite[Theorem 3.5]{ITcrig}. 
The semi-stable reduction of 
$\mathcal{X}_1(\mathfrak{p}^3)$ means the underlying reduced scheme 
of $\mathcal{X}_1(\mathfrak{p}^3)$, which 
is denoted by 
$\mathcal{X}_1(\mathfrak{p}^3)_{k^{\mathrm{ac}}}$. 

\begin{lem}\label{connect}
The smooth projective curves 
$\overline{\mathbf{Y}}_{1,2}^{\mathrm{c}}$ and 
$\overline{\mathbf{Y}}_{2,1}^{\mathrm{c}}$ 
intersect with 
$\overline{\mathbf{Z}}_{1,1}^{\mathrm{c}}$ 
at $P_0^+$ and $P_0^-$ respectively 
in the stable reduction 
$\mathcal{X}_1(\mathfrak{p}^3)_{k^{\mathrm{ac}}}$. 
\end{lem}
\begin{proof}
We see this from the proof of 
Theorem \ref{cov3}. 
\end{proof}

Let $\Gamma$ be the graph defined by the following:
\begin{itemize}
\item 
The set of the vertices of $\Gamma$ consists of 
$P_0$, $P_{\infty}$, 
$P^+_{a}$ and 
$P^-_{a}$ 
for 
$a \in \mathbb{P}^1(k)\setminus\{0\}$. 
\item 
The set of the edges of $\Gamma$ consists of 
$P_0 P^+_{a}$, $P_0 P^-_{a}$, 
$P_{\infty} P^+_{a}$ and 
$P_{\infty} P^-_{a}$ for 
$a \in \mathbb{P}^1(k)\setminus\{0\}$. 
\end{itemize}
We note that 
$P^+_{a}$ and 
$P^-_{a}$ 
for 
$a \in \mathbb{P}^1(k)\setminus\{0\}$ 
are points of 
$\overline{\mathbf{Y}}_{1,2}^{\mathrm{c}}$ and 
$\overline{\mathbf{Y}}_{2,1}^{\mathrm{c}}$ 
that are not on 
$\overline{\mathbf{Z}}_{1,1}^{\mathrm{c}}$ 
by Lemma \ref{connect}. 
Let $H^1(\Gamma,\overline{\mathbb{Q}}_{\ell} )$ 
be the cohomology group of $\Gamma$ with coefficients in 
$\overline{\mathbb{Q}}_{\ell}$ (cf. \cite[Section 2]{ITcrig}). 
The group $I_K \times \mathcal{O}_D^{\times}$ 
acts on $P^+_{a}$ and $P^-_{a}$ 
for 
$a \in \mathbb{P}^1(k)\setminus\{0\}$ 
via the action on 
$\overline{\mathbf{Y}}_{1,2}^{\mathrm{c}}$ and 
$\overline{\mathbf{Y}}_{2,1}^{\mathrm{c}}$. 
Let $I_K \times \mathcal{O}_D^{\times}$ act on 
$P_0$ and $P_{\infty}$ trivially. 
By this action, we consider 
$H^1(\Gamma,\overline{\mathbb{Q}}_{\ell} )$ as a 
$\overline{\mathbb{Q}}_{\ell}[I_K \times \mathcal{O}_D^{\times}]$-module. 

\begin{thm}\label{cohstr}
We have an exact sequence 
\[
 0 \longrightarrow 
 H^1(\Gamma,\overline{\mathbb{Q}}_{\ell} )
 \longrightarrow 
 H^1_{\mathrm{c}} (\mathbf{X}_1(\mathfrak{p}^3)_{\mathbf{C}}, 
 \overline{\mathbb{Q}}_{\ell}) 
 \longrightarrow 
 H^1 (\mathcal{X}_1(\mathfrak{p}^3)_{k^{\mathrm{ac}}}, 
 \overline{\mathbb{Q}}_{\ell})^* (-1) 
 \longrightarrow 0
\]
as representations of $(W_K \times D^{\times})^0$. 
Further, as $(I_K \times \mathcal{O}_D^{\times})$-representations, 
$H^1 (\mathcal{X}_1(\mathfrak{p}^3)_{k^{\mathrm{ac}}}, 
 \overline{\mathbb{Q}}_{\ell})$ is isomorphic to 
\begin{align*}
 \bigoplus_{\tilde{\chi} \in 
 (k_2^{\times})^{\vee} 
 \backslash (k^{\times})^{\vee}} 
 \Pi_{\tilde{\chi}} \oplus 
 \begin{cases}
 \bigoplus_{\chi \in (k^{\times})^{\vee}} 
 \bigoplus_{\psi \in (k^{\vee} \backslash \{1\})/\!\sim} 
 \Pi_{\chi, \psi} \oplus \Pi'_{\chi, \psi}
 & \textrm{if $q$ is odd,}\\ 
 \bigoplus_{\zeta' \in k^{\times}} 
 \bigoplus_{\chi \in (k^{\times})^{\vee}} 
 \Pi_{\zeta', \chi} 
 & \textrm{if $q$ is even,}
 \end{cases}
\end{align*}
where we put 
$\Pi_{\tilde{\chi}} = (\tilde{\chi} \circ \lambda) 
 \otimes 
 (\tilde{\chi} \circ \kappa_1 \oplus 
 \tilde{\chi}^q \circ \kappa_1 )$, 
and 
$H^1(\Gamma,\overline{\mathbb{Q}}_{\ell} )$ is isomorphic to 
\[
 1 \oplus \bigoplus_{\chi \in (k^{\times})^{\vee}} 
 \bigl( (\chi \circ \lambda^{q+1}) \otimes 
 (\chi \circ \kappa_1^{q+1} ) \bigr)^{\oplus2}.
\]
\end{thm}
\begin{proof}
The existence of 
the exact sequence follows from 
\cite[Theorem 5.3]{ITcrig} and Lemma \ref{connect} 
using Poincar\'{e} duality 
(cf. \cite[Proposition 5.9.2]{FaCohLLC}). 
We know the structure of 
$H^1 (\mathcal{X}_1(\mathfrak{p}^3)_{k^{\mathrm{ac}}}, 
 \overline{\mathbb{Q}}_{\ell})$ 
by Proposition \ref{Yact}, 
Proposition \ref{oddact} and 
Proposition \ref{evenact}. 

We study the structure of 
$H^1(\Gamma,\overline{\mathbb{Q}}_{\ell} )$. 
By Lemma \ref{dYact} and 
Lemma \ref{iYact}, 
the action of $I_K \times \mathcal{O}_D^{\times}$ 
on $H^1(\Gamma,\overline{\mathbb{Q}}_{\ell} )$ 
factors through $k^{\times}$. 
We can check that 
\[
 H^1(\Gamma,\overline{\mathbb{Q}}_{\ell} ) \simeq 
 1 \oplus \bigoplus_{\chi \in (k^{\times})^{\vee}} 
 \chi^{\oplus 2} 
\]
as representations of $k^{\times}$. 
Hence, the claim follows from 
Lemma \ref{dYact} and Lemma \ref{iYact}. 
\end{proof}

\noindent
Naoki Imai\\ 
Graduate School of Mathematical Sciences, 
the University of Tokyo, 3-8-1 Komaba, Meguro-ku, 
Tokyo 153-8914, Japan\\ 
naoki@ms.u-tokyo.ac.jp\\ 

\noindent
Takahiro Tsushima\\ 
Department of Mathematics and Informatics, 
Faculty of Science, Chiba University, 
1-33 Yayoi-cho, Inage, Chiba, 263-8522, Japan\\
tsushima@math.s.chiba-u.ac.jp

\end{document}